\documentclass[12pt, leqno]{amsart}
 \usepackage[latin1]{inputenc}
  \usepackage[T1]{fontenc}
  \usepackage{amsmath,amssymb}
 
\usepackage{indentfirst}
\usepackage{amssymb,amsmath}
\usepackage{amstext}
\usepackage{amsopn}
\usepackage{amsfonts}
\usepackage{amsmath}
\usepackage{latexsym}
\usepackage{amscd}
\usepackage{amssymb}
\usepackage{amsmath}
\usepackage[all,cmtip]{xy}
\usepackage{leftidx}
\usepackage{graphicx}
\usepackage{tikz}
\usepackage{ulem}
\usepackage{hyperref}

=\oddsidemargin \font\teneufm=eufm10 \font\seveneufm=eufm7
\font\fiveeufm=eufm5
\newfam\eufmfam
\textfont\eufmfam=\teneufm \scriptfont\eufmfam=\seveneufm
\scriptscriptfont\eufmfam=\fiveeufm

\def\dar[#1]{\ar@<2pt>[#1]\ar@<-2pt>[#1]}
  \entrymodifiers={!!<0pt,0.7ex>+}

\let\goth\mathfrak
\def\cA{\mathcal A}
\def\cB{\mathcal B}

\def\cG{\mathcal G}

\def\cO{\mathcal O}
\def\cT{\mathcal T}

\def\cE{\mathcal E}
\def\cF{\mathcal F}

\def\gn{\mathfrak n}

\def\gP{\mathfrak{p}}

\def\cSets{\mathcal Sets}

\def\gm{\goth m}

\def\RR{\mathbb{R}}
\def\CC{\mathbb{C}}
\def\GG{\mathbb{G}}

\def\FF{\mathbb{F}}

\def\gG{\goth G}
\def\gH{\goth H}

\def\gp{\goth p}

\def\gq{\goth q}

\def\1{\mbox{\bf 1}}

\def\rad{\mathrm{rad}}

\DeclareMathOperator{\Hom}{Hom}
\DeclareMathOperator{\Aut}{Aut}

\DeclareMathOperator{\Out}{Out}

\DeclareMathOperator{\PGL}{\rm PGL}
\DeclareMathOperator{\GL}{\rm GL}

\newcommand{\incl}[1][r]
{\ar@<-0.2pc>@{^(-}[#1] \ar@<+0.2pc>@{-}[#1]}

\newtheorem{theorem}{Theorem}[subsection]

\newtheorem{claim}[theorem]{Claim}

\newtheorem{lemma}[theorem]{Lemma}

\newtheorem{proposition}[theorem]{Proposition}

\newtheorem{stheorem}{Theorem}[section]

\newtheorem{scorollary}[stheorem]{Corollary}
\newtheorem{slemma}[stheorem]{Lemma}
\newtheorem{sproposition}[stheorem]{Proposition}
\newtheorem{sremark}[stheorem]{Remark}
\newtheorem{sremarks}[stheorem]{Remarks}
\newtheorem{sexample}[stheorem]{Example}

\theoremstyle{definition}
\newtheorem{remark}[theorem]{Remark}

\numberwithin{equation}{section}

\def\ZZ{\mathbb{Z}}
\def\CC{\mathbb{C}}
\def\PP{\mathbb{P}}

\def\gE{\mathfrak{E}}
\def\gF{\mathfrak{F}}
\def\gG{\mathfrak{G}}

\def\gP{\mathfrak{P}}

\def\Par{\mathrm{Par}}

\def\rO{\mathrm{O}}

\def\QQ{\mathbb{Q}}

\def\bG{\text{\rm \bf G}}

\def\cO{\mathcal{O}}

\def\ol{\overline}

\def\Lie{\mathop{\rm Lie}\nolimits}

\def\2int{\mathop{2\int}\nolimits}

\def\rank{\mathop{\rm rank}\nolimits}

\def\Spec{\mathop{\rm Spec}\nolimits}

\def\Lie{\mathop{\rm Lie}\nolimits}

\def\Hom{\mathop{\rm Hom}\nolimits}

\def\Stab{\mathop{\rm Stab}\nolimits}

\def\Gal{\mathop{\rm Gal}\nolimits}

\def\Pic{\mathop{\rm Pic}\nolimits}

\def\Twc{\mathop{\rm Twc}\nolimits}

\def\Aut{\text{\rm{Aut}}}
\def\Out{\text{\rm{Out}}}
\def\sm{\smallskip}

\def\Par{\text{\rm{Par}}}

\def\resp.{\mathop{\rm resp.}\nolimits}
\def\limproj{\mathop{\oalign{lim\cr
\hidewidth$\longleftarrow$\hidewidth\cr}}}
\def\limind{\mathop{\oalign{lim\cr
\hidewidth$\longrightarrow$\hidewidth\cr}}}

\def\Res{\mathop{\rm Res}\nolimits}

\def\lgr{\longrightarrow}

\font\math=cmmi10
\def\varpi{\hbox{\math\char'44}}

\def\simlgr{\buildrel\sim\over\lgr}

\def\pa{\S\kern.15em }

\def\un{\uppercase\expandafter{\romannumeral 1}}
\def\deux{\uppercase\expandafter{\romannumeral 2}}
\def\trois{\uppercase\expandafter{\romannumeral 3}}
\def\quatre{\uppercase\expandafter{\romannumeral 4}}
\def\cinq{\uppercase\expandafter{\romannumeral 5}}
\def\six{\uppercase\expandafter{\romannumeral 6}}

\def\hfl#1#2#3{\smash{\mathop{\hbox to#3{\rightarrowfill}}\limits
^{\scriptstyle#1}_{\scriptstyle#2}}}
\def\gfl#1#2#3{\smash{\mathop{\hbox to#3{\leftarrowfill}}\limits
^{\scriptstyle#1}_{\scriptstyle#2}}}

\title[Loop torsors]{Loop torsors and Abhyankar's lemma}

\author[P. Gille]{Philippe Gille}
\thanks{The author was supported by the project "Group schemes, root systems, and
related representations" founded by the European Union - NextGenerationEU through
Romania's National Recovery and Resilience Plan (PNRR) call no. PNRR-III-C9-2023-
I8, Project CF159/31.07.2023, and coordinated by the Ministry of Research, Innovation and Digitalization (MCID)
of Romania. }
\address{UMR 5208 Institut Camille Jordan - Universit\'e Claude Bernard Lyon 1
43 boulevard du 11 novembre 1918
69622 Villeurbanne cedex - France}  
\email{gille@math.univ-lyon1.fr}
\address{and Institute of Mathematics "Simion Stoilow" of the Romanian Academy,
21 Calea Grivitei Street, 010702 Bucharest, Romania}

\date{\today}

\begin{document}

 \begin{abstract} We define the notion of loop torsors
 under certain group  schemes defined over the localization of a regular
henselian ring $A$ at a strict normal crossing divisor $D$. We provide a Galois cohomological criterion for 
classifying those torsors. We revisit also the 
theory of loop torsors on Laurent polynomial rings.
  
\medskip

\smallskip



\end{abstract}

\maketitle

 {\small \tableofcontents }

\bigskip

\section{Introduction}\label{section_intro}

The first result of quadratic form theory is
Gauss diagonalization,  a regular quadratic form $q$
of dimension $n$
over a field $k$ of characteristic $\not = 2$ is isometric
to a diagonal quadratic form $a_1 x_1^2+ \dots + a_n x_n^2$.
Let $\rO_n$ be the orthogonal group attached to the
diagonal quadratic form $q_0= x_1^2+ \dots + x_n^2$.  
Serre's viewpoint is  that  the Galois cohomological
set $H^1(k,\rO_n)$   classifies isometry classes of rank $n$ 
\cite[\S III.1.2]{Se1}. The diagonal embedding $(\mu_2)^n \subset O_n$, 
 induces a map
$$
\bigl( k^\times/ k^{\times 2} \bigr)^2 \to  H^1(k,\mu_2)^n 
\to H^1(k, O_n)
$$ 
of Galois cohomology sets where we used Kummer theory to 
identify $H^1(k,\mu_2)$ with  $k^\times/ k^{\times 2}$.
It maps an $n$-uple $(a_1,\dots, a_n)$ of invertible 
scalars  to the isometry class  of the quadratic form 
$a_1 x_1^2+ \dots + a_n x_n^2$. Gauss diagonalization 
rephrases by saying that any $O_n$--torsor  admits a reduction to the $k$--subgroup 
$\mu_2^n$. This is a remarkable fact and several other important 
algebraic constructions are based on finite subgroups of 
algebraic groups: construction of cyclic central simple algebras,
Cayley-Dickson doubling process for composition algebras,
Tits first construction of Albert algebras, etc..
This kind of torsors plays also an important role in the theory of 
essential dimension \cite{RY}, it relates to ramification issues \cite{CS}.
In \cite{CGR1}, there is a general result for reducing torsors to 
an uniform finite subgroup, it had been extended to semilocal rings in \cite{CGR2}.
In the lectures \cite{Gi26}, we explained how
the notion comes actually from topological torsors
and discussed Milnor's result \cite{Mr}.

The torsors reducing to a finite \'etale subgroup are important special cases
of loop torsors over a base ring (or a scheme). Those have been defined firstly for a 
connected variety $X$ over the field $k$ (with separable closure 
$k_s$) equipped with a base point $x_0 \in X(k)$.
We denote by $x: \Spec(k_s) \to X_{k_s}$ the point given by $(x_0,id_{k_s})$.

For a smooth  algebraic $k$--group $G$, the loop classes consist in the image
of  the map
$$
H^1(\pi_1(X,x), G(k_s))  \to H^1(\pi_1(X,x), G(X_{k_s}))
\hookrightarrow H^1(X,G) .
$$  
Loop torsors  are then constructed with 
 a quite   special kind of Galois cocycles (called loop cocycles)
In the reference \cite{GP13}, we investigated a theory of loop torsors 
 over the ring of Laurent polynomials $R_n=k[t_1^{\pm 1}, \dots ,
t_n^{\pm 1}]$ over a field $k$ of characteristic zero. 
The main application is  the study of forms of toroidal 
Lie algebras \cite{CGP14} and conjugacy of Cartan subalgebras in extended affine Lie algebras \cite{CNP}.

Given a linear algebraic $k$--group $G$, 
the loop $G$--torsors over $R_n$  arise
then with Galois cocycles  which do not involve any denominator.
This fact is important and permits some extra functoriality.
Using Bruhat-Tits' theory, this permitted to relate
the  study of those torsors  to that of reductive algebraic groups
over the field of iterated Laurent series $F_n=k((t_1)) \dots ((t_n))$. 
 One important result is the acyclicity
fact $H^1_{loop}(R_n,G) \simlgr H^1(F_n,G)$ providing
a nice dictionnary between  loop $R_n$-torsors under $G$ and
$F_n$-torsors under $G$ \cite[Theorem 8.1]{GP13}.

The first part of the paper is to extend  a bunch of 
the above results
 to an  arbitrary base field and also to allow certain useful locally  algebraic groups which are not algebraic (e.g.\ automorphism
 groups of reductive groups). 
  This requires some 
 preliminary work in the legacy
of SGA3 involving ind-quasi-affine schemes (Gabber). 
Furthermore we  have to limit ourselves to a smaller class of loop torsors called
tame torsors for extending the results. We warn the reader that 
{\it tame loop torsors} are called simply {\it loop torsors} in 
the note \cite{Gi24}. Beyond the tame case, loop torsors are of course of interest, see the Remark \ref{rem_wild}, but the main issue in this paper is the tame case.


The next issue  is to start a similar approach
with the  localization  $A_D$ of 
a regular henselian ring $A$ at a  strict normal crossing divisor $D$ and to relate with  algebraic groups defined over a natural field
associated to $A$ and $D$, namely the completion $K_v$ of
the fraction field $K$ with respect to the valuation arising
from the blow-up of $\Spec(A)$ at its maximal ideal.
It can be seen as  a thickening of  the previous setting
and this explains why we have to consider the case 
of Laurent polynomials first.  
A nice example is the local ring $A=\ZZ_p[[x]]$
for a prime  $p$ where $D= \mathrm{div}(x) + \mathrm{div}(p)$;
we have $A_D= \ZZ_p[[x]]\bigl[   \frac{1}{p},\frac{1}{x}\bigr]$, $K=\QQ_p((x))$ and $K_v$ has residue field $\FF_p((y))$ where $y$ is the class of $\frac{x}{p}$.

Another way to relate the two settings (say the Laurent setting and the Abhyankar setting) is the following example.
We can take $A=k[[t_1,\dots, t_n]]$
with divisor $t_1 t_2 \dots t_n=0$.
In this case $A_D=k[[t_1,\dots, t_n]][\frac{1}{t_1}, \dots \frac{1}{t_n}]$ contains the ring $R_n$ and
we have   $K_v \cong k\bigl( \frac{t_1}{t_n} , \dots, \frac{t_1}{t_{n-1}})((t_n))$. 
The two rings $R_n$ and $A_D$ are close in the sense
they share the same tame covers.

Returning to the case of general $A_D$,
the main result  is the injectivity of the base change map $H^1_{\substack{tame \\ loop}}(A_D,\widetilde G) \to  H^1(K_v,G)$ for a smooth $A$--group scheme $\widetilde G$ which is the extension of a 
  twisted contant $S$--group scheme by a
 reductive $A$--group scheme $G$ (Theorem \ref{thm_main}).
Furthermore this base change map controls reducibility and
isotropy issues for the relevant twisted group schemes. 
In \ref{subsec_witt}, we discuss the case of the orthogonal
group and of quadratic forms. As in  the Laurent case 
\cite[Corollary  5.8]{CGP17}, it turns out that 
the tame loop quadratic $A_D$-forms are the diagonalizable ones.

Finally the applications concern local-global principles for
torsors and homogeneous spaces 
\`a  la Harbater-Hartmann-Krashen, see \cite{GiP23, GiP24}
and also the study of torsors over Laurent polynomials rings
in positive characteristic \cite{CGP24}.

\bigskip

\noindent{\bf Acknowledgements.}  
We thank Raman Parimala for
sharing her insight about the presented results. 
We thank Laurent Moret-Bailly for useful discussions on descent theory and Danny Ofek for suggesting Corollary \ref{cor_GP_pure2}.
Finally we thank Vladimir Chernousov  and Arturo Pianzola for useful comments.

\section{Preliminaries} 

Let $S$ be a base scheme. We recall that a {\it geometric point} $\ol{s}$ is a map 
$\ol{s}: \Spec(F) \to S$ where $F$ is an algebraic closed field.
We define the notion of {\it  quasi-geometric point} as a map 
$s: \Spec(E) \to S$ where $E$ is a separably closed field.

\subsection{Grothendieck's fundamental group} \label{subsec_galois}
We assume that $S$ is connected. If $S$ is noetherian and is equipped 
with a geometric point $\ol{s}$, Grothendieck defined the 
fundamental group $\pi_1(S,\ol{s})$  \cite{SGA1};
the main result is that the 
category of finite \'etale covers 
of $S$ is equivalent to that of 
finite $\pi_1(S, \ol{s})$-sets. 
 This has been  extended to an arbitrary connected
scheme in the Stack Project  \cite[Tag 0BQ8]{St}.
In  Fu's book, it was noticed that  one can deal also with a quasi-geometric point $s: \Spec(k) \to S$ in the noetherian case and that  the category of finite \'etale covers
of $S$ is also equivalent to the category
of finite $\pi_1(S,s)$-sets \cite[Theorem \ 3.2.12]{Fu};
 this was extended to an arbitrary connected 
 scheme  \cite[Theorem 2.3.27]{Sa}.
Let us explain why it provides the same theory by 
considering the  geometric point 
$\ol{s}: \Spec(\ol{k}) \to \Spec(k) \xrightarrow{s} S$
where $\ol{k}$ denotes an algebraically closure
of the separably closed field $k$.
According to the proposition 2.3.35 of
the last reference, we have a natural morphism  \break
$id_*:\pi_1(S, \ol{s}) \to \pi_1(S, s)$ which is actually
an isomorphism since it induces an equivalence of 
categories on finite Galois sets.
If $k$ is a field and $k_s$ is a separable closure, 
we can take as base point $s: \Spec(k_s) \to \Spec(k)$.
In this case, there is a canonical isomorphism $\Gal(k_s/k) \simlgr \pi_1(\Spec(k),s)$ (see \cite[Proposition 3.2.14]{Fu}) which avoids then to deal with an algebraic closure of $k$.
This provides some freedom which is used for example in the recent paper \cite{RS}.

If $S$ is assumed furthermore 
quasi-compact and quasi-separated,
then $(S,s)$ admits a universal cover $(S^{sc}, s^{sc})$
in the sense of  \cite{VW} which is   connected and simply connected  ({\it ibid}, Proposition 3.4).

\subsection{Ind-quasi-affine schemes} We use fpqc covers and topology
in the sense of Kleiman, see \cite[\S 2.3.2]{V} or alternatively \cite[Tag 03NW]{St}. 
The fpqc topology is the default topology for dealing with sheaf torsors under a group scheme. 

A scheme $X$ is ind-quasi-affine if every quasi-compact open of $X$ is quasi-affine; a  morphism of schemes $f: X \to S$ is ind-quasi-affine if $f^{-1}(V)$ is ind-quasi-affine for each affine open $V$ in $Y$ \cite[0AP5]{St}. 
This notion is stable by base change and local for the fpqc topology \cite[Tags 0AP7, 0AP8]{St}.

\begin{sexample}\label{ex_field} 
{\it Field case.}{\rm \, We assume that $S=\Spec(k)$ for a field $k$.
We remind  the reader that the assignement $X \to X(k_s)$ provides
an equivalence of categories between the category of \'etale $k$-schemes and the category of Galois sets \cite[I.4.6.4]{DG}.
Furthermore the correspondence 
exchanges monomorphisms with injections so that monomorphisms in the category of \'etale $k$-schemes are clopen immersions.
The correspondence 
exchanges epimorphisms with surjections so that 
so that epimorphisms in the category of \'etale $k$-schemes are the surjective \'etale morphisms.

The above correspondence induces furthermore
an equivalence of categories between Galois modules
and commutative \'etale $k$--groups.
For example we can deal with the $k$--group
$\mu_{l^\infty}$ of $l^\infty$--roots of unity 
for any prime $l$ invertible in $k$.


}
\end{sexample}

\subsection{Constant  schemes}\label{subsec_constant}

 A constant $S$-scheme is an $S$-scheme isomorphic
 to \break $M_S= \coprod\limits_{m\in M} S$ for a set $M$ \cite[I, \S 8]{SGA3}. 
 For each $S$--scheme $T$ we  have 
 $$
M_S(T)
= \Bigl\{ \hbox{locally constant functions} \quad T_{top} \to M \Bigr\}
$$
denoted by $C^0(T,M)$. On the other hand we have
$$
\Hom_{S-sch}(M_S,T ) = \Hom_{\cSets}\bigl( M, \Hom_S(S,T) \bigr).
$$
The functor $M  \mapsto M_S$ is faithful if $S$ is not empty
and is fully faithfull if $S$ is connected.
 Given two sets $M$ and $N$,  we have then by taking $T=N_S$
 $$
 \Hom_{S-sch}(M_S,N_S ) = \Hom_{\cSets}\bigl( M, C^0(S,N) \bigr)
 $$
 whence a morphism 
$C^0\bigl(S, \Hom_{Sets}(M,N) \bigr) \to \Hom_{S-sch}(M_S,N_S )$
which maps a function $f$ to $m \mapsto f_m$ with 
$f_m(s)= f(s)(m)$.
Furthermore we have 
$$
\Bigl(\Hom_{Sets}(M,N)_S\Bigr)(T) =\Hom_{S-sch}(M_S, N_S)(T)
$$
if $T$ is locally connected or if $M$ is finite.
We obtained then a monomorphism of $S$--functors
$$
\Hom_{Sets}(M,N)_S \to \underline{\Hom}_{sch}(M_S, N_S)
$$
which is an isomorphism when $M$ is finite.

\begin{slemma}\label{lem_constant}
 An $S$-morphism $f: M_S \to N_S$ is \'etale.
\end{slemma}

\begin{proof} We are allowed to localize 
on the source  $N_S$ \cite[Tag 02GJ, (3)]{St}
so that the statement reduces to the case of 
$M_S=S$ (since each $S_m \to M_S$ is an $S$--morphism). An $S$-morphism $f=S \to N_S$
 is  given by a locally constant function $h: S \to N$
 whose image is denoted by $I$.
 It provides a partition $S= \coprod\limits_{i \in I} S_i$ in clopen subschemes
 such that  $f$ reads $\coprod\limits_{m \in I} f_i$ where 
 $f_i: S_i \to N_S$ is the composite $S_i \to S \cong S  \xrightarrow{i \enskip piece} N_S$.
 Again we can reduce to each $S_i$ and $f_i$ is then an open immersion
 and  a fortiori \'etale.
\end{proof}

\subsection{Twisted constant  schemes}\label{subsec_descent}
A  {\it twisted constant} $S$-scheme $X$ is an $S$-scheme 
which is locally isomorphic to a constant scheme with respect to the  fpqc topology, that is, there exists a
a fpqc cover $(S_i)_{i\in I}$ 
such that each $X\times_S S_i$ is a constant $S_i$--scheme
$M_{i,S_i}$ \cite[X, Definition 5.1]{SGA3}.
Such a cover is called a splitting cover.
The  twisted constant $S$-scheme $X$ is 
{\it isotrivial}
 (resp.\ {\it quasi isotrivial})  if 
 we can take the $S_i$'s finite \'etale (resp.\ \'etale) over $S$ in the above definition.

We denote by $\Twc_S$ the full subcategory of $Sch_S$ 
whose objects are  twisted 
constant $S$--schemes. We have the analogous notion for $S$--group schemes
(which are actually those of the SGA3 reference).

\begin{slemma}\label{lem_ind}
Let $X$ be a twisted constant $S$-scheme.

\sm

\noindent (1)  $X$ is ind-quasi-affine over $S$.

\sm

\noindent (2) The morphism $X \to S$ is a separated \'etale morphism and satisfies the
valuative criterion of properness. 

\end{slemma}

\begin{proof} Both statements are local 
for the fpqc topology according to \cite[Tag 0AP8]{St},
\cite[Proposition $_2$.2.7.1.(i), Proposition $_4$.17.7.4]{EGA4}
and a result of M.~Lara on the valuative criterion of properness \cite[lemma 2.40]{L}.
 We can assume that $X$ is constant by fpqc localization, that is, $X= M_S$ for a set $M$;   furthermore 
 we can assume that  $S=\Spec(A)$ is affine.
 \sm
 
 \noindent (1)  We are given  a quasi-compact open
subset $U \subset X=M_S$ and want to show
that it is quasi-affine.
We have $U= \coprod\limits_{m \in M} U_m$ where
$U_m$ is a quasi-compact open subset of $\Spec(A)$. 
Since $U$ is quasi-compact, 
there exists a finite subset $M_0 \subset   M$
such that  $U= \coprod\limits_{m \in M_0} U_i$.
Since each $U_i$ is quasi-affine, it follows that 
$U$ is quasi-affine.

\smallskip
 \noindent (2) It is clear that $M_S$ is separated \'etale over $S$.
 For each $S$--valuation ring $A$ of fraction field $K$,
 we have $M_S(A)=M=M_S(K)$ so that $M_S$ satisfies
 the valuative criterion of properness. 
\end{proof}

\begin{slemma} \label{lem_mono}
(1) The morphisms in the category  $\Twc_S$ are \'etale.
\sm

\noindent  (2) The monomorphisms in the category  $\Twc_S$
are the clopen immersions.

\sm

\noindent (3)  The epimorphisms in the category  $\Twc_S$
are the surjective \'etale morphisms.
\end{slemma}

Of course the analogous statements hold in the category of  group schemes.

\begin{proof} All statements are local with respect to the fpqc topology so 
boil down to the constant case already handled in \S \ref{subsec_constant}.
\end{proof}

\begin{slemma} \label{lem_mono_iso}
Let $f: X \to Y$ be an $S$-morphism between twisted constant
\break  $S$-schemes.

\sm

\noindent (1) If $f$ is a clopen immersion and $Y$ is constant then 
there exists a partition $S= \coprod\limits_{i \in I} \,  S_i$ such that
each $X_{S_i}$ is constant.

\sm

\noindent (2) If $f$ is a clopen immersion and $Y$ is quasi-isotrivial
 then $X$ is  quasi-isotrivial.

 

\end{slemma}

\begin{proof} 
\noindent (1) We assume that $Y=N_S$ for a set $S$.
Lemma \ref{lem_mono}.(1) shows that $X \to N_S=Y$ is a clopen immersion
so that $N_S= X \coprod X'$ where $X= \coprod\limits_{n \in N} X_n$,  $X'= \coprod\limits_{n \in N} X'_n$
with $S= X_n \coprod X'_n$ for each $n \in N$.
It provides a partition $S=\coprod\limits_{i \in I} S_i$
such that $X_{S_i}=\coprod\limits_{n \in N_i} S$ for $N_i \subset N$.

\smallskip

\noindent (2) We apply the previous reasoning after base change by a splitting
\'etale morphism $S' \to S$ of $Y$.
\end{proof}

Let $(S_i)_{i \in I} \to S$ be a fpqc cover.
According to Gabber \cite[Tag 0APK]{St},  any $(S_i)_{i \in I}/S$-data descent  of ind-quasi-affine schemes $(X_i \to S_i)$ is effective.
Let us state a few applications of that.

\begin{slemma} \label{lem_twc_descent} The category $\Twc_S$ satisfies fpqc descent.
\end{slemma}

\begin{proof} Let  $(S_i)_{i \in I} \to S$ be a fpqc cover
and  and consider a data descent twisted constant schemes $(X_i \to S_i)$.
Since each $X_i$ is ind-quasi-affine over $S_i$ (Lemma \ref{lem_ind}.(1)),
Gabber's result provides an $S$--scheme $X$ together with isomorphisms
$X \times_S S_i \cong X_i$. It follows that $X$ is twisted constant over $S$
so we are done. 
\end{proof}

Gabber's result yields also the following generalization
of the case of affine group schemes.

\begin{slemma} \label{lem_descent0} Let $G$ be an ind-quasi-affine $S$--group scheme. 
 
 \sm
 
 \noindent (1) Sheaf $G$--torsors are representable by ind-quasi-affine $S$--schemes.

 \sm
 
 \noindent (2) If $\cE$ is a sheaf $G$--torsor, then the inner twist
 $^{\cE}G$ is  representable by an ind-quasi-affine 
 $S$--group  scheme. 
 
\end{slemma}

\begin{slemma} \label{lem_descent} 
Let $G$ be a twisted constant $S$-group scheme.

\sm

\noindent (1)  The $G$--torsors (for the fpqc topology) are
representable by twisted constant $S$-schemes.
Furthermore they are quasi-isotrivial.

\sm

\noindent (2)
 If $\cE$ is a sheaf $G$--torsor, then the inner twist
 $^{\cE}G$ is  representable by  a twisted 
 constant $S$-scheme.
\end{slemma}

\begin{proof}
(1)  Let $\cE$ be a sheaf $J$--torsor, it is representable by an $S$--scheme $E$
in view of  Lemma \ref{lem_descent0}.(1)
Furthermore $E$ is a twisted constant $S$--scheme
according to  Lemma \ref{lem_ind}.
 Since $E$ is \'etale and $E \times_S J \simlgr E \times_S E$,
the $G$-torsor $E$ is quasi-isotrivial.

\sm

\noindent (2) Lemma \ref{lem_descent0}.(2) insures
representability of  $^{\cE}G$ and this  
$S$-scheme is twisted constant according to \ref{lem_twc_descent}.

\end{proof}

We discuss  variations of \cite[X, Proposition 7.0.3]{SGA3}.

\begin{sproposition} \label{prop_BS}
Assume that $S$ is locally noetherian,
connected and  normal. Let $K$ be 
the fraction field of $X$, let $K_s$ be the separable closure of $K$ and
let $\eta_s: \Spec(K_s) \to S$ be the associated point.
Let  $\pi_1(S,\eta_s)$ be the Grothendieck
fundamental group of $S$.
 
 \sm
\noindent (1) 
 An object $X$ of $\Twc_S$ is
 isomorphic to a disjoint union $\coprod\limits_{i \in I} X_i$
where each $X_i$ is a finite \'etale connected cover of $S$.
Furthermore  we have $X(S)=X(K)$.

 \sm
 
\noindent (2)  Assume that $S$ is noetherian. Then 
the  category of continuous $\pi_1(S, \eta_s)$-sets
 and the category  $\Twc_S$ are equivalent.
 
\end{sproposition}

\begin{proof}
 (1) We are given an object $X$ of $\Twc_S$.
Since it is \'etale
 the irreducible components $(X_i)_{i \in I}$  are its connected components 
and are open according to \cite[X, before Corollary 5.14]{SGA3}. 
Furthermore denoting by  $\eta_i$  the generic point of $X_i$, it is above 
the generic point $\eta$ of $S$ 
and $X_{i,K}=\{ \eta_i \}$ so that $K_i=\kappa(X_i)$ is a finite separable extension of $K$.
In view of Lemma \ref{lem_ind}.(2), $X$ satisfies the valuative criterion of properness,
and do the closed subschemes $X_i$'s  of $X$.
Summarizing each $X_i$ is  \'etale separated  over $S$
and satisfies the valuative criterion of properness.
Let $\widetilde X_i$ be the normalization of $S$ with respect to the finite
field extension $K_i/K$ \cite[12.42]{GW}, it is finite over $X_i$ ({\it ibid}, 12.50). 
Since $X_i$ is normal \cite[I.3.17.(b)]{M}, we obtain then a birational morphism $p_i : X_i \to \widetilde X_i$ by means of the universal property of the normalization ({\it ibid}, 12.44.(iii)). We observe that $p_i$ is separated in view of \cite[9.13.(2)]{GW}.
Since $X_i$ (and $\widetilde X_i$) satisfies  the valuative criterion of properness
we have 
\[
\xymatrix{
\Hom_S(\widetilde X_i, X_i)&= &X_i( \widetilde R_i) \ar[d]&= &
X_i( K_i ) \ar[d]^{=}\\
\Hom_S(\widetilde X_i, \widetilde X_i)&= & \widetilde X_i(  \widetilde R_i) &= &
\widetilde X_i( K_i ) .
}
\]
This provides a section $s_i: \widetilde X_i \to X_i$ of $p_i$. Since $p_i\circ s_i$ is generically
$id_{X_i}$, we have $p_i\circ s_i =id_{X_i}$
and similarly we have $s_i \circ p_i  =id_{\widetilde X_i}$.
Each $p_i$ is an isomorphism and we have then a decomposition $X= \coprod\limits_{i \in I} X_i$
where the $S$-schemes $X_i$'s are finite  \'etale and connected.

 Next we want to show that $X(S) =X(K)$.
Since $S$ is dense in $\Spec(K)$ and $X$ is separated, the map $X(S) \to X(K)$ is injective.
Let $K_i$ be the function  field of $S_i$.
Since $\Spec(K_i)= S_{i,K}$, 
we have $S(K)=  \coprod\limits_{i \in I} X_i(K)= \coprod\limits_{i \in I} \Hom_{K-alg}(K_i,K)$.
It follows that $I_0 \cong S(K)$ where
$I_0$ stands for the indices $i$ such that 
$K=K_i$ or equivalently $S=S_i$.
Thus $X(S) \to X(K)$ is onto and bijective.

\sm

\noindent (2) 
The assumptions imply that $S$ is qcqs \cite[Tags 01OV, 01OY]{St};
 let $S^{sc}$ be 
the universal cover  of $S$ as defined in  \cite[Proposition 3.4]{VW}.
Since a constant $S$--scheme $M_S$ is ind-quasi-affine
(Lemma \ref{lem_ind}.(1)), we can twist it by 
a continuous morphism $f: \pi_1(S,\eta_s) \to \Aut(M)$.
This defines a functor  from the category of 
continuous $\pi_1(S, \eta)$-sets to the category  $\Twc_S$,
$(M,f) \mapsto M^f$.
Extended to $S^{sc}$, this functor is nothing
but $M \mapsto   M_{S^{sc}}$ which is
an equivalence of categories between
the category of sets and  that of constant $S^{sc}$-schemes.
By considering the action of $\Pi_1(S,\eta_s)$ on both sides
and taking invariants, it follows that the 
functor $(M,f) \mapsto M^f$ is fully faithful. 
The essential surjectivity follows from (1).
\end{proof}

\begin{sremark}{\rm  In the proof of (1),
if $X$ is quasi-isotrivial,  each  $X_i$ is finite in view of \cite[X, Lemme 5.13]{SGA3}.
So the new case is when $X$ is not quasi-isotrivial.
Another way to prove the statement is to use Bhatt-Scholze's theory \cite[Theorem 1.10]{BS}.
} 
\end{sremark}

\begin{scorollary}\label{cor_BS2} 
Assume that $S$ is connected, normal and  noetherian.
Let $X=\coprod\limits_{i \in I} S_i$
be a twisted constant $S$-scheme 
where each $S_i$ is a finite \'etale cover of $S$.
Then the following are equivalent:

\sm

(i) the $S_i$'s admit a common splitting finite Galois cover;

\sm

(ii) $X$  is isotrivial;

\sm

(iii) $X$  is quasi-isotrivial.
\end{scorollary}

\begin{proof}
Again we can deal with   
the universal cover  $S^{sc}$ of $S$ as defined in  \cite[Proposition 3.4]{VW}.
The implications  $(i) \Longrightarrow (ii)
\Longrightarrow (iii)$ are obvious.
Let us etablish $(iii) \Longrightarrow (i)$.
We assume that $X$ is quasi-isotrivial. Next we consider the continuous 
 surjective map $q: \Gal(K_s/K) \to \pi_1(S,\eta_s)$
and denote by $H_i$ the preimage of $G_i$ for each $i \in I$.
We denote by $G= \bigcap_{i \in I} G_i$
and by $H=q^{-1}(G)= \bigcap_{i \in I} H_i$.
We have $X_K \cong \coprod\limits_{i \in I} \Spec(K_s^{H_i})$.
Since $X_K$ is quasi-isotrivial, Example \ref{ex_field}
provides a normal open subgroup  $\widetilde H$ 
of  $\Gal(K_s/K)$
such that $K_s^{\widetilde H}$ is the minimal splitting field 
of $X_K$. More precisely  $\widetilde H$ is the 
 largest normal open  subgroup of $H$ such that 
 $\widetilde H \subset H_i$ for all $i$.
  We put $\widetilde G= q(\widetilde H)$,
it is an open normal subgroup of $\pi_1(S, \eta_s)$
such that $\widetilde G \subset G_i$ for each $i$.
We put $\widetilde S= S^{sc}/G$, this is a finite Galois cover
of $S$ which splits each $S_i$.  
\end{proof}

\subsection{Extensions of twisted constant group schemes} \label{subsec_isotrivial}

If $G \to S$ is an ind-quasi-affine,
fpqc descent \cite[Tag 0APK]{St} implies that 
sheaf fpqc $G$-torsors are representable 
by ind-quasi-affine $S$-schemes.
Similarly if $E$ is a fpqc sheaf $\Aut(G)$--torsor,
then the inner twist ${^EG}$ of $G$ by $E$ is 
representable by an ind-quasi-affine $S$-group scheme.  
In the spirit of \cite[VI$_B$, Proposition 9.2]{SGA3}, we have the following fact.

\begin{slemma}\label{lem_9.2} 
Let  $u : G' \to  G$ be a monomorphism of 
$S$--group schemes and assume that the fpqc quotient
$G/G'$ is representable by an $S$--scheme $X$. 
Then $G'$ is ind-quasi-affine over $S$ if and only if 
the quotient morphism $q:G \to X$ is ind-quasi-affine.  
\end{slemma}

\begin{proof} If $q$ is ind-quasi-affine so is $G' \to S$ by 
base change. We assume that $G' \to S$ is quasi-affine. 
Let $(X_i)_{i \in I}$ be an fppc cover of $X$ such that 
$q^{-1}(X_i) \cong X_i \times_S G'$.
Since $q^{-1}(X_i)$ is ind-quasi-affine over $X_i$ for each $i$,
fpqc descent enables us to conclude that $q$ is ind-quasi-affine
\cite[Tag 0AP8]{St}.
\end{proof}

We deal with an $S$-group scheme $\widetilde  G$ fitting 
in an exact sequence of the shape

\begin{equation} \label{eq_shape}
1 \to G \to \widetilde G \to J \to 1
\end{equation}
where $G$  is affine and $J$ is twisted $S$--constant.
Combining Lemmas \ref{lem_ind} and \ref{lem_9.2}
yields that $\widetilde G$ is  ind-quasi-affine  so that Lemma \ref{lem_descent0} applies to  $\widetilde G$.

\begin{slemma}\label{lem_isotrivial} 
We assume that $S$ is connected, normal, locally noetherian.
We denote by $K$ the function field of $S$.

\sm

\noindent (1) The map $H^1(S,J) \to H^1(K,J)$ is injective.

\sm

\noindent (2) $J$-torsors are isotrivial.

\sm

\noindent (3)  If $S$ is regular and $U \subset S$
is a dense open subset which is the complement
of a closed subscheme $Z$ of codimension $\geq 2$,
then the map $H^1(S,J) \to H^1(U,J)$ is bijective.
\end{slemma}

\begin{proof}
(1) The standard torsion argument reduces to establish the triviality of
the kernel of $H^1(S,J) \to H^1(K,J)$.
We have seen that $J$-torsors
are representable by twisted constant $S$--schemes.
Let $[E]$ be an element of the kernel of $H^1(S,J) \to H^1(K,J)$.
According to Proposition \ref{prop_BS}, we have $E(S) =E(K)$.
It follows that $E(S) \not = \emptyset$ so that $[E]=1$.

\sm 

\noindent (2) Let $E$ be a $J$--torsor. Corollary \ref{prop_BS}
states that $E \cong \coprod\limits_{i \in I} S_i$ where the 
$S_i$'s are connected finite \'etale covers of $S$.
The set $I$ is non empty and we pick $i_0 \in I$.
In particular we have $E(S_{i_0}) \not = \emptyset$
so that $E$ is isotrivial. 

\sm

\noindent (3) The injectivity follows from (1).
For establishing the surjectivity we are given
 a $J$-torsor $E$ over $U$. From (2), we know that is split
 by a Galois cover $U' \to U$ say of group $\Gamma$.
 According to Zariski-Nagata purity's theorem \cite[X, Corollaire 3.3]{SGA1}, $U' \to U$  extends uniquely to 
 a Galois cover  $S' \to S$ of group $\Gamma$.
 Then $S'$ is  connected, normal \cite[Tag 0BQL]{St}
 and  locally noetherian \cite[Proposition 6.2.2]{EGA1}.
 We denote by $K'$ the function  field of $S'$.
 According to Proposition \ref{prop_BS}, we have
 $J(S') =J(U') = J(K')$. It follows that 
 the map $H^1(\Gamma, J(S')) \to H^1(\Gamma, J(U'))$ is bijective so that the map 
 $$
 \ker\Bigl( H^1(S,J) \to H^1(S',J) \Bigr)
 \to 
 \ker\Bigl( H^1(U,J) \to H^1(U',J) \Bigr)
 $$
 is bijective. Thus the $J$--torsor $E$ over $U$
 extends to a $J$-torsor over $S$. 
\end{proof}

\begin{sremark} {\rm With Grothendieck's method of \cite[X, \S 5]{SGA3},
one can prove that statement provided the twisted 
$S$--group scheme $J$ is quasi-isotrivial.
}
\end{sremark} 

\begin{sexample}{\rm Let us illustrate the statement
for a field $k$ with the $k$-group $\mu_{l^\infty}$
of $l^\infty$ root of unity attached to an invertible prime
  $l$ (considered in Example \ref{ex_field}).
Given $a \in k^\times$,
we consider the Galois set $\limind_n \{ x_n \in k_s^\times
\, \mid \, x_n^{l^{n}}=  a^{l^{n+1}} \}$
where the transition maps are $x_n \mapsto (x_n)^l$.
It defines a twisted constant $k$--scheme $X$
which is  a $\mu_{l^\infty}$-torsor.
Putting $b = \sqrt[l]{a} \in k_s$,
the elements $b_n= b^{l^{n}}$ defines a point of 
$X(k_s)$ so that this $\mu_{l^\infty}$-torsor  is 
isotrivial.
}
\end{sexample}

\begin{slemma}\label{lem_isotrivial2} 
Assume that $S$ is connected, locally noetherian and normal. 
 In the sequence  \eqref{eq_shape},
 assume that $G$ is reductive.
Then the $\widetilde G$-torsors over $S$ are semi-locally isotrivial
(and a fortiori quasi-isotrivial).
\end{slemma}

\begin{proof} It is known for $J$ by Lemma \ref{lem_isotrivial}.(2)
and for $G$ by \cite[XXIV, Corollaire 4.1.6]{SGA3}.
The d\'evissage from these two cases is similar with the argument of 
the proof of \cite[XXIV, Corollaire 4.2.4]{SGA3}.
\end{proof}

From now on we assume that  $G$ is reductive.
The $S$--group $\widetilde G$
acts on its normal $S$--subgroup $G$ and we consider the commutative diagram
\[
\xymatrix{
1 \ar[r] & \ar[r] G \ar[d]^{Int} & \widetilde G \ar[r] \ar[d]^{Int} & J \ar[d]^{h} \ar[r] & 1 \\
1 \ar[r] & G_{ad}=G/C(G) \ar[r] &  \Aut(G) \ar[r] & \Out(G), \ar[r] & 1.
}
\]
where the bottom exact sequence is \cite[XXIV, Th\'eor\`eme 1.1]{SGA3}.
Note that $\Out(G)$ is a twisted constant $S$--group scheme
(but not necessarily quasi-isotrivial, this is the case however if
$G$ is quasi-isotrivial).

\subsection{Normalizers, I}
Let  $P$ be an $S$-parabolic subgroup of $G$  equipped 
with a Levi $S$-subgroup $L$.
 The normalizer $N_{\widetilde G}(P,L)$ is representable
by a separated smooth  $S$--group scheme 
which is $S$-closed in $\widetilde G$ \cite[Lemme 3.4.54]{Gi15}.
We claim that the sequence  $1 \to G \to \widetilde G \to J \to 1$
induces an exact sequence of $S$--group schemes
\begin{equation}\label{eq_JPL}
1 \to L \to N_{\widetilde G}(P,L) \to J_{P,L} \to 1
\end{equation}
where $J_{P,L}$ is an $S$--subgroup scheme of $J$.
Since $L=N_G(P,L)$ is smooth (and a fortiori flat over $S$),
we know that $N_{\widetilde G}(P,L)/L$ is representable by
an $A$--group scheme $J_{P,L}$ which is locally of finite 
presentation \cite[XVI, Corollaire 2.3]{SGA3}.
Furthermore the homomorphism $J_{P,L} \to J$
is  a monomorphism. 
According to \cite[VI$_B$, Proposition 9.2.(xii)]{SGA3},
$J_{P,L}$ is $S$--smooth, so is $S$-\'etale since 
$J$ is \'etale. We shall say more on 
$J_{P,L}$ in Lemma \ref{lem_JPL}.
 To pursue we deal with the following special case.

\newpage

\begin{slemma}\label{lem_clopen} (1) The $S$--functor $\Aut(G,P,L)$ is representable by a smooth 
$S$--group scheme which fits in  an exact sequence of smooth
$S$-group schemes
$$
1 \to L/C(G) \to \Aut(G,P,L) \to \Out(G,P,L) \to 1.
$$

\noindent (2) The $S$--group scheme $\Out(G,P,L)$ is twisted constant and
 is a clopen $S$-subgroup of $\Out(G)$.
 
 \smallskip
 
\noindent (3)  The fppf quotient  $\Out(G)/\Out(G,P,L)$ is 
representable by a finite \'etale $S$--scheme.
 
\end{slemma}

\begin{proof}(1)  This is the special case of the above 
fact when taking $\Aut(G,P,L)$ for $\widetilde G$.

\sm

\noindent (2) The statement is local for the fpqc topology
so that we can assume that $G$ is split and that $P$ is a standard parabolic 
subgroup. If $G$ is adjoint, then   $\Out(G,P,L)$ and $\Out(G)$ are finite constant \cite[lemme 5.1.2]{Gi15} so that the statement is obvious.
We consider the exact sequence $1 \to C(G)\to G \to G_{ad} \to 1$
and the natural map $\Aut(G) \to \Aut(G_{ad})$.
	The correspondence \cite[Lemma 3.2.1.(2)]{Gi15}
	shows that  $\Aut(G,P,L)=\Aut(G) \times_{\Aut(G_{ad})} \Aut(G_{ad},P_{ad}, L_{ad})$.
We obtain a commutative diagram 
	\[
\xymatrix{
\Aut(G,P,L) \ar[r]^{\sim \qquad \qquad}  \ar[d] & \Aut(G) \times_{\Aut(G_{ad})} \Aut(G_{ad},P_{ad}, L_{ad}) \ar[d]  \\
\Out(G,P,L) \ar[r]^{r \qquad \qquad \quad } & \Out(G) \times_{\Out(G_{ad})} \Out(G_{ad},P_{ad}, L_{ad}) 
}
\] 
\begin{claim} \label{claim_out} The bottom horizontal map $r$ is an isomorphism.
\end{claim}

Since $\Out(G,P,L) \to \Out(G)$ is a monomorphism so is
$r$.  It is then enough to prove that the right vertical map
is an epimorphism of flat sheaves. 
Let $u \in \Bigl( \Out(G) \times_{\Out(G_{ad})} \Out(G_{ad},P_{ad}, L_{ad}) \Bigr)(T)$ for an $S$-scheme $T$.
Up to localize for the flat topology, 
$u$ is represented by  elements 
$a \in \Aut(G)(T)$ and $b \in \Aut(G_{ad},P_{ad},L_{ad}) (T)$
 having same image in  $\Out(G_{ad})(T)$. 
 It means that there exists $y \in G_{ad}(T)$
such that $a \, \mathrm{Int}(y) = b \in \Aut(G_{ad})$.
The pair $(a \, \mathrm{Int}(y) , b)$ defines an element
of \break
${\Bigl( \Aut(G) \times_{\Aut(G_{ad})} \Aut(G_{ad},P_{ad}, L_{ad})\Bigr)(T)}$
mapping to $u$.
The claim is then  established.

It follows that $\Out(G,P,L) \to \Out(G)$ is a clopen immersion.
Also since the category of twisted $S$-group schemes is stable
by cartesian product, we obtain that $\Out(G,P,L)$ is a twisted
constant $S$--group scheme. 

\sm

\noindent (3) We can continue with the same reductions.
We have seen that 
$\Out(G_{ad})/\Out(G_{ad},P_{ad},L_{ad})$ is finite $S$-\'etale.
According to \S \ref{subsec_descent}, the fppf
quotient $\Out(G)/\Out(G,P,L)$ is representable by
a twisted constant $S$-scheme and so is $\Out(G_{ad})/\Out(G_{ad},P_{ad},L_{ad})$. The map 
$\Out(G)/\Out(G,P,L) \to \Out(G_{ad})/\Out(G_{ad},P_{ad},L_{ad})$
is a monomorphism so is a clopen immersion according to
Lemma \ref{lem_mono}.(1). Thus $\Out(G)/\Out(G,P,L)$ is 
 finite $S$-\'etale.
\end{proof}

\begin{slemma} \label{lem_JPL}
(1) The map $J_{P,L} \to J$ is a clopen immersion 
and  $J_{P,L}$ is a twisted  constant  
$S$--group scheme. If $J$ is furthermore quasi-isotrivial, 
so is $J_{P,L}$.

\smallskip

\noindent (2) The fppf quotient $J/J_{P,L}$ is representable
by a finite \'etale $S$--scheme.

\end{slemma}

\begin{proof} 
(1) The $S$--functor $\Aut(G,P,L)$ is representable by a smooth $S$-scheme \cite[Proposition 3.4.3]{Gi15}
and we have $N_{\widetilde G}(P,L) \simlgr \widetilde G \times_{\Aut(G)} \Aut(G,P,L)$.
We obtain then the commutative diagram
\[
\xymatrix{
N_{\widetilde G}(P,L) \ar[r]^{\sim \qquad \qquad}  \ar[d] & \widetilde G \times_{\Aut(G)} \Aut(G,P,L) \ar[d]  \\
J_{P,L}   \ar[r]^{r \qquad \qquad} &  J \times_{\Out(G)} \Out(G,P,L).
}
\]
\begin{claim} \label{claim_JPL} The bottom map $r$ is  an isomorphism.
\end{claim}
Since $J_{P,L} \to J$ is a monomorphism so is
$r$.  It is then enough to prove that the right vertical map
is an epimorphism of flat sheaves. For that we deal 
with an $S$-scheme $T$ and an element 
$u \in \Bigl( J \times_{\Out(G)} \Out(G,P,L) \Bigr)(T)$.
Up to localize for the flat topology we 
may assume that there exists 
$\widetilde g \in \widetilde G(T)$ and $a \in \Aut(G,P,L)(T)$
such that $\widetilde g$ and $a$ have same image in 
$\Out(G)(T)$. It means (again up to localize) 
that there exists $y \in G(T)$
such that $\mathrm{int}(\widetilde g \,  y ) = a$.
It follows that $\widetilde g \,  y$ normalizes
$(P,L)$ so  that  $(\widetilde g \,  y,a)$
defines an element of $\Bigl(\widetilde G \times_{\Aut(G)} \Aut(G,P,L)\Bigr)(T)$ which maps to $u$. The claim is then  established.

According to Lemma \ref{lem_clopen}.(2), 
the map $\Out(G,P,L) \to \Out(G,P)$ is 
a clopen immersion between twisted constant $S$-group schemes.
Claim \ref{claim_JPL} implies that $J_{P,L} \to J$ is a clopen immersion
and also that $J_{P,L}$ is a twisted constant $S$-group.
Finally it is quasi-isotrivial since $J$ is (Lemma \ref{lem_mono_iso}.(2)). 

\sm

\noindent (2) 
According to \S \ref{subsec_descent}, the fppf
quotient $J/J_{P,L}$ is representable by
a  twisted constant $S$-scheme.
On the other hand we know that  $\Out(G)/\Out(G,P,L)$ is  representable 
by finite $S$-\'etale scheme.
The map 
$J/J_{P,L} \to \Out(G)/\Out(G,P,L)$
is a monomorphism so is a clopen immersion according to
Lemma \ref{lem_mono}.(1). Thus $J/J_{P,L}$ is 
 finite $S$-\'etale.

\end{proof}

The conclusion is that the sequence \eqref{eq_JPL}
has the same shape than the initial sequence
$1 \to G \to \widetilde G \to J \to 1$.

\subsection{Normalizers, II}
Now let $P$ be an $S$-parabolic subgroup of $G$.
According to \cite[Proposition 3.4.3]{Gi15}, the fppf
sheaf $N_{\widetilde G}(P)$ is representable
by a smooth $S$--scheme which is closed in $\widetilde G$.
Furthermore the quotient $\widetilde G/ N_{\widetilde G}(P)$ is representable
by a smooth $S$-scheme.
In the same manner as in the previous section, we 
can construct an exact sequence of smooth  $S$--group schemes
$$
1 \to P \to N_{\widetilde G}(P) \to J_P \to 1
$$
such that $J_P$ is a twisted constant $S$-group
and that $J/J_P$ is representable by a finite \'etale $S$-scheme.
A  complement is the following (which extends  \cite[XXII,
Corollaire 5.8.5]{SGA3}).

\begin{slemma}\label{lem_proj} 
 The scheme $\widetilde G/ N_{\widetilde G}(P)$
is a projective $S$-scheme.
\end{slemma}

\begin{proof}  
We establish first that the $S$-scheme  $\widetilde G/ N_{\widetilde G}(P)$ is proper. 
 In view of \cite[Proposition $_2$.2.7.1.(vii)]{EGA4}, the statement is local with respect to the fpqc topology;
 since  $J/J_P$ is finite \'etale over $S$
 we can  assume that $J/J_P=S \coprod  \dots \coprod S$ ($d$ times)
and that there exist $\widetilde g_1 , \dots, \widetilde g_d \in 
\widetilde G(S)$    mapping to the pieces of $J/J_P$.
In this case we have $\widetilde G/N_{\widetilde G}(P)
\simlgr G/P  \coprod \dots \coprod G/P$ ($d$ times)
which is proper over  $S \coprod \dots \coprod S$
so that $\widetilde G/N_{\widetilde G}(P)$ is proper over $S$.

The assignment $\widetilde g \to 
 [^{\widetilde g}\!P]$ defines a monomorphism
 $h: \widetilde G/N_{\widetilde G}(P) \to \Par(G)$.
 Since $\widetilde G/N_{\widetilde G}(P)$ is $S$-proper, 
 the morphism $h$  is proper \cite[Tag 01W6, (2)]{St}
 so is a closed immersion \cite[Proposition $_3$.8.11.5]{EGA4}.
 Since $\Par(G)$ is $S$-projective, we conclude 
 that the $S$-scheme $\widetilde G/N_{\widetilde G}(P)$
 is projective.
\end{proof}

\subsection{Normalizers, III}

\begin{slemma} \label{lem_rep}
Let $T$ be a maximal $S$-torus of $G$
and let $N=N_G(T)$ be the normalizer $S$-group  scheme of 
$T$ \cite[Exp.\ XI, Corollaire 5.3 bis]{SGA3}.

\sm

\noindent  (1) The fppf sheaf normalizer
$\widetilde N=N_{\widetilde G}(T)$ is representable
by a smooth $S$--group scheme which fits in an exact 
sequence
$$
1 \to N \to \widetilde N \to J \to 1.
$$
Furthermore $\widetilde N$ is a closed $S$-subgroup
scheme of $\widetilde G$ and is ind-quasi-affine over $S$.

\sm

\noindent (2) The fppf quotient $\widetilde W=\widetilde N/T$ is representable by a $S$-group scheme
smooth and ind-quasi-affine.

\sm

\noindent (3)  The image of
$H^1(S,\widetilde N)\to H^1(S, \widetilde G)$ consists 
of the classes
of $\widetilde G$-torsors $\widetilde
E$ over $S$ such that the  reductive twisted\footnote{i.e.\ 
$^{\widetilde E}\!G= {\widetilde E} \wedge^{\widetilde G} G$
for the natural action of $\widetilde G$ on $G$.}  
$S$-group scheme $^{\widetilde E}\!G$ admits a maximal $S$-torus. 

\end{slemma}

\begin{proof}
We recall from the quoted reference that $N$ is 
a closed $S$-subgroup scheme of $G$ which is smooth. 

\sm

\noindent
(1) We claim that the map $\widetilde N \to J$ is an epimorphism of fppf sheaves.
 We are given a ring $A$ and an element
$x \in J(A)$ and want to find a fppf cover $B$ of $A$
such that $x_B \in J(B)$ lifts to $\widetilde G(B)$.
The preimage of $x$ in $\widetilde G_A$
is a $G_A$--torsor so is an affine smooth $A$--scheme 
$\Spec(B_1)$. It follows that there exists $\widetilde g_1 \in \widetilde G(B_1)$ mapping to $x_{B_1} \in J(B_1)$.
Next we consider the maximal $B_1$--torus
$^{\widetilde g_1}\! T$ of $G_{B_1}$.
We consider the strict transporter 
$\mathrm{Transpst}_{G_{B_1}}( T_{B_1}, \, ^{\widetilde g_1}\! T)$
as defined in \cite[VI$_B$, \S 6]{SGA3}. This is the
$B_1$-functor defined by
$$
\mathrm{Transpst}_{G_{B_1}}( T_{B_1}, \, ^{\widetilde g_1}\! T)(C)
= \bigl\{ g \in G(C) \, \mid \,  g T_{B_1}(C')g^{-1} = \, (^{\widetilde g_1}\! T)(C')
\enskip \hbox{for each $C$--ring $C'$}  \bigr\}
$$
for each $B_1$--ring $C$.
According to Grothendieck \cite[XI, Corollary 5.4.bis]{SGA3},  it  is representable
by an $N$-torsor over $\Spec(B_1)$ so is an affine  smooth $B_1$--scheme  $\Spec(B_2)$.
It follows that there exists $g_2 \in G(B_2)$
such that    $^{g_2}\!T_{B_1} = \, ^{\widetilde g_1}\! T$
so that $\widetilde n_2=g_2^{-1} g_{1, B_2} \in \widetilde N(B_2)$.
Thus $\widetilde n_2$ maps to $x_{B_2} \in J(B_2)$
and the claim is established.

We have then an exact sequence of flat $S$-sheaves in groups $1 \to N \to \widetilde N \to J \to 1$.
Then $\widetilde N$ is a sheaf $N$-torsor over $J$
so that it is representable by a $J$-scheme, hence
by an $S$-scheme denoted also by $\widetilde N$.
In particular, $\widetilde N$ is affine over $J$
hence ind-quasi-affine over $S$
in view of \cite[Tag 0F1V]{St}.
Since $N$ and $J$ are smooth, 
so is $\widetilde N$ in view of \cite[VI$_B$, Proposition 9.2.(vii)]{SGA3}.

To establish that $\widetilde N$ is closed in $\widetilde G$, we can reason  locally for fppf over $J$  in view of the permanence property \cite[Proposition 2.7.1.(xii)]{EGA4}. But locally over $J$, $\widetilde N \to J$ 
admits a splitting so we are done.

\smallskip

\noindent (2) Moding out the sequence of 
fppf $S$-sheaves in groups $1 \to N \to \widetilde N \to J \to 1$. by $T$ provides an 
exact sequence of fppf $S$-sheaves in groups
$$
1 \to W \to \widetilde W \to J \to 
1.$$
The argument of (1) shows that 
$\widetilde W$ is representable by an $S$-scheme
which is smooth and ind-quasi-affine.

\sm

\noindent (3) 
Let $\widetilde E$ be a $\widetilde G$-torsor
and consider the twisted $S$--group schemes
$G'= \, ^{\widetilde E}\!G$ and $\widetilde G'=\, ^{\widetilde E}\!\widetilde G$.   We note that $G'$ is reductive
and consider the  exact sequence of smooth $S$-group schemes
$1 \to G' \to \widetilde G' \to \, {^{\widetilde E}\!J} \to 1$.
Clearly we have $G' = (\widetilde G')^0$.
According to \cite[Lemme 2.6.2.(2)]{Gi15}, 
the class $[\widetilde E]$ belongs to 
the image of 
$H^1(X,\widetilde N)\to H^1(X, \widetilde G)$ 
if and only if $\widetilde G'$ admits an  $X$--subgroup 
sheaf $H'$
which is locally conjugated to $T \subset \widetilde G$. 
We claim  that the second condition is equivalent
to require that $\widetilde G'=\, ^{\widetilde E}\!G$ admits a maximal $S$--torus. 

 Assume that $\widetilde G'$ admits an  $S$--subgroup 
sheaf $H'$
which is locally $\widetilde G$-conjugated to $T \subset \widetilde G$.
By faithfully flat descent, $H'$ is representable
by an affine $S$-group scheme $T'$. Since $T'$ is an $X$-form of $T$,  $T'$ is an $S$-torus of same rank as $T$. Next $T \to \widetilde G$ is a closed immersion and so is $T' \to \widetilde G'$ 
since this property descends \cite[Tag 02L6]{St}.
Using that  $T'$ has connected fibers, we 
get that $T' \subset G'=(\widetilde G')^0$. Since $T'$ has same rank than $T$,
it is a maximal $S$-torus of $G'$. \newline
\indent 
Conversely, assume that $G'$  admits a maximal $S$-torus $T'$ and we want to establish that $T'$ is locally $\widetilde G$-conjugated to $T \subset \widetilde G$. 
Up to localize in the flat topology, we can assume that
$\widetilde E$ is the trivial $\widetilde G$-torsor
so that we have $G'=G$ and  $\widetilde G'=
\widetilde G$. Since the maximal $S$-tori
$T$ and $T'$ of $G$  are locally
$G$-conjugated for the \'etale 
topology \cite[XII,  Theorem 1.7.(c)]{SGA3}, 
they are a fortiori $\widetilde G$-conjugated for the flat topology.
\end{proof}

\subsection{Reductive group schemes and reduciblility}\label{sub_reductive}

Let $H$ be a reductive $S$-group scheme.
We denote by $\Par(H)$ the total scheme 
of parabolic subgroups of $H$ \cite[XXVI, \S 3.2]{SGA3}.
It decomposes as  $\Par(H) =\Par^+(H) \coprod S$,
where $S$ corresponds to the fact that  $H$ itself is a  parabolic
subgroup scheme. The scheme $\Par^+(H)$ is called the total scheme
of proper parabolic subgroups and is  also  projective over $S$.

We say that $H$ is {\it reducible} if it admits 
a proper  parabolic subgroup  $P$  such that $P$ contains a Levi subgroup $L$ The opposite notion is  {\it irreducible}.

If $S$ is affine, the notion of reducibility for $H$ is equivalent to the existence of a proper parabolic subgroup  $P$ \cite[XXVI, Corollaire 2.3]{SGA3}, so there is no ambiguity with the terminology of \cite{GP07}.

We say that $H$ over is {\it  isotropic}
if   $H$ admits a subgroup isomorphic to $\GG_{m,S}$. The opposite notion is {\it anisotropic}. According to \cite[Theorem 7.3.1.(ii)]{Gi15}, if $S$
is connected,  $H$ is isotropic if and only if $H$ is reducible or the radical torus $\rad(H)$ is isotropic. 

By extension, if an  $S$-group scheme $M$ acts on  $H$,
we say that the action is {\it reducible} if 
it normalizes  a couple $(P,L)$ where $P$ is a  proper  parabolic subgroup  of $H$ and $L$ a Levi subgroup of $P$. The action is otherwise called {\it irreducible}.

We say that the  action of $M$ on $H$ is  {\it isotropic} if 
it centralizes  a  $S$-subgroup $\GG_m$ of $H$.
Otherwise the action  is {\it anisotropic.}  
We shall use several times the next equivalence
$(i) \Longleftrightarrow (ii)$  which was implicitely used in \cite[\S 2.4]{GP13}.

\begin{sproposition} \label{prop_richardson} 
Assume that $S$ is connected. The following are equivalent:

\sm

(i) The action of $M$ on $H$ is isotropic;

\smallskip

(ii) The action of $M$ on $\rad(H)$ is isotropic or
the action of $M$ on $H$ is reducible;

\smallskip

Furthermore assume that the centralizer $C_H(M)$ is representable
by an $S$--group scheme. Then (i) and (ii) are also equivalent to

\sm

(iii) $C_H(M)$ admits an $S$-subtorus $\GG_{m,S}$.

\end{sproposition}

\begin{proof}
The  equivalence $(i) \Longleftrightarrow (ii)$ 
is \cite[Corollary \ B.2]{GN}. The equivalence 
 $(i) \Longleftrightarrow (iii)$ is obvious. 
\end{proof}

A useful  complement is the following.

\begin{sproposition} \label{prop_morozov} \cite[Proposition B.3]{GN}
Assume that $S=\Spec(R)$ is affine and connected, that $M$ is 
a flat affine  $R$--group scheme whose geometric fibers are linearly reductive (e.g. $M$ is of multiplicative type).
If $M$ normalizes an $R$--parabolic subgroup $P$ of
$H$,  there exists $\lambda: \GG_m \to H$ which is $M$-invariant such that $P=P_H(\lambda)$. In particular  $L=C_H(\lambda)$ is   a Levi subgroup of $P$ which is normalized by $M$.
\end{sproposition}

\section{Loop torsors on varieties}  
We come back here on variant of results of \cite{CGP, GP13}.
We extend them to the  characteristic free  case
and to a wider class of $k$--group schemes.

\subsection{Various fundamental groups}
Let $k$ be a field.
We denote by $p \geq 1$ the characteristic exponent of $k$ and by 
 $k_s/k$ be an absolute Galois closure. We denote by
$\Gamma_k=\Gal(k_s/k)$ the absolute Galois group.
Let $X$ be a geometrically connected  $k$-scheme
of finite type.
 We  choose a $k_s$-point $x$, this is a quasi geometric point of $X$ so that we can deal with the fundamental 
 group $\pi_1(X,x)$ as explained in \S \ref{subsec_galois}.
We remind the following important sequence
to the reader.

\begin{proposition} \cite[Proposition 3.3.7]{Fu}\label{prop_homotopy}
We have an  exact sequence of profinite 
groups
\begin{equation}\label{eq_fund0}
1 \to \pi_1( X_{k_s}, x) \to \pi_1( X, x) \to \Gal(k_s/k) \to 1.
\end{equation}
\end{proposition}

By using geometric base points, 
this is equivalent to the homotopy
sequence of Grothendieck \cite[th\'eor\`eme IX.6.1]{SGA1}.

For a profinite group $\cG$, we denote by $\cG^{(p')}$ its maximal prime to $p$
quotient it is the quotient of $\cG$ by the closure of the subgroup
$\cG^{[p]}$ generated by the pro-$p$--Sylow subgroups of $\cG$.
We can define then $\pi_1( X, x)^{(p')}$.
Since $\pi_1( X_{k_s}, x)^{[p]}$ is normal and closed in $\pi_1( X, x)$,
we can define the quotient $\pi_1( X, x)^{(p'-geo)}= \pi_1( X, x) / \pi_1( X_{k_s}, x)^{[p]}$.
Altogether we have a commutative  exact diagram of exact sequences of 
profinite groups
\[
\xymatrix{
1 \ar[r]& \pi_1( X_{k_s}, x) \ar[r] \ar[d]& \pi_1( X, x) \ar[d] \ar[r]&\Gal(k_s/k) \ar[r]
\ar[d]^{=}& 1 \\
1 \ar[r]& \pi_1( X_{k_s}, x)^{(p')} \ar[d]^{=} \ar[r]& \pi_1( X, x)^{(p'-geo)}
\ar[d] \ar[r]&\Gal(k_s/k) \ar[d] \ar[r]& 1 \\
1 \ar[r]& \pi_1( X_{k_s}, x)^{(p')} \ar[r]& \pi_1( X, x)^{(p')} \ar[r]&\Gal(k_s/k)^{(p')} \ar[r]& 1 .
}
\]
 By Galois theory we have associated the cover $X^{sc, p'-geo}$ (resp.\
 $X^{sc, p'}$) of $X$ named the universal geometrical $p'$-cover  (resp.\ 
 universal $p'$--cover). In particular $k_s$ admits the
 maximal $p'$-Galois subextension $k^{(p')}$.

On the other hand, if $X$ is smooth,
 Kerz and Schmidt showed that the four notions of 
 tamely ramified covers of $X$ coincide \cite[Theorem 4.4]{KS}.
 As explained in  the introduction of \cite{HS},
 tameness of a covering should be thought of as ``at most tamely ramified along the
boundary of compactifications over the  base''.
 We use mostly here the so-called {\it
divisor-tameness} definition. 
We say  then that a finite connected \'etale cover 
$Y \to X$ is {\it divisor-tame} if 
for every normal compactification $X^c$ of $X$ and every point $x \in X^c \setminus X$
of codimension 1,  the discrete rank one valuation $v_x$ on $k(X)$ associated with
$x$ is tamely ramified in the finite, separable field extension $k(Y )/k(X)$
 \footnote{Denoting by $K_x$ the completion of 
$k(X)$ for the valuation $v_x$, it means that in the decomposition $ k(Y) \otimes_{k(X)} K_x= L_1 \times \dots \times L_r$ in valued fields,
 each finite separable field extension 
$L_i/K_x$ is tamely ramified, that is, the residue field extension  is separable and the ramification indices are
coprime with the characteristic exponent of the residue field.
 }.
Let $(X^{tsc}, x^{tsc})$ be the {\it simply connected tame cover} of 
$(X,x)$. We have the fundamental exact sequence of profinite 
groups
\begin{equation}\label{eq_fund}
1 \to \pi_1^t( X_{k_s}, x) \to \pi_1^t( X, x) \to \Gal(k_s/k) \to 1.
\end{equation}
Since $p'$-Galois covers of $X_{k_s}$ are tame,
$\pi_1^t( X_{k_s}, x)$ maps onto $\pi_1( X_{k_s}, x)^{(p')}$
so that we have   a commutative  exact diagram of exact sequences of 
profinite groups
\[
\xymatrix{
1 \ar[r]& \pi_1( X_{k_s}, x) \ar[r] \ar[d]& \pi_1( X, x) \ar[d] \ar[r]&\Gal(k_s/k) \ar[r]
\ar[d]^{=}& 1 \\
1 \ar[r]& \pi_1^t( X_{k_s}, x) \ar[d] \ar[r]& \pi_1^t( X, x)
\ar[d] \ar[r]&\Gal(k_s/k) \ar[d] \ar[r]& 1 \\
1 \ar[r]& \pi_1( X_{k_s}, x)^{(p')} \ar[r]& \pi_1( X, x)^{(p'-geo)} \ar[r]&\Gal(k_s/k) \ar[r]& 1. 
}
\]

\begin{sremark}{\rm A basic tamely ramified cover is of the shape $Y_l$
where $Y \to X$ is a geometrically connected  $M$-torsor for some finite 
quotient $M$ of $\pi_1^t(X,x)$ and $l/k$ a finite Galois extension
such that $M_l$ is constant. In this case we have $\Gal(Y_l/X)=M(l) \rtimes \Gal(l/k)$.  We observe that every continuous finite quotient of 
$\pi_1^t(X,x)$ factorizes through such a cover.
}
\end{sremark}

\subsection{Loop torsors}
Let $G$ be a locally algebraic group defined over $k$.
We get then a map $$
H^1(\pi_1(X,x), G(k_s)) \to
H^1(\pi_1(X,x), G(X^{sc})) \hookrightarrow H^1(X,G)
$$
 whose image is denoted by $H^1_{loop}(X,G)$.
 We say that a sheaf $G$--torsor over $X$ is 
 {\it loop} if its class belongs to $H^1_{loop}(X,G)$.

We have variants called tame loop torsors
(resp.\  $p'$-loop torsors) when we consider 
the group $\pi_1^t(X,x)$ (resp. \, $\pi_1(X,x)^{(p'-geo)}$).
This is gives rise to the subsets 
$$
H^1_{\substack{p'-geo \\ loop}}(X,G) \subseteq  H^1_{\substack{tame \\ loop}}(X,G) \subseteq H^1_{loop}(X,G)
.
$$
Finally the image of the map $$
H^1(\pi_1(X,x)^{(p')}, G(k^{(p')}) \bigr) \to
H^1(\pi_1(X,x)^{(p')}, G(X^{sc,p'})) \hookrightarrow H^1(X,G)
$$
is called the class  of {\it $p'$-loop torsors}.
We denote $H^1_{iso}(X,G)$ (resp.\ $H^1_{p'-geo-iso}(X,G)$, $H^1_{p'-iso}(X,G)$) the subset 
of isotrivial (resp. $p'$-geo-isotrivial, $p'$-geo)
sheaf $G$-torsors $E$, that is, such that $E$ is
trivialized by a finite \'etale cover 
(resp.\ finite \'etale cover $p'$-geometric, 
finite \'etale cover of $p'$-degree).

\begin{slemma} 
(1) Assume that $G$ is a finite \'etale group 
of prime to $p$ order. Then 
$$
H^1_{p'-loop}(X,G) 
= H^1_{p'-iso}(X,G)=  H^1(X,G).
$$

\smallskip

\noindent (2) Assume that $G$ is  commutative
and connected.  Then we have 
$$
H^1_{p'-loop}(X,G)
= H^1_{p'-iso}(X,G)= H^1(X,G). 
$$
\end{slemma}

\begin{proof}
(1) The direct inclusions
$H^1_{p'-loop}(X,G)
\subset  H^1_{p'-iso}(X,G) \subset  H^1(X,G)$ are
obvious. Let  $\gamma =[E] \in H^1(X,G)$.
Then $E \to X$ is a finite \'etale cover 
of $p'$ degree which anihilates $\gamma$.
We have proven that $\gamma$ belongs to 
$H^1_{p'-iso}(X,G)$.
It follows that there exists a Galois cover
$Y \to X$ of $p'$ degree which anihilates $\gamma$.
It decomposes as $Y \to X_K \to X$ for a Galois $p'$-degree extension $K$ of $k$ such that...
Since $G(K)=G(Y)$
we have 
$$
H^1(\Gal(Y/X), G(K))= H^1(\Gal(Y/X), G(Y))
= \ker\big( H^1(X,G) \to H^1(Y,G) \big). 
$$
\sm

\noindent (2)
Here $G$ is then a commutative algebraic $k$-group.
 Let $\gamma \in H^1(X,G)_{p'-tors}$, it means that there 
exists $m$ prime to $p$ such that $m \gamma =0$.
Using the exact sequence $1 \to {_mG} \to G \xrightarrow{\times m} G \to G \to 1$, it follows that 
$\gamma$ arises from a class $\alpha \in H^1(X,{_mG})$.
According to \cite[Expos\'e VII, \S 8.4]{SGA3},
the $k$-group ${_mG}$ is \'etale and is of prime to $p$ order. Part (1) shows that $\alpha \in 
 H^1_{p'-loop}(X,{_mG})=H^1(X,{_mG})$
 so that $\gamma \in H^1_{p'-loop}(X,G)$. 

 The inclusion  $H^1_{p'-loop}(X,G) \subset H^1_{p'-iso}(X,G) $ is obvious. Let $\gamma \in H^1_{p'-isot}(X,G)$.
 It is split by a finite Galois $p'$-cover $f: Y \to X$.
 We consider the Deligne trace $tr_f: f_*f^*G \to G$.
 According to \cite[Exp.\ XXIV, Proposition 8.4]{SGA3}, the natural map
 $H^1(X, f_*f^*G) \to H^1(Y,G)$ is an isomorphism
 and gives rise to a trace map $Tr: H^1(Y,G) \to H^1(X,G)$
 and the composite $H^1(X,G) \xrightarrow{f^*}  H^1(Y,G) \xrightarrow{Tr} H^1(X,G)$ is the multiplication by $deg(f)$. It follows that 
 $deg(f)\, \gamma=0$. Thus
  $\gamma \in H^1(X,G)_{p'-tors}$.
\end{proof}

From now on we assume that $x \in X(k_s)$
is a $k$-point.
The sequence \eqref{eq_fund0} comes with a
splitting $s_x: \Gal(k_s/k) \to \pi_1( X, x)$.
We have then a decomposition 
$$
\pi_1(X,x) = \pi_1(X_{k_s},x) \rtimes \Gal(k_s/k).
$$ 
It follows that the profinite group $\pi_1(X,x)$
is equipped with a structure of affine algebraic $k$--group.
We are given a loop cocycle $\eta \in  Z^1\big(\pi_1(X,x), \bG(k_s)\big)$.  Its  restriction
$\eta_{\mid \Gamma_k}$ is called the {\it arithmetic part} of $\eta$
and its denoted by $\eta^{ar}$: it is  an element of $Z^1\big(\Gamma_k, \bG(k_s)\big)$.
Next we consider the restriction of $\eta$ to
  $\pi_1( X_{k_s} ,x)$ that we denote by $\eta^{geo}$
and called the {\it geometric part} of $\eta$.
By taking into account the Galois action,
the map $\eta^{geo}: \pi_1( X_{k_s} ,x) \to \, _{\eta^{ar}}\!G$ is a homomorphism of algebraic $k$--groups.

\begin{slemma} \label{lem_dico} We assume that $G$ is ind-quasi-affine
and locally algebraic.

\sm

\noindent (1) The map $\eta \to (\eta^{ar}, \eta^{geo})$
provides a bijection between $Z^1\bigl( \pi_1( X,x), G(k_s)
 \bigr)  $ and the couples
$(z, \phi)$ where $z \in Z^1\bigl( \Gamma_k,  G(k_s) )$ and 
$\phi: \pi_1( X ,x) \to {_zG}$
is a $k$--group homomorphism.

\smallskip

\noindent (2) We have an exact sequence of pointed sets
\begin{equation} \label{eq_conj}
1 \to \Hom_{k-gr}(  \pi_1( X_{k_s} ,x),G ) /
G(k) \to H^1\bigl(  \pi_1( X ,x) ,  G(k_s) \bigr)
\xrightarrow{Res_x}  H^1\bigl( \Gamma_k ,  G(k_s) \bigr)
\to 1.
\end{equation}
Furthermore the first map is injective.

\smallskip

\noindent (3) We have a decomposition
$$
H^1\bigl(  \pi_1( X ,x) ,  G(k_s) \bigr)
= \coprod\limits_{[z ] \in H^1(k,G) }
\Hom_{k-gr}(  \pi_1( X_{k_s} ,x),{_zG} ) /
{_zG}(k). 
$$
\end{slemma}

\begin{proof}
(1) This is similar with \cite[lemma 3.7]{GP13}.

\smallskip

\noindent (2) Part (1) defines a map 
$$\Hom_{k-gr}(  \pi_1( X_{k_s} ,x),G ) \to 
Z^1\bigl( \pi_1( X,x), G(k_s)
 \bigr)  \to H^1\bigl( \pi_1( X,x), G(k_s)
 \bigr)$$ which is $G(k)$-invariant.
It applies a $k$--homomorphism
$\phi:   \pi_1( X_{k_s} ,x) \to  G$ to the
loop cocycle $\widetilde \phi:  \pi_1( X ,x) \to  G$
defined by $\widetilde \phi( \tau ,\sigma)= 
\phi(\sigma)$ for $(\tau, \sigma) \in 
{\pi_1( X,x) \rtimes \Gamma_k}$.
 The first part shows that the only  thing to
  do is to establish injectivity for the map
 $\Hom_{k-gr}(  \pi_1( X_{k_s} ,x),G ) /
G(k) \to H^1\bigl(  \pi_1( X ,x) ,  G(k_s) )$.
We are given two $k$--homomorphisms
$\phi, \psi: \pi_1( X ,x) ,  G(k_s))$
having same image in $
H^1\bigl(  \pi_1( X ,x) ,  G(k_s) )$.
This means that there exists $g \in G(k_s)$
such that $\widetilde \psi(\tau, \sigma)= g^{-1} \,
\widetilde \phi(\tau, \sigma) \, {^\sigma \! g}$
for all $(\tau, \sigma) \in \pi_1( X,x) \rtimes
\Gamma_k$ (observe that $\tau$ acts trivially on $G(k_s)$).
Since $1= \widetilde \psi(1, \sigma)= g^{-1} \,
\widetilde \phi(1, \sigma) \, {^\sigma \! g}= 
g^{-1} \, {^\sigma g}$ we obtain that $g \in G(k)$.
Thus $\phi$ and $\psi$ are $G(k)$--conjugated.

\smallskip

\noindent (3)  By considering the fibers of the surjective map
$\Res_x: H^1\bigl(  \pi_1( X ,x) ,  G(k_s) \bigr) \to H^1(k,G)$,
we have a decomposition 
$$
H^1\bigl(  \pi_1( X ,x) ,  G(k_s) \bigr)
= \coprod\limits_{[z] \in H^1(k,G)} \Res_x^{-1}([z]).
$$
Assertion (2) provides a bijection 
$ \Hom_{k-gr}\bigl(\pi_1( X_{k_s} ,x),G \bigr) /
G(k) \simlgr    \Res_x^{-1}([1])$
and the usual twisting argument provides a bijection 
$ \Hom_{k-gr}\bigl(  \pi_1( X_{k_s} ,x), \, {_zG} \bigr) /
{_zG}(k) \simlgr    \Res_x^{-1}([z])$ for each $[z] \in H^1(k,G)$.
The above decomposition provides then the desired decomposition.
\end{proof}

We have of course analogous statements for the other kinds
of loop torsors.

\begin{sremark} \label{rem_wild}{\rm
Loop torsors which are not tame are called wild. An important example
is the following. Assuming that $p>1$, we consider the Artin -Schreier cover
$f: \GG_{a,k} \to \GG_{a,k}$, $x \mapsto x^p-x$. This is a cyclic Galois cover
of order $p$ which is widely ramified at infinity.
If $g \in G(k)$ is an element of order $p$, it defines a loop cocycle for 
$G$  and this cover.
 This construction suggested by Serre is a key ingredient to the study 
of unipotent elements of semisimple groups, see \cite{Gi02}. 
}
\end{sremark}

\subsection{Extensions of reductive groups} \label{subsec_ext}
We are mostly interested in $k$--groups of the next shape.
Let $1 \to G \to \widetilde G \to J \to 1$ be an exact
sequence of locally algebraic groups with
$G=(\widetilde G)^0$ reductive and $J$ is 
twisted constant.
It follows that $J$ is \'etale, that $\widetilde G$ is smooth 
 and that $J$ is the group of 
connected components of $\widetilde G$ \cite[II.5.1.8]{DG}.
We have seen that the $\widetilde G$--sheaf torsors are representable and
also that  the inner
 twist of $\widetilde G$ by such a torsor
 is representable by a locally algebraic group 
 which is of the same shape than $\widetilde G$ (see \S \ref{subsec_isotrivial}).
We observe that  $J(k_s)=J(X^{sc})$ so that

\medskip

\begin{equation}\label{eq_J}
H^1\bigl(  \pi_1( X ,x) ,  J(k_s) \bigr) 
= H^1\bigl(  \pi_1( X ,x) ,  J(X^{sc}) \bigr)=
\ker\bigl( H^1(X,J) \to  H^1(X^{sc},J) \bigr)
\end{equation}

\noindent in view of \cite[Corollary 2.9.2]{Gi15}.
In other words, this kernel consists of 
loop torsors.

\begin{slemma} \label{lem_J}
 Let $\phi,\phi'$ be two loop cocycles with value
in $\widetilde G(k_s)$ having same image in 
$H^1(X,J)$. Then there exists $\widetilde g \in \widetilde G(k_s)$
such that $\phi$ and $\sigma \mapsto \widetilde g^{-1} \, \phi' \, \sigma(\widetilde g) $ have same image in 
$Z^1\bigl(  \pi_1( X ,x) ,  J(k_s) \bigr)$.
\end{slemma}

\begin{proof}
According to the fact \eqref{eq_J},  the  loop cocycles $\phi$, $\phi'$
have same image in $Z^1\bigl(  \pi_1( X ,x) ,  J(k_s) \bigr)$.
Since $\widetilde G(k_s)$ maps onto $J(k_s)$,
it follows that  there exists $\widetilde g \in \widetilde G(k_s)$
such that $\phi$ and $\sigma \mapsto \widetilde g^{-1} \, \phi' \, \sigma(\widetilde g) $ have same image in 
$Z^1\bigl(  \pi_1( X ,x) ,  J(k_s) \bigr)$.
\end{proof}

\sm

We say that a loop cocycle $\phi: \pi_1(X,x) \to 
\widetilde G(k_s)$
is {\it reducible} if the $k$-homomorphism $\phi^{geo}: \pi_1(X_{k_s},x)
\to  {_{\phi^{ar}}\!{\widetilde G}}$ is reducible, that is,  normalizes a pair $(P,L)$ where $P$ is a proper  parabolic $k$--subgroup of 
$_{\phi^{ar}}\!{G}$ and $L$ a Levi subgroup of $P$.

\begin{slemma} \label{lem_irr} 
Let $\phi \in Z^1\bigl(  \pi_1( X ,x) , G(k_s) \bigr)$ be a  purely geometric loop cocycle. 
Let $(P,L)$ be a pair normalized by $\phi^{geo}$
where $P$  is  a $k$--parabolic subgroup of 
$G$  and $L$ is a Levi $k$--subgroup of $P$.
We assume that $(P,L)$ is minimal for this property (with respect to the inclusion). Then the loop cocycle $\phi$ takes value in $N_{\widetilde G}(P,L)(k_s)$ and it is irreducible seen as loop cocycle 
for $N_{\widetilde G}(P,L)$.
\end{slemma}

\begin{proof} We put $\widetilde L= N_{\widetilde G}(P,L)$.
The assumption implies that the 
geometric loop cocycle $\phi$ takes value in $\widetilde L(k_s)$. We assume that  the 

$\psi$ of $\phi$
in $Z^1( \pi(X,x), \widetilde L(k_s))$ is  reducible, that is, there exists a pair $(Q,M)$ normalized by $\psi^{geo}=\phi^{geo}$ such that 
$Q$ is  proper $k$--parabolic subgroup of $L= (\widetilde L)^0$ and $M$ a Levi subgroup of $Q$.
We have a Levi decomposition $P=U \rtimes L$
and remind  the reader that $P'=U \rtimes Q$
is a  $k$--parabolic subgroup of $G$ satisfying
$P' \subsetneq P$ \cite[Proposition 4.4.c]{BoT65}. 
Also  $M$ is a Levi subgroup 
of $P'$ normalized by  $\phi^{geo}$ 
contradicting the minimality of $(P,L)$.
\end{proof}

We say that a loop cocycle $\phi: \pi_1(X,x) \to 
\widetilde G(k_s)$
is {\it isotropic} if the $k$-homomorphism $\phi^{geo}: 
\pi_1(X_{k_s},x) 
\to  {_{\phi^{ar}}\!{\widetilde  G}}$ is isotropic,
 that is,   centralizes a non trivial  $k$--split subtorus of 
$_{\phi^{ar}}{G}$. We record now the formal following
consequence of Proposition \ref{prop_richardson}.

\begin{scorollary} \label{cor_PY} The following are equivalent:
\sm

(i) $\phi$ is isotropic;

\sm

(ii) $\phi$ is reducible or the torus $({_{\phi^{ar}}C}^{\phi^{geo}})^0$
is isotropic.

\end{scorollary}

\section{Loop torsors on Laurent polynomials}

\subsection{Basic tame \'etale covers}
Let $n \geq 1$ be an integer. 
For $r=0,..,n$,
We denote by $R_{r,n}= k[t_1^{\pm 1}, \dots, t_r^{\pm 1}, t_{r+1}, \dots ,t_n]$
the ring of partial  Laurent polynomials and by $K_{r,n}=k(t_1, \dots , t_n)$ its fraction field.
For $m \geq 1$ prime to the characteristic exponent of $k$, we put
$R_{r,n,m}=k [t_1^{\pm \frac{1}{m}}, \dots, t_r^{\pm \frac{1}{m}}, t_{r+1}, \dots ,t_n]$.

We take $1 \in \GG_m^r(k) \times \mathbb{A}^{n-r}(k) $ as base point.
An example of tame cover of $R_{r,n}$ is 
$R_{r, n,m} \otimes_k l= l[t_1^{\pm \frac{1}{m}}, \dots, t_r^{\pm \frac{1}{m}}, t_{r+1}, \dots ,t_n]$ 
where  $l$ is a finite Galois field extension of $k$ containing a primitive $m$--root of unity. 
Covers of this shape are called  basic tame covers.

\begin{slemma}\label{lem_tame}
(1) The above basic tame cover  is Galois and
 we have \break $\Gal(R_{r,n,m} \otimes_k l/R_{r,n})= \mu_m(l)^r \rtimes \Gal(l/k)$.

\smallskip

\noindent (2) 
 The inductive limit of the  basic tame covers is
the universal  tame cover of $R_{r,n}$ and 
$\pi_1^{tame}(R_{r,n},1)= \ZZ'(1)^r \rtimes \Gal(k_s/k)$.  
\end{slemma} 

\begin{proof}
(1)The $R_{r,n}$--algebra $R_{r,n,m} \otimes_k l$ is connected, 
\'etale and free of rank $m^r [l:k]$ over $R_{r,n}$.
 The finite group $\Gamma=\mu_m(l)^r \rtimes \Gal(l/k)$
acts on $R_{r,n,m} \otimes_k l$ by 
$$\bigl((\zeta_1,\dots, \zeta_r) \sigma \bigr) \, . \, 
(t_i^{1/m} \otimes x)= \zeta_i \, t_i^{1/m} \,  \sigma(x)$$
for $i=1,\dots, r$ and trivially on $t_{r+1}, \dots , t_n$.
According to \cite[Proposition 5.3.7]{Sz},
the $(R_{r,n,m} \otimes_k l)^\Gamma$-algebra 
$R_{r,n,m} \otimes_k l$ is Galois of group $\Gamma$.
Since $R_{r,n}=(R_{r,n,m} \otimes_k l)^\Gamma$, we 
obtained the wished statement.

\sm

\noindent (2) We consider the smooth compactification 
 $X^c=(\PP^1_k)^n$ of $X=(\GG_{m,k})^r \times \mathbb{A}^{n-r}_k$ and its boundary
 is a  normal crossing divisor. In this case
 an \'etale connected cover $Y \to X$ is 
tamely ramified
if the discrete valuations of $k(X)$ defined by the prime divisors of 
$X^c \setminus X$ ramify tamely in the extension $k(Y)/k(X)$
\cite[Theorem 4.4]{KS}. 
Without loss of generality we can 
assume that $k$ is separably closed and we have to prove that 
the map $\pi^t_1(R_{r,n},1) \to  \ZZ'(1)^r$ (arising from
the basic covers) is an isomorphism.
The  compactification $\PP^1_k$ of $\GG_{m,k}$
is obviously good in the sense   that it  is 
the complement a  normal crossing divisor. According to 
Orgogozo's theorem \cite[Theorem  5.1]{O}, we have a  decomposition 
 $\pi^t_1(R_{r,n},1)= (\pi_1^t(R_1, 1))^r \times (\pi_1^t(k[t], 1))^{n-r} $.
We are reduced then to the case of $\GG_m$ and of $\GG_a$.
In the case of $\GG_a$, this is well-known, 
see \cite[Theorem 1]{Ka}.
We deal then  with the case of $\GG_m$, 
we have to show that any connected Galois tame cover
is a Kummer cover. 
Let $f: Y \to \GG_m$ be a tamely ramified connected Galois
cover of group $G$. 
Then the field extension $k(Y)/k(t)$ is tamely ramified at $0$
and $\infty$ so that $k(Y) \otimes_{k(t)} k((t)) \cong
k((t^{\frac{1}{a}}))^{G/G_0}$
and $k(Y) \otimes_{k(t)} k((t^{-1})) \cong
k((t^{- {\frac{1}{b}}} ))^{(G/G_0)}$ where 
$G_0$ (resp.\ $G_\infty$) is the inertia group
at $0$ (resp.\ $\infty$) and $a$ (resp.\ $b$) the 
ramification index at $0$ (resp.\ $\infty$).
By assumption $a$ and $b$ are prime to the characteristic exponent of $k$.
We put $n = g.c.m.(a,b)$ and consider the 
Kummer cover $h_n : \GG'_m \to \GG_m$, $u \mapsto t^n$.
Then $Y \times_{\GG_m} \GG'_m$ is a finite $G$--cover of
$\GG'_m$ which is unramified at $0$ and $\infty$.
Since the projective line is simply connected, 
it follows that $Y \times_{\GG_m} \GG'_m = \GG'_m \times_k G_k$.
It follows that $f$ is dominated by $h_n$, so that 
$f$ is a Kummer cover as well.
\end{proof}

\begin{sremarks} {\rm (a) In the case $p>1$,
Katz-Gabber's correspondence provides many more \'etale
covers of $\GG_m$ \cite[Theorem 1.4.1]{Kz}.

\sm

\noindent (b) Since $\pi^{p'-geo}(R_{r,n},1)$ is a quotient of 
 $\pi^{t}(R_{r,n},1)$,  Lemma \ref{lem_tame}.(2)  shows that the map 
 $\pi^{t}(R_{r,n},1) \to \pi^{p'-geo}(R_{r,n},1)$ is an isomorphism.
 
\sm

\noindent (c) We observe that basic tame basic covers
have trivial Picard groups and so have the other tame covers
by using inflation-restriction sequences.
This is not true anymore for arbitrary
finite \'etale covers if $p \geq 2$ even 
in the case $k$ algebraically closed.
According to \cite[Theorem 4 .6]{MP}, there are examples
of connected Galois covers $X$ (for suitable altenating groups) of the  affine line $\mathbb{A}_k^1$ 
which are smooth affine curve whose 
smooth completion   has positive genus.
It follows that $X$ has an infinite Picard group 
and so has the cover $X \times_{\mathbb{A}^1_k} \GG_m$ of $\GG_m$.
}
\end{sremarks}

\subsection{Fixed point statement}

The following is  a mild generalization of 
\cite[Theorem \ 7.1]{GP13} so that we let the reader
to check that its proof can be readily adapted.

\begin{stheorem} \label{thm_GP_fix}
Let $G$ be a $k$-group scheme  locally of finite presentation
acting on a projective  $k$--scheme $Z$.
Let $\phi$ be a tame loop cocycle for $G$ and $R_{r,n}$. Then 
\break $Y=  \bigl( _{\phi^{ar}}Z\bigr)^{\phi^{geo}}$ is  a
 projective $k$--scheme and  
the following are equivalent:

\sm

(i) $Y(k) \not \not = \emptyset$;

\sm 

(ii)  $Y(R_{r,n}) \not \not = \emptyset$;

\sm 

(iii) $({_\phi\!Z})(R_{r,n}) \not = \emptyset$;

\sm

(iv)  $({_\phi\!Z})(K_{r,n}) \not = \emptyset$;

\sm 

(v) $({_\phi\!Z})(F_{n,r}) \not = \emptyset$
where $F_{n,r}=k((t_1))\dots ((t_n))$.
\end{stheorem}

\begin{sremarks}\label{rem_proj}{\rm  (a) Note that the projectivity
of $Z$ is used to insure  that  
the twisted fppf  $R_{r,n}$--sheaf ${_\phi\!Z}$ by Galois descent is representable according to \cite[\S 6.2]{BLR}.

\sm

\noindent (b) If $Z$ is proper smooth over $k$,
the statement is still true if the various twists are understood
in the category of flat sheaves.
}
\end{sremarks}

Let $1  \to G \to \widetilde G \to J \to 1$ be an exact 
sequence of $k$-groups as in \S \ref{subsec_ext}.

\begin{sproposition} \label{prop_GP_red} 
Let $\phi: \pi^t(R_{r,n},1) \to \widetilde G(k_s)$ be a tame loop cocycle. Then the following are equivalent:

\smallskip

(i) $\phi$ is reducible;

\smallskip

(ii) the $R_{r,n}$-reductive group scheme ${_\phi G}$
is reducible; 

\smallskip

(iii)  $({_\phi G})_{K_{r,n}}$ is reducible;

\smallskip

(iv)  $({_\phi G})_{F_{n,r}}$ is reducible.
  
\end{sproposition}

\begin{proof}  We apply Theorem \ref{thm_GP_fix} to the total variety of
parabolic subgroups $Z= \Par^+(G)$, see \S \ref{sub_reductive}. It is projective 
and is equipped with a natural action of $\widetilde G$.
We have a canonical isomorphism of $R_{r,n}$-schemes
${_\phi\!\Par^+(G)} \simlgr  \Par^+({_\phi\! G})$ 
so that the assertions (iii), (iv) and (v) of 
Theorem \ref{thm_GP_fix} corresponds respectively to 
the assertions (ii), (iii) and (iv) of Proposition \ref{prop_GP_red}; again reductibility  in those cases is
equivalent to the existence of a proper parabolic subgroup.
 It remains to deal with (i).  
We have $Y=  \bigl( _{\phi^{ar}}\!\Par^+(G)\bigr)^{\phi^{geo}}
=\bigl( \Par^+(_{\phi^{ar}}\!G)\bigr)^{\phi^{geo}}$
so that $Y(k)$ is the set of proper parabolic $k$--subgroups of $_{\phi^{ar}}\!G$
which are normalized by $\phi^{geo}$.
According to Proposition \ref{prop_morozov},
$P$ admits a Levi subgroup $L$ which is normalized by 
$\phi^{geo}$. Thus $\phi$ is reducible.
\end{proof}

\subsection{Tame Galois cohomology}

We deal with  the integer $r\in [0,n]$ 
and consider the extensions $F_{r,n,m}=F((t_1^{1/m}))\dots 
((t_r^{1/m}))((t_r))\dots ((t_n))$ of $F_{n,r}$
for $(m,p)=1$; we use also $F_{n,m}=F_{n,n,m}$ in 
the maximal case $r=n$.

We define the tame Galois cohomology set by 
by
$$
H^1_{r-tame}(F_{n,r}, \widetilde G)= \bigcup H^1( F_{r,n,m} \otimes_k l/F_n, \widetilde G).  
$$
where $m$ runs over the positive integers which are
prime to the characteristic exponent $p$  of $k$ and $l$ runs over the (finite)
Galois extensions of $k$.

\subsection{Acyclicity}

We extend the injectivity part of \cite[Theorem 8.1]{GP13} beyond the affine case
and in characteristic free.
Let $1 \to G \to \widetilde G \to J \to 1$ be 
a sequence of smooth  $k$--groups
such that $G= (\widetilde G)^0$ is reductive
and $J$ is a twisted constant $k$-group scheme.
The main result of this section is the following.

\begin{stheorem} \label{thm_gp_main}
The map $H^1_{\substack{tame \\ loop}}(R_{r,n},\widetilde G) \to H^1_{r-tame}(F_{n,r}, \widetilde G)$
is injective.
\end{stheorem}


\begin{sremarks}{\rm
(a) If $k$ is of characteristic zero and $\widetilde G$ is linear,  then bijectivity holds \cite[Theorem 8.1]{GP13}.

\noindent 
(b) In \cite[Example 2 page 18]{CGP24}, there is an example of a field $k$ of characteristic $p>0$ such that the  map $H^1_{\substack{tame \\ loop}}(k[t^{\pm 1}],\PGL_p) \to H^1_{1-tame}(k((t)), \PGL_p)$ is not onto. 
}
\end{sremarks}

For the proof of Theorem \ref{thm_gp_main},
we start  by dealing with the following special  case.

\begin{slemma} \label{lem_gp_main}
The map $H^1_{\substack{tame \\ loop}}(R_{r,n},J) \to H^1_{r-tame}(F_{n,r}, J)$
is  bijective.
\end{slemma}

\begin{proof}
For each  basic tame cover $R_{r,n,m} \otimes_k l$ of $R_{r,n}$, we have isomorphisms
\[
\xymatrix{
H^1\bigl( \mu_m^r(l) \rtimes \Gal(l/k), J(l) \bigr)   \ar[d]_\wr
 & \simlgr &
 H^1(R_{r,n,m} \otimes_k l/ R_{r,n}, J) \ar[d] \\
 H^1\bigl( \mu_m^r(l) \rtimes \Gal(l/k), J(F_{r,n,m}\otimes_k l)  \bigr)
   & \simlgr &
 H^1(F_{r,n,m} \otimes_k l/ F_{n,r}, J)
}
\]
so that the map  $H^1(R_{n,m} \otimes_k l/ R_{r,n}, J) \to H^1(F_{r,n,m} \otimes_k l/ F_{n,r}, J)$ is an isomorphism.
By taking the inductive limit on those covers
\cite{Mg} we get the wished statement.
\end{proof}

The proof of Theorem \ref{thm_gp_main}  goes by steps and uses crucially
the notion of reductibility and of isotropicity.
By using Theorem \ref{thm_FG} of the appendix, 
we extend verbatim \cite[lemma 7.12]{GP13}.

\begin{slemma} \label{lem_GP_special}
If $[\phi],[\phi'] \in H^1(\pi_1^t(R_{r,n},1), \widetilde G(k_s))$
have same image in $H^1(F_{n,r},\widetilde G)$, then 
$[\phi^{ar}]=[{\phi'}^{ar}] \in H^1(k, \widetilde G)$. \qed
\end{slemma}

\begin{sproposition} \label{prop_GP_iso} 
Let $\phi: \pi^t(R_{r,n},1) \to \widetilde G(k_s)$ be a tame loop cocycle.  Then the following are equivalent:

\smallskip

(i) $\phi$ is isotropic;

\smallskip

(ii) the $R_{r,n}$-reductive group scheme ${_\phi G}$
is isotropic; 

\smallskip

(iii)  $({_\phi G})_{K_{r,n}}$ is isotropic;

\smallskip

(iv)  $({_\phi G})_{F_{n,r}}$ is isotropic.

\end{sproposition}

\begin{proof} The implications $(i) \Longrightarrow (ii)
\Longrightarrow (iii) \Longrightarrow (iv)$ are obvious.
Let us show the implication 
$(iv) \Longrightarrow (i)$ by induction on $n$, the case $n=0$ being  obvious. We reason by sake of contradiction and 
 assume that $\phi$ is anisotropic.
We can deal with a  basic tame cover $R_{r,n,m} \otimes_k l$ of $R_{r,n}$
such that $G_l$ is split. We put $\Gamma= \mu_m(l)^r \rtimes
\Gal(l/k)$.

Without loss of generality we can assume that $\phi^{ar}=1$,
i.e. $C_G(\phi^{geo})^0$ is anisotropic.
According to Corollary  \ref{cor_PY}, 
$\phi$ is then irreducible and 
the torus $C_G(\phi^{geo})^0$ is anisotropic.

We want to establish that $({_\phi G})_{F_{n,r}}$ is anisotropic.
According to the Bruhat-Tits-Rousseau's theorem \cite[5.1.27]{BT2} 
 applied to the field $F_{n,r}=F_{r,n-1}((t_n))$, this 
rephrases to show that the extended Bruhat-Tits building $\cB_e\bigl( ({_\phi  G})_{F_{n,r}} \bigr)$ consists in one point.  
According to the tamely ramified descent theorem
\cite[Proposition 5.1.1]{Rou} (see also
\cite{P20}), we have
$$
\cB_e\bigl(({_\phi G})_{F_{n,r}} \bigr) = \cB_e\bigl(G_{F_{r,n,m} \otimes_k  l} \bigr)^{\Gamma_\phi}
$$
where the fixed points are taken with respect to the twisted Galois action relatively to $\phi$.
The right handside contains the center $c_e$ which is fixed so that we have to prove that  $ \{c_e\} =
\cB_e\bigl(G_{F_{r,n,m} \otimes_k l}\bigr)^{\Gamma_\phi}$.
In other words we have to prove that $\{c\} =\cB\bigl(G_{F_{r,n,m} \otimes_k l} \bigr)^{\Gamma_\phi}$ and that $0= E^{\mu_m(l)^r \Gal(l/k)}$ where $E= \widehat C_l \otimes_\ZZ \RR$.

\sm

\noindent{\it The toral part.}
We have $$
0=(\widehat C^0)^{\Gamma_\phi}= \Hom_{F_{n,r}}(  \GG_m,
{_\phi C}_{F_{n,r}})
$$
 so that $E^{\Gamma_\phi}= (\widehat {_\phi C}^0)_{F_{n,r}}
 \otimes_\ZZ \RR=0$.
\sm

\noindent{\it The semisimple part.} 
 According to Proposition \ref{prop_GP_red},
$({_\phi \,G})_{F_{n,r}}$ is irreducible so that 
$\cB( {_\phi G}_{F_{n,r}}) \simlgr  \cB\bigl(G_{F_{r,n,m} \otimes_k l} \bigr)^{\Gamma_\phi}$ consists in the point $c$.
\end{proof}

\begin{sremark}{\rm The former proof was not correct \cite[Corollary 7.2.(3)]{GP13}
since we implicitely used  the new Corollary \ref{cor_PY}.
}
\end{sremark}

\begin{sproposition} \label{prop_GP_pure} \cite[Theorem 7.9]{GP13} 
Let $\phi, \phi'$ be purely geometrical tame loop cocycles given by $\phi^{geo}, \phi'_{geo} : \mu_m^r \to \widetilde G$.
Assume that $\phi$ is anisotropic.
Then  the following are equivalent:

\sm

(i) $\phi^{geo}$ and ${\phi'}^{geo}$ are $\widetilde G(k)$-conjugated;

\sm

(ii) $[\phi]= [\phi'] \in H^1(R_{r,n},\widetilde G)$;

\sm

(iii)  $[\phi]= [\phi'] \in H^1(K_{r,n}, \widetilde G)$;

\sm

(iv)  $[\phi]= [\phi'] \in H^1(F_{n,r},\widetilde G)$.
 
\end{sproposition}

\begin{proof}
We prove the statement by induction on 
$n \geq 0$, the case $n=0$ being obvious.
The implications $(i) \Longrightarrow (ii) \Longrightarrow (iii)
\Longrightarrow (iv)$ are obvious. Let us  prove
the implication $(iv) \Longrightarrow (i)$.

We remark that it enough to deal with the case of $R_n$, that is, when $r=n$.
Indeed if $r<n$, we can extend $\phi^{geo}: (\mu_m)^r \to G$
to   $\phi^{geo, \sharp}: (\mu_m)^n \to G$ by precomposing 
by the projection  $(\mu_m)^n \to (\mu_n)^r$ (and similarly for $\phi'$). We have
$[\phi^{geo}]_{F_n} = [\phi^{geo, \sharp}]_{F_n} \in H^1(F_n,G)$.

We work at finite level with a basic 
tame cover $R_{n,m} \otimes_k l$ of $R_n$
such that $G_l$ is split.
Our assumption is that there exists 
$\widetilde g \in \widetilde G(F_{n,m} \otimes_k l)$
such that 
\begin{equation} \label{eq_cocycle}
\phi(\sigma)= \widetilde g^{-1} \, \phi'(\sigma)\,  \sigma(\widetilde g).
\end{equation}
for all $\sigma \in \Gamma=\Gal(R_{n,m} \otimes_k l/R_{n})= \mu_m(l)^n \rtimes \Gal(l/k)$.
The key step  is the following.

\begin{claim} \label{claim63}  $\widetilde g  \in \widetilde 
G\bigl( F_{n-1,m} \otimes_k k[[t_n^{\frac{1}{m}}]] \bigr)$.
 \end{claim}

 We consider the extended Bruhat-Tits building 
$\cB_{e,n}=\cB_e(G_{F_{n,m,l}})$. It comes with an action 
of $\widetilde G(F_{n,m,l}) \rtimes \Gal( F_{n,m} \otimes_k l/F_n)$ and 
we denote by  $c$ the hyperspecial point (which is sometimes called
the center of the building)
which is the unique point fixed by $(DG)^{sc}( F_{n-1,m} \otimes_k l[[t_n]])$ \cite[9.1.19.(c)]{BT1}.
According to Lemma \ref{lem_stab2},
${\widetilde G\bigl( F_{n-1,m} \otimes_k l[[t_n^{\frac{1}{m}}]] \bigr)}$ is the stabilizer of 
$c$ for the standard action of
 $\widetilde  G(F_{n,m} \otimes_k l)$ on
$\cB_{e,n}$ so that we have to prove that  
$\widetilde g . c=c$.
Denoting by $\star$ the twisted action of 
$\Gamma$ on $\cB_{e,n}$,
we have that 
\begin{equation}\label{eq_fix}
\bigl( \cB_{e,n} \bigr)^{\Gamma_\phi}=\{c\}.
\end{equation}

\noindent Indeed ${_\phi G}_{F_n}$ is anisotropic 
according to  the Bruhat-Tits-Rousseau's theorem \cite[5.1.27]{BT2}. Since $c$ belongs to $(\cB_{e,n})^{\Gamma_\phi}$,
it follows that $\cB_{e,n}^{\Gamma_\phi}=\{c\}$.
For each $\sigma \in \Gamma$,
we have

\vskip-4mm

\begin{eqnarray} \nonumber
\sigma  \star (  \widetilde g \, . \, c) & = &
 \phi(\sigma) \, \, \sigma(\widetilde g) \, . \, 
 \sigma(c) \\  \nonumber 
& = & \phi(\sigma) \, \,  \sigma(\widetilde g) \, . \, c \qquad 
\hbox{ [$c$ is invariant under $\tilde \Gamma_{m,n}$]}  \\  \nonumber 
& = &   \widetilde g \, . \, \phi'(\sigma)  \, \,  c \qquad 
\hbox{[relation \enskip \ref{eq_cocycle}]}  \\  \nonumber 
& = &   \widetilde g \, . \,  c \qquad 
\hbox{[$\phi(\gamma) \in \widetilde G(k_s)]$} .
\end{eqnarray}
Thus $\widetilde g \, . \, c=c$ 
so that $\widetilde g \in \widetilde G\bigl( F_{n-1,m} \otimes_k l[[t_n^{\frac{1}{m}}]] \bigr)$.
Since $\phi$ and $\phi'$ are purely geometrical,
the equation \eqref{eq_cocycle} implies that 
$\widetilde g$ is $\Gal(l/k)$-invariant.
Thus $\widetilde g \in \widetilde G\bigl( F_{n-1,m} \otimes_k k[[t_n^{\frac{1}{m}}]] \bigr)$ which establishes  the Claim \ref{claim63}.

We can then specialize the relation 
\eqref{eq_cocycle} with respect to 
the $\Gamma$-equivariant map 
map $\widetilde G\bigl( F_{n-1,m} \otimes_k k[[t_n^{\frac{1}{m}}]] \bigr) \to 
\widetilde G\bigl( F_{n-1,m} \bigr)$
and obtain
\begin{equation} \label{eq_cocycle1}
\phi(\sigma)= \widetilde g_n^{-1} \phi'(\sigma)\,  \sigma(\widetilde g_n) \qquad (\sigma \in \Gamma)
\end{equation}
with $\widetilde g_n \in  \widetilde G\bigl( F_{n-1,m}\bigr)$.
We consider the transporter
$$
X= \bigl\{ x \in  \widetilde G \mid \, \phi_n^{geo}= x^{-1} \, {\phi'}_n^{geo} \, x \bigr\}
$$
where $\phi_n^{geo}$ (resp.\ ${\phi'}_n^{geo})$ stands
for the restriction of $\phi^{geo}$ (resp.\ ${\phi'}^{geo})$) to the last factor $\mu_m$. Equation \eqref{eq_cocycle1} tells us that 
$X(F_{n-1,m}) \not = \emptyset$ so that $X$ is not empty.
It follows that $X$ is a $\widetilde G^{\phi_n^{geo}}$-torsor
and Theorem \ref{thm_FG} yields that $X(k) \not = \emptyset$.
Without loss of generality we can then assume that 
$\phi_n^{geo}={\phi'}_n^{geo}$.

We put $\widetilde H= \widetilde G^{\phi_n^{geo}}$
and $H= \bigl( G^{\phi_n^{geo}} \bigr)^0$.
Then $H$ is reductive and  $H= (\widetilde H)^0$.
The restriction  $\psi$ (resp.\ $\phi'$)
of $\phi$ to $\mu_m(l)^{n-1} \rtimes \Gal(l/k)$
of $\phi$ (resp.\ $\phi'$) take values in 
$\widetilde H$, we see them as tame loop cocycles for
$\widetilde H$.

\begin{claim} \label{claim39} $\psi$ is anisotropic as tame loop cocycle with value in 
$\widetilde H$.
\end{claim}

The $F_n$--group ${_\psi \!H} \times_k {F_{n}}$
is a subgroup of $_{\phi'}\! G \cong _{\phi}\! G$
which is anisotropic according to  Proposition \ref{prop_GP_iso} .
A fortiori ${_\psi \!H}$ is $F_{n-1}$-anisotropic
so that the same statement shows that 
$\psi$ is anisotropic as tame loop cocycle with value in 
$\widetilde H$. The Claim is established.

The equation \eqref{eq_cocycle1} applied to the element $(1,\dots, 1 , \zeta_n)$
shows that $\widetilde g_n \in \widetilde H(F_{n-1,m})$ so that  $[\psi]=[\psi'] \in H^1(F_{n-1}, \widetilde H)$,
the induction hypothesis shows that there exists $h\in \widetilde H(k)$
such that 
$\psi(\sigma) = \widetilde h^{-1} \, \psi'(\sigma) \,  \widetilde h$
for all $\sigma \in \mu_m(l)^{n-1} \rtimes \Gal(l/k)$.
Thus $\phi(\sigma) = \widetilde h^{-1} \, \phi'(\sigma) \,  \widetilde h$
for all $\sigma \in \mu_m(l)^{n} \rtimes \Gal(l/k)$.
\end{proof}

\begin{scorollary} \label{cor_GP_pure} \cite[Theorem 7.8]{GP13} 
Let $\phi, \phi'$ be  tame loop cocycles with values in $\widetilde G(k_s)$.
Assume that $\phi$ is anisotropic.
Then  the following are equivalent:

\sm

(i) There exists $\widetilde g \in \widetilde G(k_s)$ such that 
$\phi'(\sigma) = \widetilde g^{-1} \,  \phi( \sigma) \, \sigma(\widetilde g)$
for all $\sigma \in \pi_1^t(R_{r,n},1)$.

\sm

(ii)  $[\phi]= [\phi'] \in H^1(R_{r,n},\widetilde G)$;

\sm

(iii)  $[\phi]= [\phi'] \in H^1(K_{r,n},\widetilde G)$;

\sm

(iv)  $[\phi]= [\phi'] \in H^1(F_{n,r},\widetilde G)$.
 
 \sm
 
Furthermore under these assumptions we have
$[\phi^{ar}]=[{\phi'}^{ar}] \in H^1(k, \widetilde G)$.
 
\end{scorollary}

\begin{proof}
Once again the implications
$(i) \Longrightarrow (ii)  \Longrightarrow (iii)
 \Longrightarrow (iv)$ are obvious.
 We prove now $(iv) \Longrightarrow (i)$.
  Lemma \ref{lem_GP_special} shows that 
$[\phi^{ar}]=[{\phi'}^{ar}] \in H^1(k, \widetilde G)$
which is the last fact of the statement.
Without loss of generality we can assume 
that $\phi^{ar}={\phi'}^{ar}$. 
Up to twist $\widetilde G$ by  $\phi^{ar}$, the usual torsion 
argument boils down to the case $\phi^{ar}=1$.
This case is handled by Proposition  \ref{prop_GP_pure}, so we are done.
\end{proof}

We proceed now to the proof of Theorem \ref{thm_gp_main}.

\begin{proof}
We deal with the injectivity of the map
$H^1_{\substack{tame \\ loop}}(R_{r,n}, \widetilde G) \to H^1(F_{n,r}, \widetilde G)$, that is
to show that the fiber at the class $[\phi]$ of any
loop cocycle $\phi$ consists in one element.
If $\phi$ is an anisotropic tame loop cocycle, 
Corollary  \ref{cor_GP_pure} shows that the fiber
at $[\phi]$ of $H^1_{\substack{tame \\ loop}}(R_{r,n}, \widetilde G) \to H^1(F_{n,r}, \widetilde G)$
is $\{[\phi] \}$.

A first generalization is the irreducible case.
We assume then that the tame loop cocycle $\phi$
is irreducible.
One again we can assume that $\phi^{ar}=1$.
Let $\phi'$ be another tame loop cocycle
such that $[\phi']=[\phi] \in H^1(F_{r,n}, \widetilde G)$.
According to Lemma \ref{lem_gp_main}
the map  $H^1(R_{r,n}^{tsc}/ R_{r,n}, J) \to   H^1_{r-tame}(F_{r,n}, J)$ is bijective.
It follows that 
 $\phi$ and $\phi'$ have same image in  $H^1(R_{r,n},J)$. 
 Lemma \ref{lem_J} permits to assume without lost
 of generality that  $\phi$ and $\phi'$ have same image in 
 $Z^1(\pi_1^t(R_{r,n}),J(k_s)))$.
 We denote by $J_1$ the image of $\phi^{geo}: \widehat \ZZ'(1)^r \to J$,
 this is a finite smooth algebraic $k$--group of multiplicative type.
We put $\widetilde G_1= G_1 \times_J J_1$, by construction
$\phi$ and $\phi'$ have value in $\widetilde G_1(k_s)$. 
To avoid any confusion we denote them by $\phi_1$ and $\phi'_1$.
We consider the commutative diagram of pointed sets 
\[
\xymatrix{
(J/J_1)(k_s)^{\pi_1^t(R_{r,n},1)_{\phi_1}}
 \ar[r] \ar[d]^{=} & 
H^1(\pi_1^t(R_{r,n},1), \,  {_{\phi_1}\!G_1(k_s)})  \ar[r] \ar[d] &
H^1(\pi_1^t(R_{r,n},1), \, {_{\phi_1}\! \widetilde G(k_s)})  \ar[d] \\
\bigl({_{\phi_1}(J/J_1)}\bigr)(F_{r,n}) \ar[r] & H^1(F_{r,n}, \, {_{\phi_1}\! G_1})  \ar[r] & H^1(F_{r,n},\, {_{\phi_1}\!\widetilde G}).
}
\]
The second one is associated to the exact sequence 
$1 \to {_{\phi_1}\!(\widetilde G_1)} \to {_{\phi_1}\!(\widetilde G )} \to {_{\phi_1}\!(J/J_1)} \to 1$ of $F_{r,n}$-sheaves 
and the first one is associated to the exact sequence
of $\pi_1^t(R_{r,n},1)$-sets $1 \to {_{\phi_1}\!(\widetilde G_1(k_s))} \to 
{_\phi\!(\widetilde G(k_s)) \to _{\phi_1}\!(J/J_1)(k_s)} \to 1$.
By diagram chase involving the torsion bijection
$H^1(\pi_1^t(R_{r,n},1), {_{\phi_1}\!\widetilde G_1(k_s)})
\simlgr H^1(\pi_1^t(R_{r,n},1), \widetilde G_1(k_s))$, we see that we can arrange
 $\phi'_1$ in order
that $\phi'_1$ has same image than $\phi_1$ in $H^1(F_{r,n},\widetilde G_1)$.
We can work then with then $\widetilde G_1$
which is generated by $G$ and the image of $\phi_1^{geo}$. 

Since the $k$--torus $C$ is central in $G$,
the  $k$--subgroup $C^{\phi^{geo}}$ is central in $\widetilde G_1$.
We denote by $C_0$ the maximal split $k$-subtorus of $C^{\phi^{geo}}$ 
and consider the central exact sequence of algebraic $k$--groups
$$
1 \to C_0 \to \widetilde G_1 \to  \widetilde G_1/C_0 \to 1.
$$
The gain is that the image of $\phi_1$ in 
$Z^1( \pi_1^t(R_{r,n},1), (\widetilde G_1/C_0)(k_s))$
is an anisotropic tame loop cocycle by applying the criterion of 
Corollary \ref{cor_PY}.
Since $H^1(R_{r,n},C_0)=H^1(F_{r,n},C_0)=1$, we obtain the following commutative diagram
 \[
\xymatrix{
 H^1(R_{r,n}, \widetilde G_1)  \ar[d]  & \hookrightarrow & H^1(R_{r,n}, \widetilde G_1/C_0)  \ar[d] \\
 H^1(F_{r,n}, \widetilde G_1)   & \hookrightarrow & H^1(F_{r,n}, \widetilde G_1/C_0) 
}
\]
 where the horizontal maps are injections \cite[III.3.4.5.(iv)]{Gd}. 
Corollary \ref{cor_GP_pure} shows that $[\phi_1]$ and 
$[\phi'_1]$ have same image in $H^1(R_{r,n}, \widetilde G_1/C_0)$.
The diagram shows that $[\phi_1]= 
[\phi'_1] \in H^1(R_{r,n}, \widetilde G_1)$.
By pushing in  $H^1(R_{r,n}, \widetilde G)$
we get that $[\phi]=  [\phi'] \in H^1(R_{r,n}, \widetilde G)$ as desired.

\sm

We deal now with the general case. The above reduction (with $\widetilde G_1$)
permits to assume that $J$ is finite \'etale so that $\widetilde G$ is affine
and also that $\phi$ is purely geometric.
 Let $(P,L)$ be a pair normalized by $\phi^{geo}$ 
 where $P$ is  $k$--parabolic subgroup of $G$,
 $L$ is a $k$--Levi subgroup of $P$  
 and which is minimal for this property (with respect to the inclusion).
Then the tame loop cocycle $\phi$ takes value in $N_{\widetilde G}(P,L)(k_s)$.
We have  an exact  sequence \eqref{eq_JPL}
$$
1 \to L \to N_{\widetilde G}(P,L) \to J_P \to 1
$$ 
of smooth affine $k$--groups where $J_P \subset J$ is a clopen $k$--subgroup
and $J_P$ is a twisted constant (Lemma \ref{lem_JPL}).
 We denote by $\psi$ the image of $\phi$ in 
 $Z^1\bigl(\pi_1^t(R_{r,n},1), \, N_{\widetilde G}(P,L) (k_s)\bigr)$.
 Lemma \ref{lem_irr}.(2) states that $\psi$ is irreducible.

 We deal now with the tame loop cocycle $\phi'$ having
 same image in $H^1(F_{r,n}, \widetilde G)$ as $\phi$. 
 Lemma \ref{lem_GP_special} implies that 
 $\phi'$ is purely geometrical.
We consider the  projective
 $k$--variety  $X= \widetilde G/ N_{\widetilde G}(P)$ (Lemma \ref{lem_proj}). 
Theorem \ref{thm_GP_fix} shows that 
$({_\phi X})(F_{r,n}) \not = \emptyset$
so that $({_{\phi'} X})(F_{r,n}) \not = \emptyset$.
The same result shows that
 $X^{{\phi'}^{geo}}(k)$ are not empty.
We pick   $x' \in X^{{\phi'}^{geo}}(k)$ and
choose $x \in  X^{{\phi}^{geo}}(k)$ such that $G_x=P$.
 observe and choose 
$\widetilde g \in \widetilde G(k_s)$
such that $x'=\widetilde g \, . \, x$.
For $\sigma \in \Gal(k_s/k)$, we have that 
$x'=\sigma(\widetilde g) \, . \, x$ so that
$\sigma \to n_\sigma =\widetilde g^{-1} \, \sigma(\widetilde g)$
is a $1$-cocycle with value in $\widetilde G_x(k_s)=N_{\widetilde G}(P)(k_s)$.
For $\sigma \in \Gal(k_s/k)$, we have 
$\phi'(\sigma).x'=x'$ so that
$\phi'(\sigma).\widetilde g \, . \, x=\widetilde g \, . \, x$.
It follows that $\widetilde g^{-1} \phi'(\sigma) 
\widetilde g \, . \, x=x$. Since $n_\sigma\ x \, =x$
it follows that 
 $\phi''(\sigma)= \widetilde g^{-1} \, \phi'(\sigma)$
 fixes $x$.
\noindent Up to replace $\phi'$ by $\phi''$ we can then assume that 
 $\phi'$ takes value in $N_{\widetilde G}(P)(k_s)$.
Proposition \ref{prop_morozov} tells  us that 
$P$ admits a Levi subgroup $L'$ normalized by ${\phi'}^{geo}$.
Then $L'= \, {^g\!L}$ for some $g \in P(k)$.
Up to replace $\phi'$ by $^{g^{-1}}\phi'$, we
can then assume that $\phi'$ has value in $N_{\widetilde G}(P,L)$.

 We note that   $(P,L)$ is minimal for this property (otherwise it will not be minimal for $\phi$).
  The above argument tells us that the image $\psi'$ of $\phi'$ in 
 $Z^1\bigl(\pi_1^t(R_{r,n},1), N_{\widetilde G}(P,L)(k_s))$
 is irreducible. We consider now the commutative diagram

\[
\xymatrix{
 H^1(R_{r,n},\widetilde G)  \ar[r] & H^1(F_{r,n},\widetilde G)  \\ 
 H^1(R_{r,n},N_{\widetilde G}(P,L))_{irr} \ar[r] \ar[u]  & H^1(F_{r,n},N_{\widetilde G}(P,L))_{irr} \enskip. \incl[u]  
}
\]
The  second horizontal map is well-defined in view of Proposition \ref{prop_GP_red}.
 The   right vertical map is injective ({\it ibid}, lemme 4.2.1.(2)).
 We have seen that $\phi$, $\phi'$ define tame loop elements of 
 $H^1(R_{r,n},N_{\widetilde G}(P,L))_{irr}$ which give then the same image in
 $H^1(F_{r,n},N_{\widetilde G}(P))$. 
 Taking into account the already handled irreducible case,  diagram chasing
 enables us to conclude that 
 $[\phi]= [\phi'] \in H^1(R_{r,n},N_{\widetilde G}(P,L))$.
 Thus  $[\phi]= [\phi'] \in H^1(R_{r,n},\widetilde G)$ as desired.
\end{proof}

\begin{sremark} {\rm  The reduction involving $\widetilde G_1$ is unfortunately
missing in the original proof, i.e. \cite[Theorem 8.1]{GP13}.
}
\end{sremark}

The next variant is due to D.~Ofek.

\begin{scorollary} \label{cor_GP_pure2} 
Let $\phi, \phi'$ be  tame loop cocycles with values in $G(k_s)$
which are purely geometrical. The following assertions are equivalent: 

\sm

(i) There exists $g \in  G(k)$ 
and a split $k$-subtorus $S$ of $G$ such that
 $(\phi')^{geo}, \, ^{ g}\!\phi^{geo}$ centralize
 $S$ and
$$
\phi'(\sigma)= \, ^{ g}\!\phi(\sigma)  \, S(k_s)
\quad \hbox{for all $\sigma \in \pi_1^t(R_{r,n},1)$}.
$$

\sm

(ii)  $[\phi]= [\phi'] \in H^1(R_{r,n}, G)$;

\sm

(iii)  $[\phi]= [\phi'] \in H^1(K_{r,n}, G)$;

\sm

(iv)  $[\phi]= [\phi'] \in H^1(F_{r,n}, G)$.

\end{scorollary}

\begin{proof} 
We can reason at finite level by  considering
$\phi^{geo}, \, \phi'^{geo}: \mu_m^r \to G$
for some integer $m \in k^\times$
The implications $(ii) \Longrightarrow (iii)
\Longrightarrow (iv)$ are trivial. 

 \sm
 
 \noindent $(i) \Longrightarrow (ii)$. Without  loss of generality we can 
 assume that  $(\phi')^{geo},\,  \phi^{geo}$ centralize
 $S$ and that $\phi'(\sigma)= \phi(\sigma)  \, S(k_s)$
for all $\sigma \in \pi_1^t(R_{r,n},1)$.
This implies that  $[\phi],[\phi'] \in H^1(R_{r,n},C_{ G}(S))$
have same image in $H^1(R_{r,n},C_{G}(S)/S)$.
Since $1=H^1(R_{r,n},S)$ acts transitively on the fibers of 
$H^1(R_{r,n},C_{ G}(S)) \to H^1(R_{r,n},C_{G}(S)/S)$, it follows that $[\phi]= [\phi'] \in H^1(R_{r,n},C_{G}(S))$.
A fortiori we have $[\phi]= [\phi'] \in  H^1(R_{r,n},G)$.

 \sm
 
 \noindent $(iv) \Longrightarrow (i)$.
 Let $(P,L)$ be a pair consisting of a $k$-parabolic subgroup
 $P$ with a Levi subgroup  which is normalized by $\phi^{geo}$ and which 
 is minimal for this property. In view of Proposition \ref{prop_morozov},
$P$ is itself minimal among the  $k$-parabolic subgroup
 which are normalized by $\phi^{geo}$
We denote by  $Z$ the  $k$-scheme 
 of parabolic subgroups of $G$ of same type as $P$.
 
 Applying Theorem \ref{thm_GP_fix}, $(i) \Longrightarrow (ii)$  to the $k$-scheme $Z$ and to $\phi^{geo}$ 
shows that $({_\phi\!Z})(F_{n,r}) \not = \emptyset$.
Since $[\phi]= [\phi'] \in H^1(F_{r,n}, G)$, 
the $F_{n,r}$-schemes ${_\phi'\!Z}$ and ${_\phi'\!Z}$
are isomorphic so that 
$({_\phi'\!Z})(F_{n,r}) \not = \emptyset$.
Next
 Theorem \ref{thm_GP_fix}, $(v) \Longrightarrow (i)$, applied   to the $k$-scheme $Z$ and $(\phi')^{geo}$ yields  that $Z^{(\phi')^{geo}}(k) \not = \emptyset$. Since $G(k)$ acts transitively on $Z(k)$, up to conjugate $\phi'$ 
 by an element of $G(k)$, we can assume that $(\phi')^{geo}$
 normalizes $P$ as well.

Proposition \ref{prop_morozov} shows that  $\phi'^{geo}$
normalizes a $k$-Levi subgroup $L'$ of $P$.
Since $R_u(P)(k)$ acts (simply) transitively on the set of 
Levi subgroups of $P$, we can assume that $L=L'$.
We notice that  $\phi^{geo}$ and $(\phi')^{geo}$ factorize through  
$L= N_{G}(P,L)$. 
According to Proposition   helo applied to $L$,
we have that the reductive $F_{r,n}$-groups $_\phi\!L$ and $_\phi\!L$ 
are both irreducible. Next the the map 
$$
H^1(F_{r,n}, L) \to H^1(F_{r,n},G)
$$
is injective  by Borel-Tits' theorem. 
Since $[\phi]$, $[\phi']$ go the same class in 
$H^1(F_{r,n}, G)$, it follows that 
$[\phi]=[\phi']$ inside $H^1(F_{r,n}, L)$.
Let $S$  be the maximal split central subtorus of $L$,
we have $L=C_G(S)$ \cite[Exp.\ XXVI, Proposition 6.8]{SGA3}.
Since the  images of $[\phi]$, $[\phi']$ in $H^1(F_{r,n}, L/S)$
 are anisotropic, hello shows that there exists 
$[l] \in (L/S)(k)$ such that  $\phi'^{geo}=\! ^l\!\phi'^{geo}:
\mu_m^r \to L$ agree in $L/S$. 
 Thus $\phi'(\sigma)= \, ^{l}\!\phi(\sigma)  \, S(k_s)$
for all $\sigma \in \pi_1^t(R_{r,n},1)$.
\end{proof}

\section{Tame fundamental group \`a  la Grothendieck-Murre}
\label{sect_murre}

\subsection{Abhyankar's lemma} Let $X=\Spec(A)$ be a regular local  scheme (not assumed henselian at this stage).
Let $\gm$ be the maximal ideal  of $A$, $k=A/\gm$ be its residue field
and $p \geq 1$ be its characteristic exponent.
We put $\widehat \ZZ'= \prod\limits_{l \not =p} \ZZ_l$.
Let $K$ be the fraction field of $A$, 
and let $K_s$ be a separable closure of $K$.
It determines a base point $\xi: \Spec(K_s) \to X$
so that we can deal with the Grothendieck fundamental 
group $\Pi_1(X, \xi)$ \cite{SGA1}. 

We assume that $A$ is of dimension $d \geq 1$. 
Let $(f_1, \dots, f_d)$ be a regular sequence of $A$
and consider the divisor $D= \sum\limits_{i=1}^r D_i= \sum\limits_{i=1}^r \mathrm{div}(f_i)$ where $0 \leq r \leq d$;
$D$ has strict normal crossings.
We put $U = X \setminus  D= \Spec(A_D)$.

We recall that a  finite \'etale cover $V \to U$ is 
{\it tamely ramified} with respect to $D$ if the associated \'etale $K$--algebra
$L=L_1 \times \dots \times L_a$ is 
tamely ramified at the $D_i's$ for $i=1,..,r$, that is, for each $i=1,..,r$, there exists $j_i$ such that 
for the  Galois closure $\widetilde L_{j_i}/K$ of $L_{j_i}/K$,
the inertia group associated to $v_{D_i}$ has order prime to 
$p$ \cite[XIII.2.0]{SGA1}.

Grothendieck and Murre defined the  
tame ({\it mod\'er\'e} in French) fundamental group 
$\Pi_1^D(U, \xi)$
with respect to  $U \subset X$  \cite[XIII.2.1.3]{SGA1} and \cite[\S 2]{GM}. This is a profinite quotient of $\Pi_1(U, \xi)$ whose quotients by open subgroups
 provides finite connected tame  covers of $U$.
 
 \sm

From now on we assume that $A$ is henselian.
We consider a special class of Galois tamely ramified 
covers called {\it basic Abhyankar's tame covers}. For $n$ a positive prime to $p$ integer,
and  $B$ a Galois cover of $A$ containing a  primitive $n$--th root of unity, we define
$$
B_{r,n}=\Bigl( B[T_1,\dots, T_r, \dots, T_d]/ ( T_1^{n} - f_1, \dots, T_r^{n}- f_r) \Bigr)
$$
and 
$$
B_{D,n}= B_{r,n}\otimes_A A_D.
$$
Then $B_{D,n}$ is a tamely ramified Galois cover of 
$A_D$ and we have $$
\Gal(B_{D,n}/A_D)= \bigl( \prod\limits_{i=1}^r \mu_n(B) \bigr) \rtimes \Gal(B/A).
$$
We know that those  basic covers filter the universal
tame cover of $A_D$ \cite[XIII, \S 5.3.0]{SGA1}.
Passing to the limit we obtain  then an isomorphism
$$
\pi_1^t(U, \xi) \cong \bigl( \prod\limits_{i=1}^r \widehat \ZZ'(1) \bigr)
\rtimes \pi_1(X, \xi).
$$
We denote by  $f: U^{tsc} \to U$ the profinite \'etale cover
associated to the quotient $\pi_1^t(U, \xi)$ of $\pi_1(U,\xi)$.
According to \cite[Theorem 2.4.2]{GM}, it is the universal tamely ramified cover of $U$. 
It is a localization of the inductive limit $\widetilde B_D$ of the  $B_{D,n}$'s. On the other hand we consider  the inductive limit $\widetilde B$ of the $B$'s and observe that $\widetilde B_D$
is a  $\widetilde B$-ring.

\begin{remark}{ \rm The $A$--algebra $B_{D,n}$ is isomorphic to the algebra \break
 ${B[T_1^{\pm 1},\dots, T_r^{\pm 1}, \dots, T_d^{\pm 1}]/ ( T_1^{n} - f_1, \dots, T_r^{n}- f_r)}$
which is the form given   in \cite[\S 2.10]{Gi24}. 
}
\end{remark}

\begin{lemma}\label{lem_picard}
Let $V\to U=\Spec(A_D)$ be a connected
tamely ramified cover. Then $\Pic(V)=0$. 
\end{lemma}

\begin{proof} We use the existence of a basic Abhyankar cover
$\Spec(B_{D,n}) \to V=\Spec(C)$ for $B,n$ as above and denote by $\Gamma$ its Galois group.
Since $\Gal(B_{D,n}/A_D)$ acts naturally on $B_{r,n}$, 
we can consider the ring    $\widetilde C= (B_{r,n})^{\Gamma}$
and we have $C=(B_{D,n})^\Gamma=
(\widetilde C)_D$.
The ring $B_{r,n}$ is local and its maximal ideal
is minimally generated by $\sqrt[n]{f_1}, \dots ,\sqrt[n]{f_r}$, $f_{r+1}$, \dots, $f_d$.
Since $A$ is of Krull dimension $d$, so  is $B_{r,n}$ \cite[00OK]{St} which is then a regular local ring.
The ring $\widetilde C$ is connected and is a finite extension of the henselian local ring $A$, so is local henselian \cite[Proposition $_4$.18.5.9.(ii) and Proposition $_4$.18.5.10]{EGA4} and
is normal \cite[Proposition 6.4.1]{BH}. 
Since the map   the map $\Pic(\widetilde C)\to \Pic(C)$
is onto \cite[Corollary 11.43]{GW}, we conclude that $\Pic(C)=0$.
\end{proof}

\begin{remark}{\rm The ring $\widetilde C$ is normal
but not regular in general, for example \break
$\widetilde C=\CC[[x,y]][\sqrt{xy}]$ is not a regular ring.
When $A$ is strictly henselian, there is a general criterion in terms of the action of $\Gamma$ on  the tangent space at
the closed point due to Serre \cite[Theorem 1']{Se3}, see also \cite[Corollary 2.13]{Wa}. 
}
\end{remark}


\subsection{Blow-up}\label{subsec_blow_up}

We assume that $d \geq 1$.
We follow a blowing-up construction arising from  
 \cite[lemma 15.1.1.6]{EGA4}.
We denote by $\widehat X$ the blow-up of $X=\Spec(A)$ 
at its closed point, this is a regular scheme \cite[\S 8.1, 
Theorem 1.19]{Liu} and the exceptional divisor $E \subset \widehat X$ is 
a Cartier divisor isomorphic to $\mathbb{P}^{d-1}_k$.
We denote by  $R=\mathcal{O}_{\widehat X,\eta}$ the local ring at the generic point $\eta$ of $E$.  The ring $R$ is a DVR of
fraction field $K$ and of  residue field 
$F=k(E)=k(t_1, \dots, t_{d-1})$ where $t_i$ is the image
of $\frac{f_{i+1}}{f_1} \in R$ by the specialization map.

On the other hand we denote  by $v: K^\times \to \ZZ$ the discrete order valuation associated  to $R$ \cite[Theorem 6.7.8]{SH}. For $a \in A \setminus \{0\}$, it is defined by 
\begin{equation} \label{valuation}
v(a) = \mathrm{Max}\big\{ n  \in \ZZ_{\geq 0} \, \mid \, a \in \gm^n \big\}. 
\end{equation}  
Then $R$ is the valuation ring of $v$ according to the next considerations.
 We denote by 
$$
\mathrm{BL}_\gm(A)= A \oplus X \gm \oplus  X^2 \gm^2 \oplus \cdots \qquad
\subset \enskip A[X]
$$
the Rees (or blow-up) algebra with the notation
of \cite[Tag 052Q]{St} and consider the 
the affine blow-up algebra $$
B=\Bigl[ \frac{\gm}{f_1} \Bigr] = \bigl(\mathrm{BL}_\gm(A)\bigr)_{X f_1}.
$$
Then  the  $A$-map
$B  \to R$, $X f \mapsto \frac{f}{f_1}$,  induces
 an isomorphism $B_{f_1} \simlgr R$ \cite[Theorem 6.7.9]{SH}.

\begin{remark}\label{rem_blow}{\rm  
According to \cite[Corollary  5.5.9, proof of Theorem 6.7.9]{SH} (see also \cite[exercise 14.4]{Ma}), another description of $R$ is $S_{Sf_1}$
where \break
$S= A[T_1,\dots, T_{d-1}]/ (f_1 T_1-f_2, \dots, f_1 T_{d-1}-f_d)$.
Actually $S  \cong A[\frac{f_2}{f_1}, \dots \frac{f_d}{f_1}]$.
}
\end{remark}
 
Let $\gp$ be the maximal ideal of $R$, we have
$\gm \subseteq \gp$.
Since $\gp\cap A$ is a proper prime ideal of $A$, we
have  $\gp \cap A \subset \gm$ hence  $\gp \cap A = \gm$. 
We consider the ring $A_D^\sharp=A_D \cap R$.

\begin{slemma} \label{lem_AC}
(1) For each $n \geq 0$, we have $\gm^n \setminus \gm^{n+1} \, 
\subseteq  \, \gp^n \setminus \gp^{n+1}$
or in other words $v( \gm^n \setminus \gm^{n+1}) =n$.

\smallskip

\noindent (2) $\gm^n = \gp^n \cap A$ for each $n \geq 1$.

\smallskip

\noindent (3)  If $r \geq 1$, we have
$A\Bigl[  \bigl(\frac{f_2}{f_1}\bigr)^{\pm 1},  
\dots , \bigl(\frac{f_r}{f_1}\bigr)^{\pm 1} \Bigr] = A_D^\sharp$. 

\smallskip

\noindent (4) If $r \geq 1$, we have $A_D \cap \gp^n = (f_1)^n \, A_D^\sharp$ for each $n \geq 1$.

\smallskip

\noindent (5) The ring  $A_D^\sharp$ is regular.

\end{slemma}

\begin{proof}
(1)   and  (2)   follow readily of the formula \ref{valuation}.
\smallskip

\noindent (3) If $r=0$, the statement is obvious
so that we can assume that $r \geq 0$. The inclusion 
$A\Bigl[  \bigl(\frac{f_2}{f_1}\bigr)^{\pm 1},  
\dots , \bigl(\frac{f_r}{f_1}\bigr)^{\pm 1} \Bigr] \subseteq A_D^\sharp=  A_D \cap R$ is obvious and
we shall prove the converse one.
Let $x \in A_D^\sharp$.
We have  that $x= (f_1 f_2 \dots f_r)^{-m} a$ for $a\in A$
for some $m \geq 0$. 
We put $n=v(a)$. It follows that 
so that $-r m+ n=v(x) \geq 0$ hence
$n \geq rm$. 
Assertion (2)
shows that  $a \in \gm^n$, that is,  
$a= \sum\limits_{i_1+ \dots + i_r \geq n} a_{i_1,\dots, i_r}
 \, f_1^{i_1} \dots f_r^{i_r}$
 with $a_{i_1,\dots, i_r} \in A$ for each uple.
 It follows that 
 \begin{eqnarray*}
x&= &(f_1f_2 \dots f_r)^{-m} a ,\\
&=& \bigl(\frac{f_2}{f_1} \bigr)^r\dots 
\bigl(\frac{f_r}{f_1}\bigr)^r  \, f_1^{-mr} \,  a \\
&=& \bigl(\frac{f_2}{f_1} \bigr)^r \dots 
\bigl(\frac{f_r}{f_1}\bigr)^r  \, \,
 \sum\limits_{i_1+ \dots + i_r \geq n} a_{i_1,\dots, i_r}
 \, \, (f_1)^{-mr + n- (n-i_1- \dots - i_r)} \,(\frac{f_2}{f_1})^{i_1} \dots \bigl(\frac{f_r}{f_1}\bigr)^{i_r}.
\end{eqnarray*}
 Since $-mr + n\geq 0$, we conclude that   $x \in  A\Bigl[  \bigl(\frac{f_2}{f_1}\bigr)^{\pm 1},  
\dots , \bigl(\frac{f_r}{f_1}\bigr)^{\pm 1} \Bigr]$.

\smallskip

\noindent (4) 
Let $x \in A_D \cap \gp^n$ for $n \geq 1$.
The same argument as in (3) shows that 
$x \in (f_1)^n A\Bigl[  \bigl(\frac{f_2}{f_1}\bigr)^{\pm 1},  
\dots , \bigl(\frac{f_r}{f_1}\bigr)^{\pm 1} \Bigr]$.

\smallskip

\noindent (5) We want to show that $A_D^\sharp$ is regular.
If $r=0$,  we have that $A=A_D= A_D^\sharp$ 
so it os obvious.
If $r=1$, $A_D=A_{f_1}$ and $A=A_D^\sharp$, so this
case works as well. We assume that $r \geq 2$. 

Let $\gq$ be a maximal ideal  of $A_D^\sharp$
and consider firstly the case $f_1 \not \in \gq$.
Since $A_D=A_D^\sharp[ \frac{1}{f_1}]$;
and $A_D$ is regular,  $(A_D^\sharp)_\gq=  
(A_D^\sharp)_\gq$ is a regular local ring.
We assume now that $f_1 \in \gq$.
From (4), we have $A_D^\sharp \cap f_1 R =
f_1 A_D^\sharp$ so that the map $A_D^\sharp/f_1   
A_D^\sharp \to R/f_1 R$ is injective.
It follows that $f_1 A_D^\sharp$ is a prime ideal
of   $A_D^\sharp$. It is then enough to show that the ring
$(A_D^\sharp)_{f_1 A_D^\sharp}$ is regular.

We  have $A_D^\sharp/f_1 A_D^\sharp
\cong A/(f_1, \dots f_r)[T_1, \dots, T_{r-1}]$;
this ring is regular and so is its localization  
$(A_D^\sharp)_{f_1 A_D^\sharp}/ f_1 (A_D^\sharp)_{f_1 A_D^\sharp}$.  On the other hand, the ring 
$(A_D^\sharp)_{f_1 A_D^\sharp}$ is integral
so that  $f_1$ is not a divisor of $0$. 
According to \cite[VIII, S 3, Corollaire 1, 
(iii) $\Longrightarrow (i)$]{BAC}, 
we can conclude that the local ring
$(A_D^\sharp)_{f_1 A_D^\sharp}$ is regular.
\end{proof}
 
 We obtain then a factorization
\begin{equation}\label{eq_laurent1}
\xymatrix{
R   \ar[r] & k(t_1,\dots, t_{r-1}, t_r, \dots, t_{d-1}) \\
A_D \cap R  \ar@{->>}[r] \incl[u] & k[t_1^{\pm 1},\dots, t_{r-1}^{\pm 1} , t_r, \dots, t_{d-1}] . \incl[u]
}
\end{equation}
which relies with Laurent polynomials on the residue field.

\smallskip 

We recall the notation $A_D^\sharp=A_D \cap R$ and consider its prime ideal 
$\gn= \gP \cap A_D^\sharp$. We have $\gn= f_1 A_D^\sharp$ according
to Lemma \ref{lem_AC}.(4).


\sm

We deal now with a Galois extension $B_{D,n}$ of $A_D$ as above.
Since $B$ is a connected finite \'etale  cover of $A$,
$B$ is regular and  local; it is furthermore henselian \cite[Proposition $_4$.18.5.10]{EGA4}.
We denote by $L$ the fraction field of $B$ and by $L_{r,n}$
that of $B_{r,n}$. We have $[L_{r,n}:L]= n^r$ and
 want to extend the valuation $v$ to $L$ and to $L_{r,n}$.

We denote by $l=B/\gm_B$ the residue field of $B$, this is a finite Galois field  extension of $k$. Also $(T_1,\dots, T_r, f_{r+1}, \dots, f_d)$ is a system of parameters for $B_{r,n}$. We denote by $w:L^\times \to \ZZ$ 
 the discrete valuation associated  to the
exceptional divisor of the blow-up of $\Spec(B)$ at its closed point. Then $w$ extends $v$ and $L_w/K_v$ is an unramified
extension of degree $[L:K]$ and of residual extension 
$F_l= l(t_1, \dots , t_{d-1})/k(t_1, \dots , t_{d-1})$.

On the other hand we denote by $w_{r,n}:L_{r,n}^\times \to \ZZ$ 
 the discrete valuation associated  to the
exceptional divisor of the blow-up of $\Spec(B_{r,n})$ at its closed point. We 
have $l= B_{r,n}/ \gm_{B_{r,n}}$.
The valuation  $\frac{w_{r,n}}{n}$ on $L_{r,n}$ extends $w$ and 
its residual extension is 
$F_{l,n}= l\Bigl(  t_1^{1/n},\dots, t_{r-1}^{1/n} , t_r, \dots, t_{d-1}\Bigr) /k\bigl(t_1, \dots, t_{r-1}, t_r, \dots, t_{d-1} \bigr)$ 
so that ${[F_{l,n}:F_l]}=n^{r-1}$.
Furthermore the ramification index $e_{r,n}$ of $L_{r,n}/L$ is $\geq n$.
Since $n^r \leq e_{r,n} \, {[F_{l,n}:F_l]} \leq [L_{r,n}:K]=n^r$ (where the last inequality is \cite[\S VI.3, Proposition 2]{BAC}) it follows that $e_{r,n}=n$. 
The same statement shows that the map 
$L_w \otimes_L L_{r,n} \to L_{w_{r,n}}$ is an isomorphism.
To summarize $L_{w_{r,n}}/L_w$ is tamely ramified of ramification index $n$ and of degree $n^r$.
Altogether we have $L_{w_{r,n}}=L_w \otimes_K L_{r,n}$ so 
that $L_{w_{r,n}}$ is Galois over 
$K_v$ of group ${\prod\limits_{i=1}^r \mu_n(B) \rtimes \Gal(B/A)}
= {\prod\limits_{i=1}^r \mu_n(l) \rtimes \Gal(l/k)}$.

We denote by $\Delta$  the diagonal
embedding $\mu_n(l) \subset \prod\limits_{i=1}^r \mu_n(l)$. 
We put $L_{w_{r,n}}^{\Delta}= L_{r,n}^{\Delta(\mu_n(B))}$.
Since $T_1$ is an uniformizing parameter of $K_v$ and 
since $\Delta(\zeta) \, . \, T_1= \zeta. T_1$ for
each  $\zeta \in \mu_n(B)$, it follows that 
$(L_{w_{r,n}})^{\Delta}$ is the maximal unramified extension of
$L_{w_{r,n}}/K_v$. We have a factorization

\begin{equation}\label{eq_laurent2}
\xymatrix{
O_{w_n}   \ar[r] & l(t_1^{\frac{1}{n}},\dots, t_{r-1}^{\frac{1}{n}}, t_r, \dots, t_{d-1}) \\
B_{D,n}\cap O_{w_n}  \ar[r] \incl[u] & l[t_1^{\pm \frac{1}{n}},\dots, t_{r-1}^{\pm \frac{1}{n}}, t_r, \dots, t_{d-1} ] . \incl[u]
}
\end{equation}

\noindent 
We put $\underline{A_D}=k\bigl[t_1^{\pm 1},\dots, t_{r-1}^{\pm 1}
, t_r, \dots, t_{d-1}\bigr]$ and $\underline{B_{D,n}}=l\bigl[t_1^{\pm \frac{1}{n}},\dots, t_{r-1}^{\pm \frac{1}{n}} , t_r, \dots, t_{d-1} \bigr] $. An important point is the
equivariance of the above diagram for the action  of
$\Gal(B_{D,n}/A_D)$. In particular
it provides an exact sequence

\begin{equation}\label{eq_sequence_Galois}
\xymatrix{
1 \ar[r] & \mu_n(B) \ar[r]^{\Delta \qquad}  & \mu_n(B)^r \rtimes \Gal(B/A) \ar[r] & \Bigl( \mu_n(B)^r/\mu_n(B) \Bigr)\rtimes \Gal(B/A) \ar[r] & 1 \\
1 \ar[r] & \mu_n(B) \ar[r]\ar[u]^{=} & \Gal(B_{D,n}/A_D)  \ar[r] \ar[u]^\wr &  \Gal(\underline{B_{D,n}}/\underline{A_D})  \ar[u]^\wr
\ar[r] & 1 .
}
\end{equation}

\subsection{Tame loop cocycles and tame loop torsors}

Let $G$ be an $X$--group scheme locally of finite presentation.
A tame loop cocycle is an element of $Z^1\bigl(\pi_1^t(U), G(\widetilde B) \bigr)$
and it defines a Galois cocycle in $Z^1(\pi_1^t(U), G(U^{tsc}))$.
We denote by $Z^1_{\substack{tame \\ loop}}(\pi_1^t(U), G(U^{tsc}))$ the image of 
the map  $Z^1\bigl(\pi_1^t(U), G(\widetilde B) \bigr) \to Z^1(\pi_1^t(U), G(U^{tsc}))$ and by $H^1_{\substack{tame \\ loop}}(U, G)$ the image of 
the map  $$
Z^1\bigl(\pi_1^t(U), G(\widetilde B) \bigr) \to 
H^1(\pi_1^t(U), G(U^{tsc})) \to H^1(U, G).
$$
We say that a $G$-torsor $E$ over $U$ (resp.\ an  fppf sheaf $G$-torsor)
is a tame loop torsor if its class belongs to
$H^1_{\substack{tame \\ loop}}(U, G) \subset H^1(U, G)$.

A given class $\gamma \in H^1_{\substack{tame \\ loop}}(U, G)$ is represented
by a $1$--cocycle $\phi: \Gal(B_{D,n}/A_D) \to G(B)$ for some cover $B_{D,n}/A_D$ as above.
Its restriction $\phi^{ar}: \Gal(B'/A) \to G(B'_n)$ to the subgroup $\Gal(B'_n/A)$ of
$\Gal(B_{D,n}/A_D)$ is called the ``arithmetic part'' and the other restriction
$\phi^{geo}: \prod\limits_{i=1}^r \mu_n(B) \to \gG(B)$ is called the geometric part.
We observe in the sequel that $\phi^{geo}$ is  a $B$-group homomorphism.

Indeed for $\sigma \in \Gal(B/A)$ and $\tau \in 
\prod\limits_{i=1}^r \mu_n(B)$ the computation of \cite[page 16]{GP13}
shows that $\phi^{geo}(\sigma \tau \sigma^{-1})=
\phi^{ar}(\sigma) \, {^\sigma \!  \phi}(\tau)  \, 
\phi^{ar}(\sigma)^{-1}$ so that $\phi^{geo}$
 descends to a homomorphism of $A$-group schemes 
$\phi^{geo}: \mu_n^r \to {_{\phi^{ar}}\!G}$.  
Altogether this provides a parameterization of tame loop cocycles.

\begin{slemma} \label{lem_dico2}
Assume that the $A$--group scheme
$G$ is ind-quasi-affine and locally of finite presentation.

\sm

\noindent (1) For $B_{D,n}/A_D$ as above, the map $\phi \mapsto (\phi^{ar}, \phi^{geo})$
provides a bijection between $Z^1_{\substack{tame \\ loop}}\bigl( \Gal(B_{D,n}/A_D) ,  G(B) \bigr)  $ and the couples
$(z, \eta)$ where $z \in Z^1\bigl( \Gal(B/A) ,  G(B) )$ and 
$\eta: \prod\limits_{i=1}^r \mu_n \to {_zG}$
is an $A$--group homomorphism.

\smallskip

\noindent (2) The map $\phi \mapsto (\phi^{ar}, \phi^{geo})$
provides a bijection between $Z^1_{\substack{tame \\ loop}}\bigl( \pi^1(U, \xi)^t ,  G(\widetilde B)\bigr)$ and the couples
$(z, \eta)$ where $z \in Z^1\bigl( \pi^1(X, \xi) , G(\widetilde B) )$ and $\eta: \prod\limits_{i=1}^r \widehat \ZZ' \to {_zG}$ is an $A$--group homomorphism.

\smallskip

\noindent (3) We have an exact sequence of pointed sets
$$
1 \to \Hom_{A-gr}(  \mu_n^r,G )/G(A) \to H^1\bigl( \Gal(B_{D,n}/A) ,  G(B) )
\xrightarrow{Res}  H^1\bigl( \Gal(B/A) ,  G(B) )
\to 1
$$ 
and the first map is injective.

\smallskip

\noindent (4) We have a decomposition
$$
H^1\bigl( \Gal(B_{D,n}/A_D) ,  G(B) ) = \coprod\limits_{[z] \in H^1(A,G)}
 \Hom_{A-gr}(  \mu_n^r, {_zG} )/{_zG}(A).
$$ 

\noindent (5) Assume furthermore that $G$ is smooth.
Then the map $H^1(A,G) \to H^1(k,G)$ is bijective and
in the following commutative diagram

\begin{equation} \label{diag14}
\xymatrix{
H^1\bigl( \Gal(B_{D,n}/A_D) ,  G(B) ) \ar[d]& = & 
\coprod\limits_{[z] \in H^1(A,G)}
 \Hom_{A-gr}(  \mu_n^r, {_zG} )/{_zG}(A) \ar[d] \\
 H^1\bigl( \Gal(B_{D,n}/A_D) ,  G(l) ) &=&
 \coprod\limits_{[z] \in H^1(k,G)}
 \Hom_{k-gr}(  \mu_n^r, {_zG} )/{_zG}(k),
}
\end{equation}
then the vertical maps are  bijective.
\end{slemma}

The further  Lemma \ref{lem_hensel} deals with  variants
of Lemma \ref{lem_dico2}.(5).

\begin{proof} The items (1) to (4) are similar with Lemma \ref{lem_dico}. 
For (5), we recall that  $H^1(A,G) \to H^1(k,G)$ is bijective
\cite[XXIV, Proposition 8.1]{SGA3} (and \cite[Remarque 11.8.3 ]{Gr3}).
 It is enough to establish the bijectivity of the right 
vertical map. For each $z \in Z^1(A^{sh}/A,G)$, we consider
the $A$-functor $\cF=\underline{\Hom}_{gr}(  \mu_n^r, {_zG}  )$.

\begin{claim} $\cF$ is representable by an ind-quasi-affine $A$-scheme $X$ which is locally of finite presentation.
\end{claim}

If $\mu_{n,A} \cong (\ZZ/n\ZZ)_A$, then $\cF$ is representable 
by a closed $A$-subscheme $X$  of $({_zG})^r$.
Since $G^r$ is ind-quasi-affine over $A$, 
so is $X$ in view of \cite[Tag 0F1W]{St}.
The general case follows by descent\footnote{Again we use
that ind-quasi-affine schemes satisfies fpqc descent \cite[Tag  0APK]{St}.} with respect to the finite \'etale cover $A[\mu_n]$ of $A$. Then $\cF$ is representable
by an ind-quasi-affine $A$-scheme and is locally of finite presentation since it commutes with direct limits of $A$-rings.

According to Grothendieck \cite[XI, Proposition 2.1]{SGA3}, 
the functor   $\cF$  is  formally smooth, hence
$X$ is a smooth $A$-scheme.
Since 
The Hensel lemma implies then the surjectivity of
$X(A)=\cF(A) \to X(k)=\cF(k)$.

For establishing the injectivity we are given $u_1,u_2 \in \cF(A)$ whose image in $\cF(k)$ are $G(k)$-conjugated.
We consider the relevant transporter $A$-sheaf $\cT$
defined by 
$$
 \cT(T)= \bigl\{ g \in G(T) \, \mid \, g \circ u_{1,T}=g_{2,T} \bigr\}
$$
 for each $A$-scheme $T$. The arguments of the above Claim
 show that $\cT$ is representable by an ind-quasi-affine  $A$-scheme $Y$ locally of finite presentation. According to 
 \cite[XI, Corollaire 2.4]{SGA3}, $\cT$ is formally smooth.
 It follows that the $A$--scheme $Y$ is smooth. 
 The Hensel's lemma states that the map $Y(A) \to Y(k)$ is onto. Since $Y(k)$ is not empty, $Y(A)$ is not empty.
 We conclude that $u_1$ and $u_2$ are $G(A)$-conjugated.
\end{proof}

\begin{sremark}\label{rem_functoriality}{\rm
The functoriality of the above correspondence is as follows.

\sm

\noindent (a) For the base change $A_D \to A'_D$ where $A'$ is a finite \'etale cover, 
we  have $B \otimes_A A'\cong (B')^c$ where $B'$ is finite Galois over $A'$.
Then the map $\Gal(B'/A') \hookrightarrow \Gal(B/A)$ induces
 a map  $\Gal(B'_{D,n}/A'_D) \hookrightarrow \Gal(B_{D,n}/A_D)$
 and the associated map \break $H^1(\Gal(B_{D,n}/A_D), G(B)) \to H^1(\Gal(B'_{D,n}/A'_D), G(B'))$ maps $[(z, \phi)]$ to $[(z', \phi_{A'})]$ where $z'$ is the image of $z$
 by the map $Z^1(\Gal(B/A), G(B)) \to H^1(\Gal(B'/A), G(B'))$.

\sm

\noindent (b)  For the base change $A_D \to A_{D,c}$ for $c$ prime to $p$,
the natural map $\Gal(B_{D,nc}/A_{c, D})  \to \Gal(B_{D,nc}/A_{D}) \to  \Gal(B_{D,n}/A_D)$ is induced by $\mu_{n}^r \hookrightarrow \mu_{nc}^r \xrightarrow{\times c}  \mu_{n}^r$.
It follows that the associated map $H^1(\Gal(B_{D,n}/A_{D}), G(B)) \to H^1(\Gal(B_{D,nc}/A_{D,c}), G(B_c))$ maps $[(z, \phi)]$ to $[(z_c, \phi_c)]$
where $z_c$ is the image of $z$
 by the map $Z^1(\Gal(B/A), G(B)) \to H^1(\Gal(B_c/A_c), G(B_c))$
 and $\phi_c= c \times \phi: \mu_n^r \to {_zG}$.

\sm

\noindent (c) In particular, if $n$ divides $c$, the image of
$[(z,\phi])$ in $H^1(\Gal(B_c/A_c), G(B_c))/ H^1(\Gal(B/A), G(B))$ is trivial.

}
\end{sremark}
We examine more closely the case of  a finite \'etale $X$--group scheme  $\gF$ of constant degree $m$.

\begin{slemma}\label{lem_finite} (1)  $\gF(\widetilde B)= \gF(X^{tsc})= \gF( U^{tsc})$.

\sm

\noindent (2) We assume that $m$ is prime to $p$.
We have $H^1_{\substack{tame \\ loop}}( U, \gF)= H^1(U, \gF)$.

\smallskip

\noindent (3) We assume that $m$ is prime to $p$.
Let $f: \gF \to \gH$ be a homomorphism of 
$A$--group schemes (locally of finite type).
Then $f_*\Bigl( H^1(U, \gF) \Bigr) \subset 
H^1_{\substack{tame \\ loop}}( U, \gH)$.
\end{slemma}

\begin{proof} 
\noindent (1) We are given a cover $B_{D,n}/A_D$  as above  such that $\gF_{B_{D,n}} \cong  \Gamma_{B_{D,n}}$ is finite constant.
 Since $B$ and $B_{D,n}$ are connected, the map $\gF(B) \to \gF(B_{D,n})$ reads
as the identity $\Gamma \cong \gF(B) \to \gF(B_{D,n}) \cong \Gamma$ so is bijective.
By passing to the limit we get  $\gF(\widetilde B)= \gF( U^{tsc})$.

\sm

\noindent (2) Let $\gE$ be a $\gF$--torsor over $U$. 
This is a finite \'etale $U$--scheme.
Since $U$ is noetherian and connected, 
we have a decomposition $\gE= V_1 \coprod \cdots \coprod V_l$ where each 
$V_i$ is a connected finite \'etale $U$--scheme of constant degree $m_i$.
We have $m_1+ \dots +m_l=d$ so that we can assume that $m_1$ is prime to $p$.
We have then $\gE(V_1) \not = \emptyset$. 

It follows that $f_1: V_1 \to U$ is a finite \'etale cover so that 
there exists a  factorization $ U^{tsc} \to V_1 \xrightarrow{h} U$ of $f$
so that $\gE(U^{stc} ) \not = \emptyset$. Therefore $[\gE]$ arises
from $H^1(\pi_1^t(U, \xi), \gF(U^{tsc})) \subset H^1(U, \gF)$.
It follows that $H^1(\pi_1^t(U, \xi), \gF(U^{tsc})) \simlgr H^1(U, \gF)$.
We use now (1) and obtain the desired bijection  $H^1(\pi_1^t(U, \xi), \gF(B)) \simlgr H^1(U, \gF)$.

\sm

\noindent (3) This follows readily from (2).
\end{proof}

A variant is to work with $B_n$ instead $B$, it does not provide new objects (under reasonable assumptions).

\begin{slemma} \label{lem_hensel}
Let $G$ be a smooth  $A$--group scheme.
Let $n$ be a prime to $p$ positive integer
such that $\mu_{n,B} \cong (\ZZ/n\ZZ)_B$
 and put $l_{r,n}=B_{r,n} \otimes_B l$.  We consider the commutative diagram 

\begin{equation}\label{diag_bij} 
\xymatrix{
H^1(\Gal(B_{D,n}/A_D), G(B)) 
 \ar[d]^\nu \ar[r]^{\alpha} 
& H^1(\Gal(B_{D,n}/A_D), G(B_{r,n})) \ar[d]^\gamma \\
 H^1(\Gal(B_{D,n}/A_D), G(l)) 
 \ar[r]^{\beta} & H^1(\Gal(B_{D,n}/A_D), G(l_{r,n}))
}
\end{equation}
induced by commutative diagram of $\Gal(B_{D,n}/A_D)$-maps
\[\xymatrix{
B \ar[d]\ar[r] & B_{r,n} \ar[d] \\
 l  \ar[r] & l_{r,n}.
}
\]

\smallskip

\noindent (1) The map $\beta$ is bijective.

\smallskip

\noindent (2) Suppose that
 the formal fibers of $A$ are geometrically regular  (e.g.\ $A$ is quasi-excellent), then all maps
of \eqref{diag_bij} are  bijective.

\end{slemma}

The proof involves two auxiliary lemmas which slightly generalize \cite[Lemma A.1]{Gi18}.

\begin{slemma}\label{lem_limit0} Let $\Theta$ be an abstract group and let $H= \limproj_{i \geq 0} H_i$ be an inverse limit of $\Theta$-groups with surjective transition maps. Then we have a natural bijection
$$
H^1(\Theta, H) \simlgr  \limproj_{i \geq 0} H^1(\Theta, H_i).
$$ 
\end{slemma}

\begin{proof}
For the injectivity, the usual twisting argument 
boils down to establish the triviality of the kernel.
It goes by a classical argument of 
successive approximations (e.g.\ \cite[Lemma 6.3.2]{GS} in the commutative case).
Let $E$ be a $H$-torsor such that $[E]$ belongs to the kernel
of 
$H^1(\Theta, H) \simlgr  \limproj_{i \geq 0} H^1(\Theta, H_i)$. We denote by $E_i=E \wedge^{H} H_i$ its
change of group with respect to the homomorphism $H \to H_i$.
Then the map $E \to \limproj_i E_i$ is an isomorphism so that 
$E^\Theta \simlgr  \limproj_i E_i^\Theta$.
Each $E_i$ is a trivial $H$-torsor so that 
 $E_{i+1}^\Theta \to E_i^\Theta$ is onto for each $i \geq 1$.
Therefore $E^\Theta$ is non empty so $E$ is a trivial $H$-torsor.

We deal now with the surjectivity and we shall use that
the map $Z^1(\Theta, H) \to \limproj_{i \geq 0} 
Z^1(\Theta, H_i)$ is bijective. 
 We are given a coherent 
system system of classes $\gamma=(\gamma_i)$
with $\gamma_i \in H^1(\Theta, H_i)$.
For each $i$ we denote by $\Upsilon_i \subseteq Z^1(\Theta, H_i)$ the set of $1$-cocycles representing $\alpha_i$.
Since $\Upsilon_i$ is homogeneous for the action of 
$H_i$, the transition map  $\Upsilon_{i+1} \to \Upsilon_i$
is onto. It follows that $\limproj_{i \geq 0}  \Upsilon_i$ is not empty. Picking an element in the limit provides
a $1$-cocycle in $Z^1(\Theta,H)$ whose class maps to 
$\gamma_i \in H^1(\Theta,H_i)$ for each $i \geq 0$. 
\end{proof}

\begin{slemma}\label{lem_limit}
Let $F$ be a field equipped with an action of a finite group
$\Theta$, that is, a homomorphism $\theta: \Theta \to \Aut(F)$.
We assume that the order of $\ker(\theta)$ is coprime  
to the characteristic exponent of $F$.

\sm

\noindent (1) Let $V$ be a $F$-vector space equipped with an action 
of $\Theta$  which is semilinear with respect with $\theta$.
Then $H^j(\Theta, V)=0$ for each $j \geq 1$.

\sm

\noindent (2)
Let $U$ be a $\Theta$--group equipped with a $\Theta$-equivariant composition serie
$$
U=U_0 \supset U_1 \supset U_2 \supset \dots  
$$
satisfying the following assumptions

\sm

(i) $U \simlgr \limproj_{i \geq 0} U/U_i$;

\sm

(ii) Each $U_n/U_{i+1}$ admits a structure of $F$-vector space
such that the action of $\Theta$ is semilinear with respect with $\theta$.

\sm

\noindent Then  $H^1(\Theta, U)=1$.

\end{slemma}

\begin{proof} We denote by $\Theta_0$ the image of $\theta$
and put $F_0=F^{\Theta_0}$. Then $F$ is a Galois extension of 
$F_0$ with Galois group $\Theta_0$  \cite[V, \S 6, Theorem 3]{BA}. and we have an exact sequence
$1 \to \ker(\theta) \to \Theta \to \Theta_0 \to 1$.

\sm

\noindent (1)  We use the Hochschild-Serre spectral sequence
$H^p(\Theta_0, H^q(\ker(\theta), V)) \Longrightarrow H^{p+q}(\Theta, V)$.
Since $\sharp(\ker(\theta))$ is coprime to  the characteristic exponent 
of $F$, we have $H^j( \ker(\theta), V)=0$ for each $j \geq 1$.
The spectral sequence provides then isomorphisms
$H^j( \Theta_0, V^{\ker(\theta)}) \simlgr H^j( \Theta, V)$
for each $j \geq 1$. On the other hand, We observe that $V^{\ker(\theta)}$ is a $F$-vector space equipped with 
a semilinear action of $\Theta_0$; Galois descent states that 
there is a  canonical $\Theta_0$-isomorphism  $V_0 \otimes_{F_0} F \simlgr V^{\ker(\theta)} $
where $V_0=\bigl( V^{\ker(\theta)} \bigr)^{\Theta_0}=V^\Theta$ \cite[V, \S 4]{BA}.
It follows that 
$H^j( \Theta_0, V^{\ker(\theta)})=  H^j( \Theta_0, V_0 \otimes_{F_0} F)$.

The additive form of Hilbert 90 theorem \cite[X, \S 1, Proposition 1]{Se2} yields that
$H^j(\Theta_0, V_0 \otimes_{F_0} F)=0$ for $j \geq 1$. We conclude that 
$0=H^j( \Theta_0, V^{\ker(\theta)})=H^j(\Theta,V)$ for $j \geq 1$ as desired.
 
 \sm

\noindent (2)  By d\'evissage of the first case, we have $H^1(\Theta, U/U_i)=1$
for all $i \geq 0$. The vanishing of  $H^1(\Theta, U)=1$ follows then of Lemma \ref{lem_limit0}.
\end{proof}

We proceed now to the proof of Lemma \ref{lem_hensel}.

\begin{proof} We put $\Gamma=\Gal(B_{D,n}/A_D)=\mu_n(B)^r \rtimes \Gal(B/A)$.

\smallskip

\noindent (1) 
Since the map $l \to l_{r,n}$ admits a $\Gamma$-equivariant splitting, the 
bottom map  \break $\beta: H^1(\Gamma, G(l)) \to  H^1(\Gamma, G(l_{r,n}))$ is split injective.
For establishing the bijectivity, the usual twisting trick reduces to see that the kernel of $H^1(\Gamma, ({_\phi G)}(l_{r,n})) \to H^1(\Gamma, {_\phi G}(l))$ is trivial
for each cocycle $\phi \in Z^1(\Gamma, G(B))$. 

Since $G$ is smooth, according to \cite[Proposition A.5.12]{CGP}, we have an exact sequence of $\Gamma$-groups
$$
1 \to \Lie(G)(l)^d \to  G(l_{r,n}) \to G(l) \to 1.
$$
Furthermore the action of $\Gamma$ on
 $\Lie(G)^r(l)$ is semilinear with respect to the projection $\Gamma \to \Gal(l/k)$.
Lemma \ref{lem_limit}.(1) shows that 
$H^1\bigl(\Gamma, \Lie(G)^r(l)\bigr)=1$ so that 
the map $H^1(\Gamma, G(l_{r,n}))  \to H^1(\Gamma,G(l))$
has trivial kernel.

\smallskip

\noindent (2)  We deal first with the case $A$  complete
(and   so is $B$). We have  $B = \limproj_{i \geq 0} B/ \gm^{i+1}$ so that we have a bijection $G(B) \simlgr \limproj_{i \geq 0} \, G( B/ \gm^{i+1})$ 
\cite[XV, Lemme 1.4]{SGA3}.  We keep in mind that the transition maps of the projective system are projective granted to the Hensel lemma.

\begin{claim} The map 
$H^1(\Gamma, G( B/ \gm^{i+1})) \to 
H^1(\Gamma, G( B/ \gm^{i}))$ is bijective for each $i \geq 1$.
\end{claim}

Lemma \ref{lem_limit0} implies then that the map $\nu$
is bijective. For showing the Claim, we fix an integer $i \geq 1$ and consider the exact sequence of $\Gamma$-groups 
\cite[Proposition A.5.12]{CGP}
 $$
 0 \to \Lie(G)^n(l) \to G(B/\gm^{i+1}) \to  G(B/\gm^{i}) \to 0
 $$
Again the action of $\Gamma$ on $\Lie(G)^n(l)$ is by 
semilinear automorphisms with respect to the projection
$\Gamma \to \Gal(l/k)$.
Lemma \ref{lem_limit}.(1) shows that for each 
$\phi \in Z^1(\Gamma, G(B/\gm^{i})$, we 
have $H^1(\Gamma, {_\phi(\Lie(G)^n(l)})=0$ and 
$H^2(\Gamma, {_\phi(\Lie(G)^n(l)})=0$. 

This implies the surjectivity part of the Claim in view of  \cite[I, \S 5.6, proposition 41]{Se1}, the injectivity 
resulting again of the twisting argument ({\it ibid}, \S 5.4).

For establising the bijectivity of the map $\gamma$, it is enough to prove that the map $\gamma^\sharp:
H^1(\Gamma, G(B_{r,n})) \to H^1(\Gamma, G(l_{r,n}))$
is bijective since the bottom map $\beta$ is bijective.
The argument is the same as before by using 
that $B_{r,n}= \limproj_{i \geq 0} B_{r,n}/ \gn^{i+1}$
where $\gn= (\sqrt[n]{f_1}, \dots, \sqrt[n]{f_r}, f_{r+1},  \dots, f_d)$ is the maximal ideal of   $B_{r,n}$.

The case when $A$ is complete is then done. It remains to deal with general case 
by considering the completion $\widehat A$ of $A$. Then 
$(f_1,\dots, f_n)$ is a regular system of parameters
of the noetherian regular local ring $\widehat A$.
We shall use an Artin-Popescu's argument 
from Bouthier-\v{C}esnavi\v{c}ius \cite[Lemma 2.1.3]{BC}.
We consider the $A$-functor $\cF(R)= H^1(\Gamma, G( R \otimes_B))$ 
on $A$-algebras, it clearly commutes with direct limits.
The quoted result states that 
$\cF(A)$ injects in $\cF(\widehat A)$ and
that the maps  $\cF(A) \to \cF(k)$ 
and $\cF(\widehat A) \to \cF(k)$
have same image.
Since $\cF(\widehat A)  \simlgr \cF(k)$,
we conclude that $\cF(A) \simlgr \cF(k)$; in other words
$\nu$ is bijective.

The same argument shows that $\gamma^\sharp:
H^1(\Gamma, G(B_{r,n}) \to H^1(\Gamma, G(l_{r,n}))$
is bijective.  
\end{proof}

\subsection{Specializing purely geometric tame loop cocycles}

Let  $\phi: \Gal(B_{D,n}/A_D) \to G(B)$ be a  purely geometrical
loop cocycle and consider the underlying 
$A$--homomorphism $\phi^{geo}: (\mu_{n,A})^r \to G$
and its restriction to the diagonal $\mu_n$.
We consider the $A$--group scheme
${G_{\Delta_\phi}=G^{\phi^{geo}(\mu_n)}/\phi^{geo}(\mu_n)}$ 
which is a central quotient of 
the fixed point locus of the diagonal $\mu_n$.
Then $\phi^{geo}$ induces a homomorphism $(\mu_{n,A})^r/\mu_{n,A} \to  G_{\Delta_\phi}$.
By taking the closed fiber we get a 
  homomorphism $\mu_{n,k}^r/\mu_{n,k} \to G_{\Delta_\phi,k}$
 whence a purely geometrical tame loop cocycle 
$\underline{\phi}:  \Gal\bigl( \underline{B_{D,n}} /\underline{A_D}\bigr)
\to G_{\Delta_\phi}(l)$. We call it the specialization 
of $\phi$.

\begin{sremark}{\rm This notion is quite ad-hoc since
$G_{\Delta_\phi}$ depends highly of $\phi$.}
\end{sremark}

\begin{slemma} \label{lem_open} 
The monomorphism $G^{\phi^{geo}}/\phi^{geo}(\mu_n) \to
 G_{\Delta_\phi}= G^{\phi^{geo}(\mu_n)}/\phi^{geo}(\mu_n)$
 induces an open immersion of $A$--group schemes
$G^{\phi^{geo}}/\phi^{geo}(\mu_n)
 \to \bigl(G_{\Delta_\phi}\bigr)^{\phi^{geo}}$.
\end{slemma}

\begin{proof} To show that the induced map 
$h: G^{\phi^{geo}}/\phi^{geo}(\mu_n)
 \to \bigl(G_{\Delta_\phi}\bigr)^{\phi^{geo}}$
is an open immersion, we can reason fiberwise according to \cite[Corollaire $_4$.17.9.5]{EGA4}. We are reduced to the case of
an algebraically closed $A$-field $E$ and since the
two $E$--groups  are smooth, it is enough to 
show that the (injective) map 

\begin{equation} \label{eq_lie}
\Lie(h)(E): \Lie\Bigl( G^{\phi^{geo}}/\phi^{geo}(\mu_n)\Bigr)(E) \to  \Lie\Bigl( \bigl(G_{\Delta_\phi}\bigr)^{\phi^{geo}}\Bigr)(E) 
\end{equation}
is an isomorphism. Since $\mu_n$ is \'etale, 
we have $\Lie\Bigl( G^{\phi^{geo}}\Bigr) 
\simlgr \Lie\Bigl( G^{\phi^{geo}}/\phi^{geo}(\mu_n)\Bigr)$.
On the other hand, we have 
$$
\Lie\Bigl( \bigl(G_{\Delta_\phi}\bigr)^{\phi^{geo}}\Bigr)(E)
\simlgr  
\Bigl(\Lie\bigl(G_{\Delta_\phi}\bigr)(E)\Bigr)^{\phi^{geo}}
$$
according to \cite[A.8.10.(1)]{CGP}.
Similarly we have 
$\Lie\Bigl( G^{\phi^{geo}(\mu_n)}\Bigr) 
\simlgr \Lie \bigl(G_{\Delta_\phi}\bigr)$ so 
altogether we are then reduced to show 
(for showing the bijectivity of \eqref{eq_lie}) that the map 
$$
 \Lie\Bigl( G^{\phi^{geo}}\Bigr)(E) \to  \Bigl(\Lie \bigl(G\bigr)(E)\Bigr)^{\phi^{geo}} 
$$
is an isomorphism. This last fact follows again from 
\cite[A.8.10.(1)]{CGP} applied to the smooth $E$--group 
scheme $G^{\phi^{geo}(\mu_n)}$.
\end{proof}

\subsection{Twisting by tame loop torsors}
We assume that the $A$--group scheme $G$ acts 
on an $A$--scheme $Z$.
Let $\phi: (\prod\limits_{i=1}^r \mu_n)(B) \rtimes \Gal(B/A) \to G(B)$
be a tame loop cocycle.
It gives rise to  an $A$--action  of
 $\mu_n^r$ on $_{\phi^{ar}}\!Z$. 
We denote by $(_{\phi^{ar}}\!Z)^{\phi^{geo}}$
the fixed point locus for this action, it is representable
by a closed $A$--subscheme of $_{\phi^{ar}}\!Z$ \cite[A.8.10.(1)]{CGP}. We have a closed embedding  $(_{\phi^{ar}}\!Z)^{\phi^{geo}} \times_X U  \subset {_\phi Z}$ of
$U$-schemes.

Assume that furthermore that $\phi$ is purely geometrical.
Then $\phi$ induces an action of  $\mu_n^r/\mu_n$ 
on $Z_{\delta_\phi}:=Z^{\phi^{geo}(\mu_n)}$.
We have also an action of  $G_{\Delta_\phi}$
on $Z_{\delta_\phi}$.
By taking the closed fiber we obtain an action 
of $G_{\Delta_\phi, k}$ on $Z_{\delta_\phi,k}$
so we can twist $Z_{\delta_\phi,k}$ by $\underline{\phi}$
for obtaining the $\underline{A_D}$--scheme
${_{\underline \phi}\!\bigl(Z_{\delta_\phi,k} \times_{k} \underline{A_D}\bigr)}$.

\section{Fixed points method}
  
We slightly refine  \cite[Theorem 3.1]{Gi24}.

\begin{stheorem}\label{thm_fix}
Let $X=\Spec(A)$ be a henselian regular local scheme  as above.
We denote by $v: K^\times \to \ZZ$
the discrete valuation associated  to the
exceptional divisor $E$ of the blow-up of $X$ at its closed point.

Let $G$ be an $A$-group scheme  locally of finite presentation
acting on a projective smooth $A$--scheme $Z$.
Let $\phi$ be a tame loop cocycle for $G$. Then 
$Y=  \bigl( _{\phi^{ar}}Z\bigr)^{\phi^{geo}}$ is  a
smooth projective $A$--scheme and  
the following are equivalent:

\sm

(i) $Y(A) \not \not = \emptyset$;

\sm

(i') $Y(k) \not  = \emptyset$;

\sm 

(ii)  $Y(A_D) \not \not = \emptyset$;

\sm 

(ii')  $Y(\underline{A_D}) \not  = \emptyset$.

\sm 

(iii) $({_\phi Z})(A_D) \not = \emptyset$;

\sm 

(iv) $({_\phi Z})(K_v) \not = \emptyset$;

\sm

If furthermore  $\phi$ is purely geometrical this is also 
equivalent to the two next assertions

\sm 

(iii') $(_{\underline \phi}\!Z_{\delta_\phi,k})( \underline{A_D}) \not = \emptyset$;

\sm 

(iv') $(_{\underline \phi}\!Z_{\delta_\phi,k})\bigl( k(t_1, \dots, t_{r-1}, t_r, \dots, t_{d-1}) \bigr)\not = \emptyset$.

\end{stheorem}

Once again the projectivity
of $Z$ is used to insure  that  
the twisted fppf  $A_D$--sheaf ${_\phi Z}$ by Galois descent is representable according to \cite[\S 6.2]{BLR} as well as other twists.
The generalization to the proper smooth case is similar
with Remark \ref{rem_proj}.

\begin{proof}  The smoothness of $Y$ follows 
of \cite[A.8.10.(2)]{CGP} so that the henselian lemma provides
the equivalence $(i) \Longleftrightarrow (i')$.
For the other implications we can assume without loss
of generality that $\phi$ is purely geometrical.
 We claim that we have the following implications
\[\xymatrix{
(i) \ar@{<=>}[d]& \Longrightarrow & (ii) & \Longrightarrow &
(iii) & \Longrightarrow & (iv) \ar@{:>}[d]\\
(i') & \Longleftrightarrow & (ii') & \Longleftrightarrow & (iii')
& \Longleftrightarrow & (iv').
}\]
The first horizontal line is obvious
and the second one is Theorem \ref{thm_GP_fix}
applied to the action on $G_{\Delta_\phi}$
on $Z_{\delta_\phi}$ by noticing
that $Y_k= ( Z_{\delta_\phi,k})^{\underline{\phi}_{geo}}$.
It is then enough to prove the implication 
$(iv) \Longrightarrow (iv')$.

Let  $\phi: \Gal(B_{D,n}/A_D) \to G(B)$ be a  tame loop $1$-cocycle for some Galois cover $B_{D,n}/A_D$ as above for  some $n$ prime to $p$.
 We assume that 
$({_\phi Z})(K_v) \not = \emptyset$.
By definition we have
$$
({_\phi Z})(K_v) = \bigl\{ z \in Z(L_{w_n})
\, \mid \, \phi(\sigma). \sigma(z)=z  \enskip
\forall \sigma  \in \Gal(L_n/K) \bigr\}
$$
and our assumption is that this set is non-empty.
Let $O_{w_n}$ be the valuation ring of 
$Z(L_{w_n})$. Since $Z$ is projective over $X$, 
we have a specialization map
$Z(L_{w_n}) = Z(\cO_{w_n}) \to Z_k\bigl(l(t_1^{\frac{1}{n}},
\dots t_{r-1}^{\frac{1}{n}}, t_r, \dots , t_{d-1} ) \bigr)$. 
We get that the set 
$$
\Bigl\{ z \in Z_k\bigl(k(t_1^{\frac{1}{n}},
\dots t_{r-1}^{\frac{1}{n}} ) , t_r, \dots , t_{d-1}\bigr) \mid \, \phi(\sigma). \sigma(z)=z  \enskip
\forall \sigma  \in \Gal(L_{w_n}/K_v) \cong  \mu_n(l)^r \rtimes \Gal(l/k) \Bigr\}
$$
is not empty.  This set is nothing but 
$(_{\underline \phi}\!Z_{\delta_\phi,k})\bigl( k(t_1, \dots, t_{r-1}, t_r, \dots, t_{d-1}) \bigr)$ so we win.
\end{proof}

\subsection{Irreducibility and anisotropicity}

Let $G$ be a reductive $A$--group scheme
and consider an exact sequence 
$1 \to G \to \widetilde G \to J \to 1$
where $J$ is a twisted constant 
$A$--group scheme.

\begin{slemma} \label{lem_irr2} 
Let $\phi$ be a purely geometric tame loop cocycle for $\widetilde G$.
Let $(P,L)$ be a pair normalized by $\phi^{geo}$
where $P$  is  an $A$--parabolic subgroup of 
$G$  and $L$ is a  Levi $A$--subgroup of $P$.
We assume that $(P,L)$ is minimal for this property (with respect to the inclusion). Then the tame loop cocycle $\phi$ takes value in $N_{\widetilde G}(P,L)(\widetilde B)$ and it is irreducible seen as tame loop cocycle 
for $N_{\widetilde G}(P,L)$.
\end{slemma}

\begin{proof} We put $\widetilde L= N_{\widetilde G}(P,L)$.
The assumption implies that the tame loop cocycle $\phi$ takes value in $\widetilde L(k_s)$. We assume that  the image of $\phi$
in $Z^1( \pi^t(X,x), \widetilde L(k_s))$ is  reducible, that is, there
exists a pair $(Q,M)$ normalized by $\phi^{geo}$ such that 
$Q$ is a proper $k$--parabolic subgroup of $L=(\widetilde L)^0$
and $M$ a Levi subgroup of $Q$.
We have a Levi decomposition $P=U \rtimes L$
and remind to the reader that $P'=U \rtimes Q$
is a  $k$--parabolic subgroup of $G$ satisfying
$P' \subsetneq P$ \cite[Proposition 4.4.c]{BoT65}. 
Also  $M$ is a Levi subgroup of $P'$
 normalized by  $\phi^{geo}$ 
contradicting the minimality of $(P,L)$.
\end{proof}

We observe that  $J(\widetilde B)=J(X^{tsc})$  so that
\begin{equation}\label{eq_J2}
\qquad H^1\bigl(  \pi_1^t( X ,x) ,  J(\widetilde B) \bigr) 
= H^1\bigl(  \pi_1^t( X ,x) ,  J(X^{tsc}) \bigr)=
\ker\bigl( H^1(X,J) \to  H^1(X^{tsc},J) \bigr)
\end{equation}

\noindent in view of \cite[Corollary 2.9.2]{Gi15}. 

\begin{slemma} \label{lem_J2}
 Let $\phi,\phi'$ be two tame loop cocycles with value
in $\widetilde G(\widetilde B)$ having same image in 
$H^1(A_D,J)$. Then there exists $\widetilde g \in \widetilde G(\widetilde B)$
such that $\phi$ and $\sigma \mapsto \widetilde g^{-1} \, \phi' \, \sigma(\widetilde g) $ have same image in 
$Z^1\bigl(  \pi_1^t( X ,x) ,  J(\widetilde B) \bigr)$.
\end{slemma}

\begin{proof}
According to the fact \eqref{eq_J2},  the  tame loop cocycles $\phi$, $\phi'$
have same image in $Z^1\bigl(  \pi_1^t( X ,x) ,  J(\widetilde B) \bigr)$.
Since $\widetilde G \to  J$ is smooth,
the Hensel lemma shows that  
$\widetilde G(B) \to  J(l)$ for each 
finite connected \'etale cover $B$ of $A$
with residue field $k$. By taking the limit
in the Galois tower, we obtain that  $\widetilde G(\widetilde B)$ maps onto $J(\widetilde B)$.
It follows that  there exists $\widetilde g \in \widetilde G(\widetilde B)$
such that $\phi$ and $\sigma \mapsto \widetilde g^{-1} \, \phi' \, \sigma(\widetilde g) $ have same image in 
$Z^1\bigl(  \pi_1^t( X ,x) ,  J(\widetilde B) \bigr)$.
\end{proof}

In the same spirit that Lemma \ref{lem_GP_special}, 
we have the next fact.

\begin{slemma} \label{lem_special}
If $[\phi],[\phi'] \in H^1(\pi_1^t(X,x), \widetilde G(\widetilde B))$
have same image in $H^1(K_v,\widetilde G)$, then 
$[\phi^{ar}]=[{\phi'}^{ar}] \in H^1(A, \widetilde G)$.
\end{slemma}

\begin{proof} 
After base change to $A_{n}= A[T_1,\dots, T_r]/ ( T_1^{n} - f_1, \dots, T_r^{n}- f_r)$ for $n >>0$,
 $\phi$ and $\phi'$ become  purely arithmetic.
Since  $H^1(A, \widetilde G)= H^1(A_n, \widetilde G)$
($A$ and $A_n$ are henselian local rings with same residue field), we can deal without loss of generality
with $[\phi], [\phi'] \in  H^1(\Gal(B/A), \widetilde G(B))$ for some finite Galois cover $B/A$
having same image  in $H^1(K_v, \widetilde G)$.
The map $H^1(O_v, \widetilde G) \to H^1(K_v, \widetilde G)$ is injective
(Theorem \ref{thm_Nis} in the appendix) so that  $[\phi], [\phi']$ have already same image
in $H^1(O_v, \widetilde G)$.
We consider the following diagram
\[\xymatrix{
H^1(\Gal(B/A), \widetilde G(B)) \ar[d]^{\wr} \ar[r] & H^1(O_v, \widetilde G)
\ar[d]^{\wr} \\
H^1(\Gal(l/k), \widetilde G_k(l)) \ar[r] & H^1\bigl(k(t_1,\dots, t_{d-1}), \widetilde G_k \bigr) 
}
\]
where the vertical maps are avatar of the Hensel lemma \cite[XXIV.8.1]{SGA3}.
Since $H^1(k,\widetilde G_k) 
 \to H^1\bigl(k((t_1))\dots, ((t_{d-1})), \widetilde G_k \bigr)$ is injective (Theorem \ref{thm_FG} in the appendix),
a fortiori  the bottom horizontal map is injective. Thus the horizontal 
map is injective. We conclude that  $[\phi]= [\phi'] \in H^1(\Gal(B/A), \widetilde G(B))$ as desired.
\end{proof}

We say that a tame loop cocycle $\phi: \Gal(B_{D,n}/A_D) \to \widetilde G(B)$
is {\it  reducible} if the $A$-homomorphism $\phi^{geo}: \mu_{n}^r 
\to  {_{\phi^{ar}}\!{\widetilde G}}$ is reducible, that is,  
normalizes  a pair $(P,L)$ where $P$ is a  proper parabolic $A$--subgroup of  $_{\phi^{ar}}\!{G}$ together with $L$ a Levi subgroup  of $P$.

\begin{sproposition}\label{prop_red} 
(1) The following assertions are equivalent:

\sm

(i) The tame loop cocycle $\phi$ is reducible;

\sm

(ii) The $A_D$--group ${_\phi G}$ is reducible;

\sm

(iii) The $K_v$--group $({_\phi G})_{K_v}$ is reducible.

\sm

If furthermore $\phi$ is purely geometrical, 
the above assertions imply   the following assertion:

\sm

(iv) The $k$--group $(G^{\phi^{geo}})_k$ is isotropic.

\end{sproposition} 

\begin{proof}
Once again we can assume that $\phi$ is purely geometrical.
 We apply Theorem \ref{thm_fix} to the 
$A$-schemes  $Z=\mathrm{Par}^+(G)$
 and $Y=Z^{\phi^{geo}}$. We obtain the equivalence between the following statements:

 $(i_0)$  $Y(A) \not \not = \emptyset$;

\sm

$(ii_0)$ $({_\phi Z})(A_D) \not = \emptyset$;

\sm 

$(iii_0)$ $({_\phi Z})(K_v) \not = \emptyset$;

\sm
 
 $(iv_0)$  $Y(k) \not \not = \emptyset$.

\sm

\noindent Using Proposition \ref{prop_morozov}, we have
the equivalences
 $(i_0) \Longleftrightarrow (i)$, 
$(ii_0) \Longleftrightarrow (ii)$ 
and $(iii_0) \Longleftrightarrow (iii)$.
It remains only to  show the implication $(iv_0) \Longrightarrow (iv)$. 
We assume then $(iv_0)$, i.e.\
there exists then a proper $k$--parabolic  
$P$ of $G_k$ which is normalized by $\phi^{geo}$.
According to Proposition \ref{prop_morozov}, there exists
a homomorphism $\lambda: \GG_{m,k} \to (G^{\phi^{geo}})_k$
such that $P=P_{G_k}(\lambda)$.
Since $P$ is a proper subgroup of $G$, $\lambda$
is non-trivial. Thus $(G^{\phi^{geo}})_k$ is isotropic.
\end{proof}

We say that the tame loop cocycle $\phi: \Gal(B_{D,n}/A_D) \to \widetilde G(B)$
is {\it isotropic} if the $A$-homomorphism $\phi^{geo}: \mu_{n}^r 
\to {_{\phi^{ar}}\!{\widetilde G}}$ is isotropic, that is,  centralizes  a non-trivial  split  $A$--subtorus of 
$_{\phi^{ar}}{G}$. Equivalently  the reductive $A$--group
$C_{_{\phi^{ar}}\!G}(\phi^{geo})^0$ is isotropic.

\begin{scorollary} \label{cor_PY2} The following are equivalent:
\sm

(i) $\phi$ is isotropic;

\sm

(ii) $\phi$ is reducible or the torus $({_{\phi^{ar}}C}^{\phi^{geo}})^0$ is isotropic;

\end{scorollary}

\begin{proof}  This is similar with
that of Corollary \ref{cor_PY}.
\end{proof}

\begin{sproposition}\label{prop_iso} The following assertions are equivalent:

\sm

(i) The tame loop cocycle $\phi$ is isotropic;

\sm

(ii) The $A_D$--group ${_\phi G}$ is isotropic;

\sm

(iii) The $K_v$--group $({_\phi G})_{K_v}$ is isotropic;

\sm


If furthermore $\phi$ is purely geometrical, the above statements are equivalent to the next ones:

\sm

(i') The tame loop cocycle  $\underline{\phi}$
(for $G_{\Delta_\phi,k}$) is isotropic;

\sm

(ii') The $\underline{A_D}$--group ${_{\underline \phi}\!(G_{\Delta_\phi,k} \times_k \underline{A_D})}$ is isotropic.

\sm

(iii') The $k(t_1,\dots, t_{d-1})$--group  $_{\underline \phi}\!(G_{\Delta_\phi, k} \times_k k(t_1,\dots, t_{d-1}))$ is isotropic.

\end{sproposition}

\begin{proof} Once again we can 
assume that $\phi$ is purely geometrical. 
the equivalences $(i') \Longleftrightarrow (ii')
\Longleftrightarrow (iii')$ follow  from
Proposition \ref{prop_GP_iso} , $(i) \Longleftrightarrow (ii) \Longleftrightarrow (iii)$,
applied to the $k$--group $G_{\Delta_\phi, k}$
and the  ring 
 $k[t_1^{\pm 1}, \dots, t_{r-1}^{\pm 1}, t_r, \dots, t_{d-1} ]=\underline{A_D}$.

\sm 

\noindent $(i) \Longrightarrow (i')$.
If $\phi$ is isotropic there exists a non trivial
homomorphism $\lambda: \GG_{m,A} \to G^{\phi^{geo}}$
and induces a non trivial homomorphism
$$
\GG_{m,A} \to G^{\phi^{geo}} /\phi^{geo}(\mu_n)
\, \subset \, (G_{\Delta_\phi})^{\phi^{geo}}. 
$$
Its base change to $k$ is non-trivial \cite[IX.6.5]{SGA3}
whence a non trivial homomorphism
$$
\GG_{m,k}  \to (G_{\Delta_\phi,k})^{\underline{\phi}^{geo}}
$$
so that $\underline{\phi}$ is isotropic.

\sm

\noindent{\it Summarizing.} 
On the other hand 
we have the obvious implications
$  (i)    \Longrightarrow  (ii)  \Longrightarrow (iii)$ so this fits in the diagram
\[\xymatrix{
(i) \ar@{=>}[d] & \Longrightarrow & (ii)  & \Longrightarrow & (iii) \\
(i') & \Longleftrightarrow & (ii') & \Longleftrightarrow & (iii')
& .
}\]
It  is then  enough to establish the implications $(iii) \Longrightarrow (i')\Longrightarrow (i)$.
  
\sm

\noindent $(i') \Longrightarrow (i)$. 
We assume that tame loop cocycle  $\underline{\phi}$ is isotropic
and want to show that  $\phi$ is isotropic as  well.
There exists a monomorphism
$\lambda_k: \GG_{m,k} \to 
 (G_{\Delta_\phi,k})^{\phi^{geo}}$. 
According to Grothendieck's smoothness theorem \cite[XI.4.2]{SGA3}, it lifts to  a monomorphism
$\lambda: \GG_{m,A} \to  (G_{\Delta_\phi})^{\phi^{geo}}$. 
Since $G^{\phi^{geo}}/\phi^{geo}(\mu_n)$ is open in $(G_{\Delta_\phi})^{\phi^{geo}}$
(Lemma \ref{lem_open}), $\lambda$ factorizes through 
$G^{\phi^{geo}}/\phi^{geo}(\mu_n)$.
Then $\lambda^n$ lifts to a non-trivial
homomorphism $\GG_{m,A} \to G^{\phi^{geo}}$.
Thus  $\phi$ is isotropic.

\sm 

\noindent $(iii) \Longrightarrow (i')$.
We assume that $({_\phi G})_{K_v}$ is isotropic,
so that $({_\phi G})_{K_v}$  is reducible or its radical 
is isotropic.
If  $({_\phi G})_{K_v}$  is reducible, Proposition \ref{prop_red},  $(i) \Longrightarrow (iv)$,
shows that $(G^{\phi^{geo}})_k$ is isotropic.
A fortiori $\underline{\phi}$ is isotropic.

The last case is  when  the radical torus $Q$ of  
$({_\phi G})_{K_v}$  is isotropic.
Denoting by  $C \subset G$ the radical torus of $G$, 
we have $Q= {_\phi C}$.
Our assumption is that 
$$
\Hom_{K_v-gp}( \GG_m, Q)= \Hom_{L_{w_n}-gp}( \GG_m, C)^{\Gal(B_{D,n}/A_D)} 
$$
is not zero where the action is via the twisted action $\phi$.
 Since we have isomorphisms 
\[
\xymatrix{
\Hom_{ l(t_1^{1/n},\dots, t_{r-1}^{1/n}, t_r, \dots, t_{d-1})-gp}( \GG_m, C)^{\Gal(B_{D,n}/A_D)}   &   \Hom_{O_{w_n}-gp}( \GG_m, C)^{\Gal(B_{D,n}/A_D)} \ar[l]^{\qquad \quad \sim} 
\ar[d]^\wr \\
 &\Hom_{L_{w_n}-gp}( \GG_m, C)^{\Gal(B_{D,n}/A_D)}, 
}
\]
 we obtain then a central non-trivial homomorphism
$\theta: \GG_{m,k} \to G^{\phi^{geo}}_k$.
As in the proof of $(i) \Longrightarrow (i')$,
we conclude that 
 $\underline{\phi}$ is isotropic. 
\end{proof}

The next statement proceeds by analogy with Proposition \ref{prop_GP_pure}.

\begin{sproposition} \label{prop_pure} 
Let $\phi, \phi'$ be purely geometrical
loop cocycles given by \break 
$\phi^{geo}, \phi'_{geo} : \mu_n^r \to \widetilde G$. Assume that $\phi$ is anisotropic.
Then  the following are equivalent:

\sm

(i) $\phi^{geo}$ and ${\phi'}^{geo}$ are $\widetilde G(A)$-conjugated;

\sm

(ii) $[\phi]= [\phi'] \in H^1(A_D,\widetilde G)$;

\sm

(iii)  $[\phi]= [\phi'] \in H^1(K_v, \widetilde G)$;
 
\end{sproposition}

\begin{proof} 
The following  implications 
$(i)  \Longrightarrow  (ii)  \Longrightarrow  (iii)$
are obvious.
It remains to prove the implication
$(iii) \Longrightarrow (i)$.
Without loss of generality we can assume that $r=n$.
We work at finite level with a basic tame cover $B_{D,n}$ of $A_D$ such that $G_l$ is split.
Our assumption is that there exists 
$\widetilde g \in \widetilde G(L_w)$
such that 
\begin{equation} \label{eq_cocycle2}
\phi(\sigma)= \widetilde g^{-1} \phi'(\sigma)\,  \sigma(\widetilde g).
\end{equation}
for all $\sigma \in \Gamma=\Gal(L_w/K_v)=
\mu_n(B)^d \rtimes \Gal(B/A)$. The key step 
is the following.

\begin{claim} \label{claim_cobord} $\widetilde g  \in \widetilde G(O_v)$.
 \end{claim}
 
We consider the extended Bruhat-Tits building 
$\cB=\cB_e(G_{L_w})$. It comes with an action 
of $\widetilde G(L_w) \rtimes \Gamma$ and with the hyperspecial point $c$
which is the unique point fixed by $(DG)^{sc}(O_w)$ \cite[9.1.19.(c)]{BT1}.
According to Lemma \ref{lem_stab2},
${\widetilde G\bigl( O_w \bigr)}$ is the stabilizer of 
$c$ for the standard action of
 $\widetilde  G(L_w)$ on
$\cB$ so that we have to prove that  
$\widetilde g \, . \, c=c$. 
Denoting by $\star$ the twisted action by $\phi$ of 
$\Gamma$ on $\cB$,
we have that 
\begin{equation}\label{eq_fix2}
\cB^{\Gamma_\phi}=\{c\}.
\end{equation}

\noindent Indeed ${_\phi G}_{K_v}$ is anisotropic 
according to  Proposition \ref{prop_iso},
$(iii) \Longrightarrow (i)$, that $(\cB)^{\Gamma_\phi}$ 
consists in one point according to the Bruhat-Tits-Rousseau's theorem \cite[5.1.27]{BT2}. Since $c$ belongs to $\cB^{\Gamma_\phi}$,
it follows that $\cB^{\Gamma_\phi}=\{c\}$.
For each $\sigma \in \Gamma$,
we have

\vskip-4mm

\begin{eqnarray} \nonumber
\sigma  \star (  \widetilde g \, . \, c) & = &
 \phi(\sigma) \, \, \sigma(\widetilde g) \, . \, 
 \sigma(c) \\  \nonumber 
& = & \phi(\sigma) \, \,  \sigma(\widetilde g) \, . \, c \qquad 
\hbox{ [$c$ is invariant under $\Gamma$]}  \\  \nonumber 
& = &   \widetilde g \, . \, \phi'(\sigma)  \, \,  c \qquad 
\hbox{[relation \enskip \ref{eq_cocycle2}]}  \\  \nonumber 
& = &   \widetilde g \, . \,  c \qquad 
\hbox{[$\phi(\gamma) \in \widetilde G(B) \subset \widetilde G(O_w)]$} .
\end{eqnarray}
Thus $\widetilde g \, . \, c=c$ 
so that $\widetilde g \in \widetilde G(O_w)$.
We use now the specialization map
$\widetilde G( O_w) \to \widetilde G\bigl( 
 l(t_1^{1/n}, \dots, t_{d-1}^{1/n})  \bigr)$.
We consider the composite $\psi: \Gamma \xrightarrow{\phi} G(B) \to G(l)$
 and similarly for $\psi'$. This is a 1-cocycle and Claim \ref{claim_cobord}
 yields
 \begin{equation} \label{eq_cocycle3}
\psi(\sigma)= \widetilde g_0^{-1} \psi'(\sigma)\,  \sigma(\widetilde g_0)
\end{equation}
for all $\sigma \in \Gamma=
\mu_n(l)^d \rtimes \Gal(l/k)$ with $\widetilde g_0 \in 
\widetilde G\bigl(l(t_1^{1/n}, \dots, t_{r-1}^{ 1/n}, t_r, \dots, t_{d-1}) \bigr)$.
Since $\psi$ and $\psi'$ are trivial on $\Gal(l/k)$,
the above equation shows that  $\widetilde g_0 \in 
\widetilde G\bigl(k(t_1^{1/n}, \dots, t_{d-1}^{ 1/n}) \bigr)$.
Next we consider the transporter 
$$
X = \bigl\{ h \in {\widetilde G}_k\, \mid \, h \, \psi^{geo}(\sigma)
\,  h^{-1}= \psi'^{geo}(\sigma) \quad  \forall \sigma \in 
\Delta(\mu_n(l))  \bigr\}.
$$
Since the diagonal $\mu_n(l)$ acts trivially on 
the field $l(t_1^{1/n}, \dots, t_{d-1}^{ 1/n})$, 
we have that $X\bigl(k(t_1^{1/n}, \dots, t_{d-1}^{ 1/n})\bigr)
\not= \emptyset$ so that $X$ is a 
 $({\widetilde G}_k)^{\psi^{geo}(\mu_n)}$-torsor.
 Theorem \ref{thm_FG} enables to conclude that 
 $X(k) \not= \emptyset$. Without loss of generality we can 
 then assume that $\psi^{geo}$ and $\psi'^{geo}$
 agree on the diagonal. In particular, $\psi^{geo}$ and $\psi'^{geo}$ have both values in 
$\widetilde G^{\phi^{geo}(\mu_n)}$. In this case, 
 Equation \eqref{eq_cocycle3} applied to
 $\Delta(\mu_n(l))$
 implies that  \break $\widetilde g_0 \in 
\widetilde G^{\phi^{geo}(\mu_n)}\bigl(k(t_1^{1/n}, \dots, t_{d-1}^{ 1/n}) \bigr)$.

We introduce the field $K_d=k(x_1,...,x_d)$
equipped with the natural action of $\mu_n^d$. Then
we have a $\mu_n^d$-isomorphism
$k(t_1, \dots, t_{d-1})\simlgr (K_d)^{\Delta(\mu_n)}$
with the assignment $t_i \mapsto \frac{x_i}{x_1}$. 

\begin{claim} The homomorphism
$\psi^{geo}: \mu_n^d \to \widetilde G^{\phi^{geo}(\mu_n)}$ is anisotropic.
\end{claim}

We denote by $\underline{\psi}$  the composite
$\mu_n^d \xrightarrow{\psi} \widetilde G^{\phi^{geo}(\mu_n)} \to     \widetilde G_{\Delta_\phi}= \widetilde G^{\phi^{geo}(\mu_n)}/\phi^{geo}(\mu_n)$. Proposition \ref{prop_iso}, $(i') \Longrightarrow (i)$, shows that  $\underline{\psi}^{geo}$ is anisotropic. A fortiori $\underline{\psi}^{geo}$ is isotropic.
The Claim is established and
we rewrite the equation  \eqref{eq_cocycle3}
in $\widetilde G^{\phi^{geo}(\mu_n)}$ 

\begin{equation} \label{eq_cocycle4}
\psi(\sigma)= h_0^{-1} \, \psi'(\sigma)\,  \sigma(h_0)
\end{equation}
for all $\gamma \in \Gamma=
 \mu_n(l)^d \rtimes \Gal(l/k)$ with $h_0 \in \widetilde G^{\phi^{geo}(\mu_n)}\bigl(k(  x_1^{1/n}, \dots, x_{d}^{ 1/n}) \bigr)$.
 This is nothing but  assertion (iii) of Proposition 
\ref{prop_GP_pure} for $d$ and $r=d$ and the $k$-group 
$\widetilde G^{\phi^{geo}(\mu_n)}$.
The equivalence $(i) \Longleftrightarrow (iii)$ of the quoted statement shows that 
$\psi^{geo}$ and $\psi'^{geo}$
are $\widetilde G^{\phi^{geo}(\mu_n)}(k)$-conjugated.
 Therefore $\psi$ and $\psi'$
 are $\widetilde G(k)$-conjugated. Since
 the map $\widetilde G(A) \to \widetilde G(k)$ is onto
 we can assume that $\psi=\psi'$ without loss of generality.
 
If $\widetilde G$ is affine, it follows that 
$\phi$ and $\phi'$ are $G(A)$-conjugated by applying
\cite[XI.5.2]{SGA3}. In the general case we reason as in the proof of Corollary \ref{cor_GP_pure} by working in an affine $A$--subgroup $\widetilde G_1 \subset \widetilde G$.
The assertion (i) is established.
\end{proof}

\subsection{The main cohomological result}

We define  the tame part $H^1_{tame}(K_v, \widetilde G) \subset H^1(K_v, \widetilde G)$ as the 
union  of $H^1(K'/K_v, \widetilde G)$ for $K'$ running over all finite tamely unramified extensions over $K$.

\begin{stheorem} \label{thm_main} The map $H^1_{\substack{tame \\ loop}}(A_D, \widetilde G) \to H^1_{tame}(K_v, \widetilde G)$
is injective.
\end{stheorem}

In the case $r=1$, $A$ is a henselian DVR of fraction field
$K$ and $K_v$ is the completion of $K$.
In this case we know that the map $H^1(K,\widetilde G) \to H^1(K_v,\widetilde G)$ is bijective
\cite[Proposition 3.3.1,(2)]{GGMB} so that the statement holds. Though the case ``$r=0$" is excluded of the statement, 
it makes sense to treat  it separately as warm-up exercise.

\begin{slemma} \label{lem_r0} The map $H^1(A, \widetilde G) \to H^1_{tame}(K_v, \widetilde G)$
is injective.
\end{slemma}

\begin{proof}  We consider the following commuative diagram

\begin{equation} \label{diag15}
\xymatrix{
H^1\bigl( A ,  \widetilde G ) \ar[r] \ar[d]^{\cong}&
H^1\bigl( R_v ,  \widetilde G ) \ar@{^{(}->}[r] \ar[d]^{\cong}&
 H^1\bigl( K_v ,  \widetilde G ) \\
H^1\bigl( k ,  \widetilde G ) \ar@{^{(}->}[r]
 & H^1\bigl( k(t_1,\dots, t_{d-1}) ,  \widetilde G ). 
}
\end{equation}
where the vertical maps are the bijections of the Hensel lemma \cite[XXIV, Proposition 8.1]{SGA3};
the top  horizontal injection is  Theorem \ref{thm_Nis}
and the bottom horizontal injection 
follows of Theorem \ref{thm_FG}.
By diagram chase, we conclude that the map
$H^1\bigl( A ,  \widetilde G ) \to
 H^1\bigl( K_v ,  \widetilde G )$ is injective.
\end{proof}

We proceed now to the proof of Theorem \ref{thm_main}.

\begin{proof} The proof goes along the same steps than 
the proof of Theorem \ref{thm_gp_main}.
The injectivity of the map
$H^1_{\substack{tame \\ loop}}(A_D, \widetilde G) \to H^1(K_v, \widetilde G)$ 
rephases 
to show that the fiber at the class $[\phi]$ of any
loop cocycle $\phi$ consists in one element.
If $\phi$ is an anisotropic tame loop cocycle, 
Proposition \ref{prop_pure}, $(iii) \Longrightarrow (ii)$,  shows that the fiber
at $[\phi]$ of $H^1_{\substack{tame \\ loop}}(A_D, \widetilde G) \to H^1(K_v, \widetilde G)$
is $\{[\phi] \}$.

A first generalization is the irreducible case.
We assume then that the tame loop cocycle $\phi$
is irreducible.
Once again the usual twisting argument enables us to assume that $\phi^{ar}=1$.
Let $\phi'$ be another tame loop cocycle
such that $[\phi']=[\phi] \in H^1(K_v, \widetilde G)$.
According to Lemma \ref{lem_isotrivial}.(1),
the map  $H^1(A_D, J) \to   H^1_{tame}(K_v, J)$ is injective;
it follows that 
 $\phi$ and $\phi'$ have same image in  $H^1(A_D,J)$. 
 Next  Lemma \ref{lem_J2} permits to assume without loss
 of generality that  $\phi$ and $\phi'$ have same image in 
 $Z^1(\pi_1^t(X,x),J(\widetilde B)))$.
 We denote by $J_1$ the image of $\phi^{geo}: \widehat \ZZ'(1)^r \to J$,
 it is a finite smooth algebraic $A$--group of multiplicative type.
We put $\widetilde G_1= G_1 \times_J J_1$, by construction
$\phi$ and $\phi'$ have value in $\widetilde G_1(\widetilde B)$. 
To avoid any confusion we denote them by $\phi_1$ and $\phi'_1$.
We consider the commutative diagram of pointed sets 
\[
\xymatrix{
(J/J_1)(\widetilde B)^{\pi_1^t(X,x)_{\phi_1}} \ar[r] \ar[d]^{=} & 
H^1(\pi_1^t(X,x), \,  {_{\phi_1}\!G_1(\widetilde B)})  \ar[r] \ar[d] &
H^1(\pi_1^t(X,x), \, {_{\phi_1}G(\widetilde B)})  \ar[d] \\
\bigl({_{\phi_1}\!(J/J_1)}\bigr)(K_v) \ar[r] & H^1(K_v, \, _{\phi_1}\!\widetilde G_1)  \ar[r] & 
H^1(K_v,\,  {_{\phi_1}\!\widetilde G}).
}
\] 
The second one is associated to the exact sequence 
$1 \to {_{\phi_1}\!(\widetilde G_1}) \to
 {_{\phi_1}\!(\widetilde G )} \to {_{\phi_1}\!(J/J_1)} \to 1$ of $K_v$--spaces
and the first one is associated to the exact sequences
of $\pi_1^t(X,x)$-sets $1 \to {_{\phi_1}\!(\widetilde G_1(\widetilde B))} \to  {_{\phi_1}\!(\widetilde G(\widetilde B))}  \to {_{\phi_1}\!(J/J_1)(\widetilde B)} \to 1$.
By diagram chase involving the torsion bijection
$H^1(\pi_1^t(X,x), {_{\phi_1}\!\widetilde G_1(\widetilde B)})
\simlgr H^1(\pi_1^t(X,x), \widetilde G_1(\widetilde B))$, we see that we can arrange
 $\phi'_1$ in order
that $\phi'_1$ has same image than $\phi_1$ in $H^1(K_v,\widetilde G_1)$.
We can work then with then $\widetilde G_1$
which is generated by $G$ and the image of $\phi_1^{geo}$. 

Since the $A$--torus $C$ is central in $G$,
the  $A$--subgroup $C^{\phi^{geo}}$ is central in $\widetilde G_1$.
We denote by $C_0$ the maximal split $A$-subtorus of $C^{\phi^{geo}}$ 
and consider the central exact sequence of  $A$--group schemes
$$
1 \to C_0 \to \widetilde G_1 \to  \widetilde G_1/C_0 \to 1
$$
The gain is that the image of $\phi_1$ in 
$Z^1( \pi_1^t(X,x), (\widetilde G_1/C_0)(\widetilde B))$
is an anisotropic tame loop cocycle by applying the criterion of 
Corollary \ref{cor_PY2}.
Since $H^1(A_D,C_0)=1$ (Lemma \ref{lem_picard}), we obtain the following commutative diagram
 \[
\xymatrix{
 H^1(A_D, \widetilde G_1)  \ar[d]  & \hookrightarrow & H^1(A_D, \widetilde G_1/C_0)  \ar[d] \\
 H^1(K_v, \widetilde G_1)   & \hookrightarrow & H^1(K_v, \widetilde G_1/C_0) 
}
\]
 where the horizontal maps are injections \cite[III.3.4.5.(iv)]{Gd}. 
Proposition \ref{prop_pure} shows that $[\phi_1]$ and 
$[\phi'_1]$ have same image in $H^1(A_D, \widetilde G_1/C_0)$.
The diagram shows that $[\phi_1]= 
[\phi'_1] \in H^1(A_D, \widetilde G_1)$.
By pushing in  $H^1(A_D, \widetilde G)$
we get that $[\phi]=  [ \phi'] \in H^1(A_D, \widetilde G)$ as desired.

We deal now with the general case. The above reduction (with $\widetilde G_1$)
permits to assume that $J$ is finite \'etale so that $\widetilde G$ is affine and also that $\phi$ is purely geometric.
 Let $(P,L) \subset G$ be an $A$--parabolic subgroup $P$ of 
 $G$ together with a Levi subgroup $L$ both normalized by $\phi^{geo}$ 
 and which is minimal for this property.
Then the tame loop cocycle $\phi$ takes value in 
$\widetilde L(B)=N_{\widetilde G}(P,L)(\widetilde B)$. We 
have the exact sequence \eqref{eq_JPL} 
$$
1 \to L \to \widetilde L \to J_{P,L} \to 1
$$ 
 of smooth affine $A$-group schemes with $L= (\widetilde L)^0$
 and such that $J_{P,L}$ is twisted constant.
 We denote by $\psi$ the image of $\phi$ in 
 $Z^1\bigl(\pi_1^t(X,x), \widetilde L(\widetilde B)\bigr)$.
 Lemma \ref{lem_irr2}.(2) states that $\psi$ is irreducible.

 We deal now with the tame loop cocycle $\phi'$ having
 same image in $H^1(A_D, \widetilde G)$ as $\phi$. 
 Lemma \ref{lem_special} implies that 
 $\phi'$ is purely geometrical.
We consider the  $A$--scheme  $X= \widetilde G/ N_{\widetilde G}(P)$
which is projective according to Lemma \ref{lem_proj}.
Since $\phi$ takes values in $N_{\widetilde G}(P)(\widetilde B)$,
we have ${_\phi  X} =  ({_\phi  \widetilde G})_{K_v}/  ({_\phi N_{\widetilde G}(P)})_{K_v}$
so in particular  
 $({_\phi X})(K_v) \not = \emptyset$.
 Since  ${_{\phi'} X} \cong {_{\phi} X}$,
 it follows that  $({_{\phi'} X})(K_v) \not = \emptyset$.
Theorem \ref{thm_fix}, $(iv) \Longrightarrow (i)$,
  shows that $X^{{\phi'}^{geo}}(A)$ is not empty.
  We pick $x' \in X^{{\phi'}^{geo}}(A)$ 
and choose $x  \in X^{{\phi}^{geo}}(A)$ such that $G_x=P$. There exists 
$\widetilde g \in \widetilde G(\widetilde B)$
such that $x'=\widetilde g \, . \, x$.
For $\sigma \in \Gal(\widetilde B/A)$, we have that 
$x'=\sigma(\widetilde g) \, . \, x$ so that
$\sigma \to n_\sigma =\widetilde g^{-1} \, \sigma(\widetilde g)$
is a $1$-cocycle with value in $\widetilde G_x(\widetilde B)=N_{\widetilde G}(P)(\widetilde B)$.
For $\sigma \in  \Gal(\widetilde B/A)$, we have 
$\phi'(\sigma).x'=x'$ so that
$\phi'(\sigma).\widetilde g \, . \, x=\widetilde g \, . \, x$.
It follows that $\widetilde g^{-1} \phi'(\sigma) 
\widetilde g \, . \, x=x$. Since $n_\sigma\ x \, =x$
it follows that 
 $\phi''(\sigma)= \widetilde g^{-1} \, \phi'(\sigma)$
 fixes $x$.
\noindent Up to replace $\phi'$ by the equivalent cocycle
 $\phi''$ we can then assume that 
 $\phi'$ takes value in $N_{\widetilde G}(P)(\widetilde B)$.
As for $\phi$, we can modify further
by a coboundary in order that $\phi'$ has value in $N_{\widetilde G}(P,L)(\widetilde B)$.
We denote by $\psi': \Gamma \to N_{\widetilde G}(P,L)(\widetilde B)$
the result of this reduction.
Proposition \ref{prop_morozov} tells  us that 
$P$ admits a Levi subgroup $L'$ normalized by ${\phi'}^{geo}$.
Then $L'= \, {^g\!L}$ for some $g \in P(A)$.
 We note that  
 $P$ is minimal for this property (otherwise it will not be minimal
 for $\phi$ and equivalently for $\phi_{K_v})$.
  The above argument tells us that the image $\psi'$ of $\phi'$ in 
 $Z^1\bigl(\pi_1^t(X,x), \widetilde L(\widetilde B))$
 is irreducible. We consider now the commutative diagram

\[
\xymatrix{
 H^1(A_D,\widetilde G)  \ar[r] & H^1(K_v,\widetilde G)  \\ 
 H^1(A_D,N_{\widetilde G}(P,L))_{irr} \ar[r] \ar[u] & H^1(K_v,N_{\widetilde G}(P,L))_{irr} \incl[u] \quad .
}
\]
The bottom  horizontal map is well-defined in view of Proposition \ref{prop_red}.
 The  right vertical map is injective  \cite[lemme 4.2.1.(2)]{Gi15}.
 We have seen that $\phi$, $\phi'$ provide elements $[\psi]$, $[\psi']$ of 
 $H^1(A_D,N_{\widetilde G}(P))_{irr}$ which give then the same image in
 $H^1(K_v,N_{\widetilde G}(P))$. By diagram chase we conclude that 
 $[\phi]= [\phi'] \in H^1(A_D,N_{\widetilde G}(P))$.
 Thus  $[\phi]= [\phi'] \in H^1(A_D,\widetilde G)$ as desired.
\end{proof}

\section{Examples}
In the Laurent polynomial setting we discussed examples 
of orthogonal groups in 
\cite[\S 10]{GP13}.
Starting with the $\GL_N$-case, we will discuss for the same groups the main results of the paper
in the Abhyankar's setting.
In other words, we continue to work with the setting
of \S \ref{sect_murre}, it means with the  regular local henselian ring $A$ of dimension $d \geq 1$, a regular system of parameters $(f_1, \dots, f_d)$ and the localization 
$A_D=A_{f_1 \dots f_r}$ where $r$ is an integer
satisfying $1 \leq r \leq d$.

\subsection{The  linear group}

\begin{slemma}\label{lem_gl} For each $N \geq 1$, we have
$H^1_{\substack{tame \\ loop}}(A_D,\GL_N)=1$.
\end{slemma}

\begin{proof} Since 
$H^1(A,\GL_N)=1$, any tame loop $A_D$--torsor under 
$\GL_N$ is purely geometrical. Let $\phi^{geo}: \mu_{n,A}^r \to \GL_{N,A}$
be a homomorphism of $A$-group schemes. According to \cite[I.4.7.3]{SGA3},
the representation $\phi^{geo}$ is diagonalizable so that factors through
 a maximal split maximal $A$-subtorus $T \subset \GL_N$. 
 Since $H^1(A_D,T)=\Pic(A_D)^r=0$, it follows
 that $[\phi^{geo}]=1 \in H^1(A_D,\GL_N)=1$.
Thus  $H^1_{\substack{tame \\ loop}}(A_D,\GL_N)=1$.
\end{proof}

Since $H^1(A_D, \GL_N)$ classifies
projective $A_D$--modules of rank $N$, we may expect that 
those $A_D$--modules are free. This is actually a conjecture
of Rao \cite{Rao} and  we list the known cases.

\begin{sproposition}\label{prop_gl}
Let $N \geq 1$ be an integer.
\smallskip

\noindent (1) Projective $A_D$--modules of  rank $N$ are stably free and are free if $N \geq d+1$.
In other words, $H^1(A_D, \GL_N)=1$ if $N \geq d+1$.

\smallskip

\noindent (2) If $d \leq 3$, a projective $A_D$-module of rank $N$ is free. In other words, $H^1(A_D,\GL_N)=1$.

\end{sproposition}

\begin{proof}
(1) We consider the commutative diagram 
\[
\xymatrix{
 K_0(A) \ar[r] \ar[d] & K_0(A_D) \ar[d]  \\
G_0(A) \ar@{->>}[r]  & G_0(A_D), 
}\]
where the surjectivity of the bottom map is 
\cite[6.4.1]{We}.
Since $A$ is regular, the vertical maps are isomorphisms 
({\it ibid}, Theorem 8.2).
Since $A$ is local, we have $\ZZ= K_0(A)$
so that we conclude that $\ZZ= K_0(A_D)$.
In other words, f.g.\ $A_D$--modules are stably free.
Since $A_D$ has Krull dimension $\leq d$, 
Bass stability theorem \cite[Corollary 3.2.4]{Kn}
enables to conclude that a projective $A_D$-module of rank $N \geq d+1$ is free.
\sm

\noindent (2) If $d \leq 3$, a special case of 
Gabber's purity theorem \cite[Theorem 2.3]{Gb}
tells us that a projective $A_D$-module of rank $N$
extends to a projective $A$-module of rank $N$ so is free.
\end{proof}

\begin{sremark}{\rm
 We discuss now Rao's conjecture implying that  f.g.\ $A_D$-projective modules
 are free. This is related with the following Nisnevich's conjecture \cite[conj.\ 3.5]{N2}.
Let $F$ be an arbitrary field, let $X$ be a local 
essentially smooth $F$-scheme.
Let $D$ be an $F$-divisor on $X$ such that all irreducible components $D_i$ of $D$ are smooth  and such that their pairwise intersections are transversal. 
Let $Y = X \setminus  D$ and let $E$ be a vector bundle on $Y$. Then $E$ is free.
}
\end{sremark}

\subsection{The  orthogonal group }\label{subsec_witt}
If $2$ is not invertible in $R$,
then tame loop torsors are not interesting for orthogonal groups.
We assume then that $2 \in R^\times$ and  consider the orthogonal group $\mathrm{O}(N)$ for $N\geq 1$ of the diagonal quadratic $\sum\limits_{i=1}^N X_i^2$.
We know that $H^1\big(A_D, \mathrm{O}(N)\big)$  classifies regular quadratic forms over $A_D^{N}$ \cite[\S 4.6]{Kn}. 
Our goal is to determine the set of tame loop quadratic forms.
We start with Witt groups of fields and rings
(as defined in \cite[\S I.6]{Sch}) and remind
the reader of the bijection $H^1\big(A, \mathrm{O}(N)\big)
\simlgr H^1(k, \mathrm{O}(N)\big)$  \cite[XXIV.8.1]{SGA3}.
This implies  that the map 
$W(A) \to W(k)$ is an isomorphism. We consider the map

\begin{equation}\label{eq_h}
\xymatrix{
W(k)^{2d} \ar[rr]^h && W(A_D) \\
(q_I)_{I \subset \{1,\dots, r \}} 
& \mapsto & \perp_{I \subset \{1,\dots, r\}} \enskip
\langle f_I \rangle \otimes \widetilde q_i 
}
\end{equation} 
is $f_I=\prod\limits\limits_{\in I} f_i$ and
where each $\widetilde q_I$ is an $A$-lift of the $k$-form $q_I$.
We will see that it is a  quite good approximation of $W(A_D)$.
Taking again $f_1$ as uniformizing parameter for $v$, we consider
the Springer isomorphism \cite[VI, Corollary 1.6]{Lam}
\begin{equation}\label{eq_residue}
W(K_v) \xrightarrow{s_{f_1} \oplus \partial}
W(k(t_1,\dots, t_{r-1})) \oplus W(k(t_1,\dots, t_{r-1}) ) .
\end{equation}
  The first projection is called the specialization along 
  $f_1$ and the second projection is called the residue.

\begin{sproposition} \label{prop_witt}
(1) The isomorphism \eqref{eq_residue} induces commutative diagram

\begin{equation}\label{diag_witt1}
\xymatrix{
W(K_v) \ar[r]^{s_{f_1} \oplus \partial \qquad \qquad \quad\qquad}_{\sim \qquad \qquad \quad\qquad}&
W(k(t_1,\dots, t_{r-1})) \oplus W(k(t_1,\dots, t_{r-1}))  \\
    W(A_D) \ar@{->>}[r]^{h \qquad\qquad\qquad\quad } \ar[u]&
    W(k[t_1^{\pm 1},\dots, t^{\pm}_{r-1}]) \oplus W(k[t_1^{\pm 1},\dots, t^{\pm}_{r-1}]) \ar@{^{(}->}[u] .
}
\end{equation}

\noindent (2) The map \eqref{eq_h} induces 
an isomorphism $$
W(k)^{2d} \simlgr \mathrm{Im}\Bigl( W(A_D) \to W(K_v) \Bigr)
$$
Furthermore given anisotropic 
quadratic $k$-forms $(q_I)_{I \subset \{1,\dots, r \}}$,
the $A_D$-form $\perp_{I \subset \{1,\dots, r\}} \,
\langle f_I \rangle \otimes \widetilde q_i$
is $A_D$-anisotropic and $K_v$--anisotropic.

\smallskip

\noindent (3) Let $q$ be a regular quadratic $A_D$-form 
of rank $N \geq 1$. There exists an unique 
diagonalizable quadratic $A_D$-form $q'$ of rank $N$
such that $q_{K_v} \cong q'_{K_v}$.

\smallskip

\noindent (4) For each $N \geq 1$, diagonalizable regular quadratic $A_D$-forms of rank $N$ are exactly the tame loop objets of $H^1(A_D,\rO(N))$.

\smallskip

\noindent (5) Diagonalizable regular quadratic $A_D$-forms of rank $N$ are classified by their Witt class in $W(K_v)$
and a fortiori by their Witt class in $W(A_D)$.

\end{sproposition}

\begin{sremark}{\rm If $d \leq 3$, we have 
$W(A_D) = W_{nr}(A_D)$ according to a general purity statement \cite[Corollary 10.3]{BW}.
}
\end{sremark}

Note that the group  $W(k[t_1^{\pm 1},\dots, t^{\pm}_{r-1}])$
is well-understood as $W(k)^{2r-2}$ by iterating Karoubi-Ranicki's decomposition for the Witt group of Laurent polynomials over a regular ring 
(\cite[\S 3]{Kb}, \cite[\S 4]{Rn}), 
see also \cite{OP} for a survey.
In particular this makes clear that 
 $W(k[t_1^{\pm 1},\dots, t^{\pm}_{r-1}])$
 identifies with the  
 unramified part of $W(k(t_1,\dots, t_{r-1}))$ 
 with respect to $k[t_1^{\pm 1},\dots, t^{\pm}_{r-1}]$.
We proceed now to the proof of Proposition \ref{prop_witt}.

\begin{proof} (1) {\it Step 1:} $h$ is well-defined.
We check first that the residue  $\partial$ applies
$W(A_D)$ on 
$W(k[t_1^{\pm 1},\dots, t_{r-1}^{\pm 1}])$.
The ring $C=A_D\cap R$ is regular (Lemma \ref{lem_AC}.(3)) and
$C/f_1 C= k[t_1^{\pm 1},\dots, t_{r-1}^{\pm 1}]$ is regular as well. We can use then the commutative diagram of residues
of Balmer-Witt theory

\[\xymatrix{
W(K_v) \ar[r]^{\partial \qquad}&
W(k(t_1,\dots, t_{r-1})) & \\
W(A_D) \ar[r]^{\partial'} \ar[u] & 
W^1_{f_1 C}(C) & \ar[l]^\sim 
W(k[t_1^{\pm 1},\dots, t_{r-1}^{\pm 1}]) \ar[ul]
}\]
where the  bottom isomorphism is due to S.~Gille
 \cite[Theorem 4.12]{Ge}.
Since the specialization can be written as 
$\partial( \langle - f_1 \rangle \otimes ?)$,
the same argument shows that it takes value as well
in $W(k[t_1^{\pm 1},\dots, t_{r-1}^{\pm 1}])$.

\smallskip
\noindent {\it Step 2:} $h$ is onto by using
the map \eqref{eq_h}.

\smallskip

\noindent (2) The first assertion follows from the step 2
above. We are given anisotropic 
quadratic $k$-forms $(q_I)_{I \subset \{1,\dots, r \}}$,
and 
$\perp_{I \subset \{1,\dots, r\}} \enskip
\langle f_I \rangle \otimes \widetilde q_i$
is $K_v$-anisotropic by iterating Springer's theorem 
\cite[VI, Proposition 1.9.(2)]{Lam}.
A fortiori it is  $A_D$--anisotropic.

\smallskip

\noindent (3) As we saw before there exists 
an diagonalizable quadratic $A_D$-form $q'$ such that $[q]_{K_v}=[q']_{K_v} \in W(K_v) $ 
with $q'_{K_v}$ anisotropic.
It follows that the $\rank(q') \leq \rank(q)$
so that $q_{K_v}= (q'\perp \mathbb{H}^c)_{K_v}$
where $c$ is the Witt index of $q_{K_v}$.
  Such a $q' \perp \perp \mathbb{H}^c$ is unique since $H^1_{\substack{tame \\ loop}}\big(A_D, \mathrm{O}(N)\big)$ injects in
  $H^1\big(K_v, \mathrm{O}(N)\big)$ in view of Theorem 
  \ref{thm_main}.
  
\smallskip

\noindent (4) 
 Diagonalizable quadratic $A_D$--forms of rank $N$
arise  from the diagonal subgroup $\mu_2^N \subset \rO(N)$ as in the introduction. It follows that  
 diagonalizable quadratic $A_D$-forms are tame loop objects
  in view of Lemma \ref{lem_finite}.(3).
  We have to check that it exhausts tame loop objects.
  Let $q$ be a regular quadratic $A_D$-form of dimension $N$
which is a tame loop object.  
By  (3), there exists  a diagonalizable $A_D$-form of rank $N$
such that $q_{K_v} \cong q'_{K_v}$.
  Once again the injection  $H^1_{\substack{tame \\ loop}}\big(A_D, \mathrm{O}(N)\big)$ enables us to conclude that  $q$ is isometric to $q'$. Thus $q$ is diagonalizable.

\smallskip

\noindent (5)  This is one more consequence of 
Theorem \ref{thm_main}.

\end{proof}

\begin{scorollary} \label{cor_witt}
Let $q$ be a diagonalizable  regular quadratic $A_D$-form of rank $N$. Then for each $c\geq 1$, the following are equivalent: 
\smallskip

(i) $q$ decomposes as $q=q_0 \perp \mathbb{H}^c$ where $q_0$
is an tame loop  regular quadratic $A_D$-form of rank $N-2c$;

\smallskip

(ii)  $q_{K_v}$ decomposes as $q_{K_v}=q_1
 \perp \mathbb{H}^c$ where $q_1$
is a regular quadratic $K_v$-form of rank $N-2c$;

\smallskip

(ii') The Witt index of  $q_{K_v}$ is $\geq c$.
\end{scorollary}

\begin{proof} The equivalence $(ii) \Longleftrightarrow (ii')$
follows of the definition of the Witt index and the implication $(i) \Longrightarrow (ii)$ is obvious.
Assume (ii), that is, $q_{K_v}=q_1
 \perp \mathbb{H}^c$ for some $K_v$-form $q_1$
 which can be assumed anisotropic. As we have seen in the proof of 
 Proposition \ref{prop_witt}.(3), there exists a diagonalizable quadratic $A_D$-form $q_0$ such that $q_{0, K_v}=q_1$. Since the diagonalizable quadratic form $q_0 \perp  \mathbb{H}^c$ becomes isomorphic to $q$ over $K_v$,
 Proposition \ref{prop_witt}.(5) enables us to conclude that 
$q \cong q_0 \perp  \mathbb{H}^c$.
\end{proof}

\begin{sremark}{\rm
We encourage the reader to concoct another proof by applying Proposition \ref{prop_iso} and the interpretation of
Witt decomposition with group schemes (e.g.\ \cite[5.5]{GN}). 
}
\end{sremark}

This opens  questions, first 
whether  any quadratic $A_D$-form is stably diagonalizable, that is, diagonalizable after adding a suitable hyperbolic form. This is equivalent to
the injectivity of the map $W(A_D) \to W(K_v)$ in view of 
Theorem \ref{thm_main}.



\appendix

\section{Injectivity property for non-abelian cohomology}

We start by extending a classical fact on algebraic $k$--groups 
which goes back to Bruhat and Tits in the 
reductive case \cite[Proposition 5.5]{Guo}, 
see \cite[Proposition 5.4]{F} and \cite[Theorem 5.4]{FG} for extensions.

\begin{stheorem} \label{thm_FG} Let $G$ be a locally algebraic $k$--group.
Then the map $H^1(k,G) \to H^1(k((t)),G)$ is injective.
\end{stheorem}

\begin{proof} 
We have an exact sequence $1 \to G^0 \to G \to J \to 1$
where $G^0$ is an algebraic $k$--group
and $J$ is an \'etale $k$-group 
and then a twisted constant $k$--group scheme \cite[\S II.5.1]{DG}.
It follows that the representability facts of Lemma \ref{lem_descent0}
apply.
In particular the usual twisting argument boils
down to establish only the triviality of the kernel
of  $H^1(k,G) \to H^1(k((t)),G)$. 
We consider the commutative diagram of exact sequence of pointed sets

\[\xymatrix{
J(k) \ar[r] \ar[d]^{\wr}_\alpha & H^1(k,G^0) \ar[r] \ar[d]_\beta  &  H^1(k,G) \ar[r] \ar[d]_\gamma & 
    H^1(k,J) \ar[d]_\delta \\
    J(k((t))) \ar[r] & H^1(k((t)),G^0) \ar[r]  &  H^1(k((t)),G) \ar[r]  & 
    H^1(k((t)),J).
}\]
The map $\alpha$ is bijective, the map $\beta$ is injective
\cite[Theorem 5.4]{FG}. On the other hand, the map $\delta$
decomposes in $$ H^1(k,J) \to  H^1(k[[t]],J) \to  H^1(k((t)),J).$$
The first map is obviously split injective
and the second one  is injective  (Lemma \ref{lem_isotrivial}.(1))
so that $\delta$ is injective.  A diagram chase enables us to conclude
that $\ker(\gamma)=1$.
\end{proof}

For a  general DVR, we can extend Nisnevich result 
as follows \cite[Theorem 4.2]{N1} \cite[Theorem 1.2]{Guo} from the reductive case
to our setting.

\begin{stheorem} \label{thm_Nis}
Let $R$ be a semilocal Dedekind ring of fraction field $K$.
Let $1 \to G \to \widetilde G \to J \to 1$ be an exact 
sequence of  smooth $R$-group schemes such that 
$G$ is reductive and $J$ is twisted constant. 
Then the map $H^1(R,\widetilde G) \to H^1(K,\widetilde G)$
is injective. 
\end{stheorem} 

\begin{proof}
The usual twisting argument
boils down to establish the triviality of the kernel of the map 
$H^1(R,\widetilde G) \to H^1(K,\widetilde G)$.
According to Proposition  \ref{prop_BS},
we have  $J(R)=J(K)$. We consider now the 
commutative diagram of exact sequences of pointed sets
\begin{equation}\label{diag_Nis}
\xymatrix{ 
  J(R) \ar[d]^{\wr} \ar[r]^{\phi} & H^1(R,G)  \ar[r] 
  \ar@{^{(}->}[d] & H^1(R,\widetilde G)  \ar[r] \ar[d]& H^1(R,J) \ar@{^{(}->}[d]\\
  J(K) \ar[r]^{\phi}  & H^1(K,G)  \ar[r]  & H^1(K,\widetilde G)  \ar[r]& 
  H^1(K, J)
}
\end{equation}
where we reported injectivity for $G$ \cite[Theorem 1.2]{Guo} 
and for $J$
 (Lemma \ref{lem_isotrivial}.(1)).  
 Let $\widetilde \gamma$ be a class in the kernel of 
 $H^1(R,\widetilde G) \to H^1(K,\widetilde G)$.
  A diagram chase shows that $\widetilde \gamma$
  comes from $\gamma \in  H^1(R,G)$ such that $\gamma_K= (\phi(x))_K$
  for some $x \in J(R)$. We consider the $G$-torsor $E=p^{-1}(x)$; its 
  class is $\phi(x)$. Then the torsion bijection \break
  $\tau: H^1(R, {^EG}) \simlgr H^1(R,G)$ satisfies
  that $\tau^{-1}(\gamma) \in \ker\Bigl( H^1(R, {^EG}) \to  H^1(K,{^EG}) \Bigr)$.
  Since $G$ is reductive, Nisnevich-Guo's theorem shows that 
  $\tau^{-1}(\gamma)=1$ so that $\gamma_= \phi(x)$.
  Thus $\widetilde \gamma=1$ as desired.
\end{proof}

The same kind of reduction provides the following useful fact.

\begin{slemma} Under the assumptions of Theorem \ref{thm_Nis}, 
we have $\widetilde G(K)= G(K) \, \widetilde G(R)$.
\end{slemma}

\begin{proof} We consider the extended left side of the diagram \eqref{diag_Nis}
above
\begin{equation}\label{diag_Nis2}
\xymatrix{ 
 G(R) \ar[r] \ar[d]&  \widetilde G(R) \ar[d] \ar[r]^q & J(R) \ar[d]^{\wr} \ar[r]^{\phi} & H^1(R,G)  
  \ar@{^{(}->}[d] \\
  G(R) \ar[r] &  \widetilde G(R) \ar[r]^q & J(K)  \ar[r]^{\phi_K} & H^1(K,G).
}
\end{equation}
We are given  $\widetilde g \in \widetilde G(K)$.
By diagram chase, we find $\widetilde g_0 \in \widetilde G(R)$
such that  $q(\widetilde g)= q(\widetilde g_0) \in J(K)$.
It follows that  $\widetilde g \, (\widetilde g_0)^{-1} 
=g_1$ with $g_1 \in G(K)$. Thus $ \widetilde g= g_1 \, \widetilde g_0$. 
\end{proof}

\section{Appendix: (enlarged) Bruhat-Tits theory of a split reductive group} \label{app_BT}
This theory simplifies quite substantially
when we deal with the case of a split reductive group
$G$ over a field $K$ which is henselian for 
a discrete valuation. We denote by $O$ the valuation 
ring of $K$ and by $k$ its residue field (not  assumed necessarily perfect).
The $K$--group $G$ is then extended from a Chevalley $O$--group 
scheme $G$ itself extended from $\ZZ$.

Let $C$ (resp.\ $Q$)
 be the radical torus of $G$. Both tori  are  split
 and the map $p: C \to Q$ is an isogeny 
 whence an isomorphism $p_*: \widehat C^0 \otimes_\ZZ
 \RR \simlgr  \widehat Q^0 \otimes_\ZZ \RR$.
Let $(B,T)$ be a Killing couple of $G$
and put $N=N_G(T)$. 
The enlarged building of $G$ over $K$
is $$
\cB_e(G_K)=  \cB(DG_K) \times (\widehat C^0 \otimes_\ZZ \RR)
 =\cB(DG_K) \times E
$$
where $\cB(DG_K)$ is the Bruhat-Tits building 
of the derived $K$--group $DG$ which is isomorphic to the building of 
its simply connected cover $(DG)^{sc}_K$ (and of its adjoint group $G_{ad,K}$).

The valuation induces a map $v_Q: Q(K) =  \widehat Q^0 \otimes_\ZZ K^\times\to 
\widehat Q^0$, it gives rise
to an action of $Q(K)$ on $\widehat Q^0 \otimes_\ZZ \RR$
by $q. \lambda= \lambda - v_Q(q)$.
It provides as well an action on 
$Q(K)$ on $E=\widehat C^0 \otimes_\ZZ \RR$ (and $E$ has to be seen
as an affine space)
and an action of $G(K)$ through the morphism
$G(K) \to Q(K)$  (see \cite[\S 4.3]{KP}).

We have a natural action of $G(K)$
on $\cB_e(G_K)$. It acts through the action 
of $G_{ad}(K)$ on $\cB(DG_K)$
and the previous action of $G(K)$
on $(\widehat C^0 \otimes_\ZZ \RR)$.
According to \cite[9.1.19.(c)]{BT1}, there exists a unique point $\phi \in 
\cB(DG_K)$ called the center of the building
such that
$(DG)^{sc}(O)=\Stab_{(DG)^{sc}(K)}(\phi)$; we consider the point $ \phi_e=(\phi, 0 )$
of the extended building.

\begin{slemma} \label{lem_stab}
$G(O)= \Stab_{G(K)}(\phi_e)$.
\end{slemma}

\begin{proof} The direct inclusion is clear. 
Conversely let $g \in G(K)$ fixing
$ \phi_e$  and denote by $q \in Q(K)$
its image under the canonical map $G(K) \to Q(K)$. Then $v_Q(q)=0$ so that 
$q \in Q(O)$. We consider the exact sequence
$1 \to DG \to G \to Q \to 1$.
According to \cite[A.2.7]{CGP}, $T'=T \cap DG$ is a maximal 
torus of $DG$ so that we have an exact sequence
of split tori $1 \to T' \to
T \to Q \to 1$.
It follows that the map $G(O) \to Q(O)$ is split 
and a fortiori onto. We can pick then $g_0 \in G(O)$
mapping on   $q$; replacing $g$ by $g g_0^{-1}$
 reduces to the case $g \in DG(K)$.
We use now the isogeny $1 \to \mu \to (DG)^{sc} \xrightarrow{p} DG \to 1$.
We have a compatible exact sequence $1 \to \mu \to T^{sc} \to 
T' \to 1$ where $T^{sc}$ is the inverse image of $T'$.
 We have  a commutative diagram of exact sequences of 
 pointed sets
\[\xymatrix{
 (DG)^{sc}(K)  \ar[r] & DG(K)  \ar[r] & 
 H^1_{flat}(K,\mu) \ar[r] & H^1(K,(DG)^{sc})   \\ 
T^{sc}(K)  \ar[r] \ar[u] & T'(K) \ar[u]  \ar[r] & 
 H^1_{flat}(K,\mu) \ar[r] \ar[u]^{\wr} & \ar[u] H^1(K,T^{sc})=0 .
}\]
By diagram chase we  get  a decomposition 
$$
(DG)^{sc}(K) = \mathrm{Im}\Bigl( (DG)^{sc}(K) 
\xrightarrow{p} DG(K) \Bigr) \, T'(K).
$$
We deal then with $g = p(g_1) \,  g_2^{-1}$ with $g_1 \in 
 (DG)^{sc}(K) $ and $g_2 \in T'(K)$.
It follows that $p(g_1). \phi= g_2. \phi$.
This point is then of same type than $\phi$
and belongs to the apartment $\cA$ defined by $T^{sc}$.
Since $T^{sc}(K)$ acts  transitively
on the points of   $\cA$ of type $0$,
there exists $g_3 \in T^{sc}(K)$ such that 
$p(g_1). \phi= g_2 .\phi=  p(g_3) .\phi$.
Up to replace $g_1$ (resp.\ $g_2$) by $g_1 \, g_3^{-1}$ 
(resp.\ $g_2 \,  p(g_3)^{-1}$) we are reduced to the case
that $p(g_1) . \phi=\phi$ and $g_2. \phi=\phi$.
It follows that  $g_1 \in (DG)^{sc}(O)$.
On the other  hand $g_2$ acts then trivially on the apartment $\cA$
hence $g_2 \in T'(O)$. Thus $g= p(g_1) \, g_2^{-1}$
belongs to $DG(O)$ as desired.
\end{proof}

We deal now with an exact sequence of locally algebraic
$k$--groups 
$1 \to G \to  \widetilde G \to J \to 1$
such that $G=\widetilde G^0$.
We claim that the action of $G(K)$
on $\cB_e(G_K)$ extends naturally to an action 
on $\widetilde G(K)$.  Pushing the above sequence
by $G \to Q$ gives rise to an exact sequence 
$1 \to Q \to  \widetilde Q \to J \to 1$
and to the commutative exact diagram of exact sequences 
aves  such that the following diagram commutes

\[\xymatrix{
 1 \ar[r] & Q(O) \ar[d] \ar[r] & \widetilde Q(O)   \ar[d]
  \ar[r] & J(O )\ar[r] \ar[d]^{\wr} & 1  \\ 
 1 \ar[r] & Q(K) \ar[r] & \widetilde Q(K)   
  \ar[r] & J(K )\ar[r] & 1  
}\]
by using the triviality of $H^1(O,Q)$ and  $H^1(K,Q)$.
Pushing one more time by \break $v_Q: Q(K) \to (\widehat Q)^0 \otimes_\ZZ \RR$
gives rise to the diagram

\[\xymatrix{ 
1 \ar[r] & Q(K) \ar[r] \ar[d] & \widetilde Q(K)   \ar[d]
  \ar[r] & J(K) \ar[r]\ar[d]^{\wr} & 1  \\ 
  0 \ar[r] & (\widehat Q)^0 \otimes_\ZZ \RR \ar[r] & 
  H    \ar[r] & J(K) \ar[r] & 1
}\]
with a canonical decomposition $H= \bigl( (\widehat Q)^0 \otimes_\ZZ \RR \Bigr) \rtimes J(K)$. The group $H$ acts on 
$(\widehat Q)^0 \otimes_\ZZ \RR \cong E$ as follows:
$$
(y, \tau).x=   \tau(x)- y .
$$
This extends then the opposite translation action of 
$(\widehat Q)^0 \otimes_\ZZ \RR$ on itself.
Composing with the projection  $\widetilde G(K) \to H$,
we extended then the action of $G(K)$ on 
$E$.
On the other hand, $\widetilde G(K)$ acts by group automorphisms
on $DG$ hence acts on $\cB(DG_K)$.
Altogether we have then an action of $\widetilde G(K)$
on $\cB_e(G_K)$.

\begin{slemma} \label{lem_stab2}
We have 
$\widetilde G(O)= \Stab_{\widetilde G(K)}(\phi_e)$.
\end{slemma}

\begin{proof} 
The direct inclusion is straightforward. 
Conversely we are given 
$\widetilde g \in \widetilde G(K)$ fixing $\phi_e$.
We  consider the following diagram
of exact sequences of pointed sets

\[\xymatrix{ 
  G(K) \ar[r] & \widetilde G(K) \ar[r] &  J(K) \ar[r] & H^1(K,G) \\
  G(O) \ar[r] \ar[u] & \widetilde G(O) \ar[r] \ar[u] &  J(O) \ar[r] \ar[u]^{=} & H^1(O,G) .\ar[u] 
}\]
Using the reduction argument to a Borel subgroup of $G$,
the map $H^1(O,G) \to H^1(K,G)$ has trivial kernel.
By diagram chase it follows that $\widetilde g= g \, \widetilde g_0$
with $\widetilde g_0 \in  \widetilde G(O)$ and $g \in  G(K)$. According to Lemma \ref{lem_stab}, we have that $g \in G(O)$.
Thus $\widetilde g$ belongs to $\widetilde G(O)$.
\end{proof}

\bigskip

\medskip


\bigskip
 
 
\begin{thebibliography}{99}


\bibitem{BW} P.~Balmer, C.~Walter, {\it 
A Gersten-Witt spectral sequence for regular schemes},
Annales scientifiques de l'Ecole Normale Sup\'erieure
{\bf  35} (2002), 127-152. 






\bibitem{BS} B.~Bhatt, P.~Scholze, {\it  The pro-\'etale topology for schemes}, Ast\'erisque {\bf 369} (2015),  99-201.




\bibitem{Bw} I.I.~Bouw, {\it Covers of the affine line in positive characteristic with prescribed ramification},
in Fields Institute Communications {\bf 60}, American Mathematical Society, Providence, RI,
2011, pp. 193-200.



\bibitem{BoT65} A.~Borel, J.~Tits, {\it Groupes r\'eductifs}, 
Publ.\,Math.\,IHES {\bf 27}(1965),   55--151.


\bibitem{BA} N.~Bourbaki, {\it Alg\`ebre}, Ch. 1 \`a  10, Springer.

\bibitem{BAC} N.~Bourbaki, {\it Alg\`ebre commutative}, Ch. 1 \`a  10, Springer.

\bibitem{BC} A.~Bouthier, K.~\v{C}esnavi\v{c}ius, {\it
Torsors on loop groups and the Hitchin fibration}, Annales scientifiques de l'Ecole normale sup\'erieure {\bf 55} (2022), 791-864.


\bibitem{BH}  W.~Bruns, H.J.~Herzog, {\it Cohen-Macaulay Rings}, 
Second edition, Cambridge Studies in Advanced Mathematics, Cambridge University Press.

   
\bibitem{BLR} S.~Bosch, W.~L\"utkebohmert, M.~Raynaud, 
{\it N\'eron models}, Ergebnisse der Mathematik und ihrer Grenzgebiete 
{\bf 21} (1990), Springer.

\bibitem{BT1} F.~Bruhat, J.~Tits,
{\it Groupes r\'eductifs sur un corps local.
 I. Donn\'ees radicielles valu\'ees},
 Inst. Hautes Etudes Sci.\ Publ.\ Math.\  {\bf  41}  (1972), 5--251.

\bibitem{BT2} F.~Bruhat, J.~Tits, {\it  
 Groupes r\'eductifs sur un corps local : II. Sch\'emas en groupes. Existence d'une donn\'ee radicielle valu\'ee},
Publications Math\'ematiques de l'IHES {\bf  60}  (1984), 5-184. 

\bibitem{C} K.~\v{C}esnavi\v{c}ius, {\it Torsors on the complement of a smooth divisor},
arXiv:2204.08233



\bibitem{CGP14} V.~Chernousov, P.~Gille, A.~Pianzola, {\it Conjugacy theorems for loop reductive group schemes and Lie algebras}, Bulletin of Mathematical Sciences {\bf 4} (2014), 281-324.



\bibitem{CGP24} V.~Chernousov, P.~Gille,   A.~Pianzola, {\it  Loop torsors.
Theory and Applications}, to appear in Documenta Mathematica.



\bibitem{CGP17} V.~Chernousov, 
P.~Gille, A.~Pianzola, {\it Classification of torsors over Laurent polynomial rings}, 
Commentarii Mathematici Helvetici {\bf 92} (2017), 37-55.


\bibitem{CNP} V.~Chernousov, E.~Neher, A.~Pianzola, {\it
Conjugacy of Cartan subalgebras in EALA's with a non-fgc centreless core}, Trans. Moscow Math.\ Soc.\ 2017, 235-256. 


\bibitem{CGR1} V.~Chernousov, P.~Gille,   Z.~Reichstein, {\it 
Resolution of torsors by abelian  extensions},  Journal of Algebra  {\bf 296}, (2006), 561-581.


\bibitem{CGR2}
V.~Chernousov, P.~Gille,   Z.~Reichstein, {\it  Reduction of structure for torsors over semi-local rings}, Manuscripta Mathematica {\bf 126} (2008), 465-480.


\bibitem{CS} V.~Chernousov, J.-P.~Serre, {\it 
Lower bounds for essential dimensions via
orthogonal representations}, Journal of Algebra {\bf 305} (2006) 1055-1070.


 



\bibitem{CGP} B.~Conrad, O.~Gabber, G.~Prasad, {\it
Pseudo-reductive groups},   Cambridge University Press, second edition (2016).

 
  


\bibitem{DG} M.~Demazure, P.~Gabriel, {\it Groupes alg\'ebriques},
    North-Holland (1970).
 
\bibitem{EGA1} A.\ Grothendieck, J.-A.\ Dieudonn\'e, {\it El\'ements 
de g\'eom\'etrie alg\'ebrique. I}, Grundlehren der Mathematischen Wissenschaften  166; Springer-Verlag, Berlin, 1971.





\bibitem{EGA4} A.~Grothendieck (avec la collaboration de J.~Dieudonn\'e), 
{\it El\'ements de G\'eom\'etrie Alg\'ebrique IV}, Publications math\'ematiques de l'I.H.\'E.S. 
no 20, 24, 28 and 32 (1964 - 1967).


\bibitem{F}  M.~Florence, {\it  Points rationnels sur les espaces homog\`enes et leurs compactifications}, Transformation Groups {\bf 11} (2006), 161-176.
 
\bibitem{FG} M.~Florence, P.~Gille, {\it  Residues on affine grassmannians}, Journal f\"ur die reine und angewandte Mathematik {\bf  776} (2021), 119-150.


\bibitem{Fu} L.~Fu, {\it  Etale Cohomology Theory},
Revised Edition, Nankai Tracts in Mathematics {\bf 14} (2015).

 \bibitem{Gb} O.~Gabber,  {\it On purity theorems for vector bundles}, 
International Mathematics Research Notices {\bf  2002} (nr.\ 15) (2002),  783-788. 


\bibitem{GGMB} O.~Gabber, P.~Gille L.~Moret-Bailly,  {\it Fibr\'es principaux sur les corps hens\'eliens}, Algebraic Geometry 
{\bf 5} (2014), 573-612.






\bibitem{Gi02} P.~Gille, {\it Unipotent subgroups of reductive groups of  characteristic  $p>0$}, Duke Math. J. {\bf 114} (2002), 307-328.

\bibitem{Gi15} P.~Gille, {\it Sur la classification des 
sch\'emas en groupes semi-simples}, ``Autour des sch\'emas en groupes, III'', 
Panoramas et Synth\`eses {\bf 47} (2015), 39-110.


\bibitem{Gi18} P.~Gille, {\it Semisimple groups that are quasi-split over a tamely ramified extension}
Rendiconti del Seminario Matematico
della Universit\`a di Padova {\bf 140} (2018),  173-183.


\bibitem{Gi24} P.~Gille, {\it Loop group schemes and Abhyankar's lemma},   Comptes Rendus de l'Acad\'emie des Sciences, Math\'ematique {\bf 362} (2024),  159-166.

\bibitem{Gi26} P.~Gille, {\it Lectures on loop torsors in Arithmetic Geometry}, 
\href{https://math.univ-lyon1.fr/~gille/prenotes/gille_lodha.pdf}{link}.

\bibitem{GN} P.~Gille, E.~Neher, {\it
Group schemes over LG-rings and applications to cancellation theorems and Azumaya algebras}, to appear in  European Journal of Mathematics.

\bibitem{GiP23} P.~Gille, R.~Parimala, 
{\it A local-global principle for twisted flag varieties}, 
to appear  in Inventiones Mathematicae.


\bibitem{GiP24} P.~Gille, R.~Parimala, 
{\it Local-global principles for torsors over semiglobal fields}, 
arXiv:2406.17355.

\bibitem{GP07} P.~Gille, A.~Pianzola, {\it Galois cohomology and forms of algebras over Laurent polynomial rings},  Mathematische Annalen  {\bf 338} (2007), 497-543.


\bibitem{GP13} P.~Gille, A.~Pianzola, {\it  Torsors, reductive group schemes and extended affine Lie algebras},
Memoirs of  AMS {\bf  1063} (2013). 

\bibitem{GS}
P.~Gille, T.~Szamuely, {\it Central simple algebras and Galois cohomology}, 2-i\`eme edition,
Cambridge Studies in Advanced Mathematics, Cambridge University Press, 2017, p. xi+417.


\bibitem{Ge} S.~Gille,
{\it On Witt groups with support}, Mathematische Annalen
{\bf 332} (2002), 103-137.  

 \bibitem{Gd} J.~Giraud, {\it Cohomologie non-ab\'elienne}, Springer (1970).
 
  \bibitem{GW} U. G\"ortz, T. Wedhorn, {\it  Algebraic Geometry I: Schemes},
   Springer (2020). 
 
 \bibitem{GM} A.~Grothendieck, J.P.~Murre, 
 {\it The tame fundamental group of a 
 formal neighbourhood of a divisor with normal crossings on a scheme}, Lecture Notes in
Mathematics {\bf 208} (1971), Springer-Verlag, Berlin-New York.


\bibitem{Gr3} A.~Grothendieck, {\it Le groupe de Brauer, III: Exemples et compl\'ements}, in  Dix expos\'es sur la cohomologie des sch\'emas, J. Giraud et al., Adv. Stud. Pure Math., 3, Masson/`North-Holland.
 
 \bibitem{Guo} N.~Guo, {\it The Grothendieck-Serre's conjecture
 for semilocal Dedekind rings}, Transformation Groups {\bf  27}
 (2022), 897-917.
 
 
 

 \bibitem{HS} K.~H\"ubner, A.~Schmidt, {\it 
 The tame site of a scheme}, Inventiones mathematicae  {\bf 223}  (2021), 379-443. 



\bibitem{KP} T.~Kaletha, G.~Prasad, {\it Bruhat-Tits Theory,
A New Approach}, New Mathematical Monographs, Cambridge University Press, 2022.


\bibitem{Ka} T.~Kambayashi, {\it  On unramified coverings of the affine line in positive characteristics},
https://www.kurims.kyoto-u.ac.jp/~kyodo/kokyuroku/contents/pdf/0713-02.pdf




 \bibitem{Kb} M.~Karoubi, {\it Localisation de formes quadratiques. II}, 
Annales scientifiques de l'Ecole Normale Sup\'erieure {\bf 8}(1975), 99-155. 

\bibitem{Kz} N.~Katz, {\it 
Local-to-global extensions of representations of fundamental groups},
Annales de l'Institut Fourier {\bf 36} (1986), 69-106. 

\bibitem{KS} M.~Kerz,  A.~Schmidt, {\it   On different notions of tameness in arithmetic geometry}, Mathematische Annalen {\bf 346}  (2010), 641-668.


\bibitem[K]{Kn} M. A.~Knus, {\it Quadratic and hermitian forms over
 rings}, Grundlehren der mat. Wissenschaften {\bf 294} (1991), Springer.

\bibitem{Lam} T.Y.~Lam, {\it
 Introduction to Quadratic Forms over Fields}, 
 AMS  Graduate Studies in Mathematics
Volume {\bf  67} (2005), 550 pp.


\bibitem{L} M.~Lara, {\it
Fundamental exact sequence for the pro-\'etale fundamental group}, Algebra and Number Theory {\bf 18} (2024), 631-683.


 
\bibitem{Liu} Q.~Liu, {\it  Algebraic geometry and arithmetic curves},  Oxford Graduate Texts in Mathematics {\bf  6} (2002),  Oxford University Press, Oxford. 
    
 

\bibitem{Mg} B.~Margaux, {\it
 Vanishing of hochschild cohomology for affine group
schemes and rigidity of homomorphisms between algebraic groups}, Documenta Math. {\bf 14} (2009), 653-672.

\bibitem{Ma} H.~Matsumura, {\it
Commutative Ring Theory}, Cambridge University Press.



\bibitem{M} J.S.~Milne, {\it Etale cohomology}, Princeton University Press (1981).


\bibitem{Mr} J.~Milnor, {\it On the existence of a connection with curvature zero}, Commentarii Mathematici


\bibitem{MP} J.~Muskat, R.~Priess, {\it  Alternating group covers of the affine line}, Israel Journal of Math.
{\bf 187} (2012),  117-139.



\bibitem{N1}  Y.A.~Nisnevich, {\it Espaces homog\`enes principaux rationnellement triviaux et arithm\'etique des sch\'emas en groupes},
r\'eductifs sur les anneaux de Dedekind, C. R. Acad. Sci. Paris S\'er. I Math.  {\bf 299} (1984), 5-8.

\bibitem{N2}  Y.A.~Nisnevich, {\it Rationally trivial principal homogeneous spaces, purity and arithmetic of reductive
group schemes over extensions of two-dimensional regular local rings}, C. R. Acad. Sci. Paris S\'er. I
Math. {\bf 309} (1989),  651-655.


\bibitem{Oe} J.~Oesterl\'e, {\it Sur les sch\'emas en groupes de type multiplicatif}, in {\it Autour des sch\'emas en groupes, vol. I}, Panoramas et Synth\`eses \textbf{42-43}, Soc. Math. France  2014.

\bibitem{OP} M.~Ojanguren, I.~Panin, {\it 
The Witt group of Laurent polynomials},
Enseign. Math. {\bf  46} (2000),  361-383.


\bibitem{O} F.~Orgogozo, {\it Alt\'erations et groupe fondamental premier \`a $p$}, Bulletin de la Soci\'et\'e Math\'ematique de France {\bf  131} (2003), 123-147.  

 
 
 
 \bibitem{P20} G.~Prasad,  {\it
Finite group actions on reductive groups and buildings and tamely-ramified descent in Bruhat-Tits theory},
Amer. J. Math. {\bf 142} (2020),  1239-1267.

\bibitem{Rn} A.A.~Ranicki, {\it  Algebraic L-theory},
 Comm. Math. Helv. {\bf 49} (1974), 137-167.


\bibitem{Rao} R.A.~Rao, {\it On projective $R_{f_1\dots f_t}$-modules}, Amer. J. Math. {\bf 107} (1985),  387-406.


\bibitem{RY} Z.~Reichstein, B.~Youssin, {\it Essential dimensions of algebraic groups and a 
resolution theorem for $G$-varieties (with an appendix by J\'anos Koll\'ar and Endre Szab\'o)}, Canad. J. Math. {\bf 52} (2000) 1018-1056.


\bibitem{Rou}  G.~Rousseau, {\it Immeubles des groupes
 r\'eductifs sur les corps locaux}, Universit\'e d'Orsay,
 thesis (1977), 
 \href{http://www.iecl.univ-lorraine.fr/~Guy.Rousseau/}{Link}.
 
 
 \bibitem{RS}  K.~R\"ulling, S.~Schr\"oer, {\it  Loops on schemes and the algebraic fundamental group}, Rev. Mat. Complut. (2024).

\bibitem{Sch} W.~Scharlau, {\it Quadratic and Hermitian Forms}, Grundlehren der mathematischen Wissenschaften (GL, volume 270, 1985), Springer.


\bibitem{Sc} P.~Scholze,   https://mathoverflow.net/questions/375442, 2021.
\bibitem{SGA1} {\it S\'eminaire de G\'eom\'etrie alg\'ebrique de l'I.H.E.S.,  Rev\^etements \'etales et groupe fondamental, dirig\'e
par  A. Grothendieck},  Documents math\'ematiques vol. 3 (2003), Soci\'et\'e math\'ematique de France.

\bibitem{SGA3} {\it S\'eminaire de G\'eom\'etrie alg\'ebrique de l'I.\ H.\ E.\ S., 1963-1964,
Sch\'emas en groupes, dirig\'e par M.\ Demazure et A.\ Grothendieck},  
Lecture Notes in Math. 151-153. Springer (1970).



\bibitem{Sa} A.~Sarlin, {\it The \'etale fundamental group, \'etale
homotopy and anabelian geometry}, 
\href{http://www.diva-portal.org/smash/get/diva2:1165816/FULLTEXT01.pdf}{Degree project in Mathematics}, Stockholm (2017).


\bibitem{Se1} J-P.\,Serre, {\it Cohomologie galoisienne}, cinqui\`eme  \'edition, 
Springer-Verlag, New York, 1997.

\bibitem{Se2} J-P.\,Serre, {\it Corps locaux}, troisi\`eme \'edition, Hermann, Paris.

\bibitem{Se3} J-P.\,Serre, {\it  Groupes
finis d'automorphismes
d'anneaux locaux r\'eguliers}, Colloque d'Alg\`ebre E.N.S. (1967).


  

\bibitem{St} Stacks project, https://stacks.math.columbia.edu

\bibitem{SH} I.~Swanson, C.~Huneke, {\it Integral Closure of Ideals, Rings, and Modules}, London Mathematical Society Lecture Note Series {\bf 336}, Cambridge University Press, Cambridge, 2006.

\bibitem{Sz} T.~Szamuely, {\it 
Galois Groups and Fundamental Groups},
    Cambridge Studies in Advanced Mathematics {\bf 117} (2009),  
Cambridge University Press, 



\bibitem{V} A.~Vistoli, {\it  Notes on Grothendieck topologies, fibered categories and descent theory}, Mathematical Surveys and Monographs {\bf 123} (2005), American Mathematical Society, 1-104.

 \bibitem{VW} R.~Vakil, K.~Wickelgren, {\it Universal covering spaces and fundamental groups in algebraic geometry as schemes}, Journal de th\'eorie des nombres de Bordeaux
 {\bf  23} (2011), 489-526. 
 
 \bibitem{Wa} J.~Watanabe, {\it 
Some remarks on Cohen-Macaulay rings with many zero divisors and an application}, Journal of Algebra
{\bf  39} (1976),  1-14.

 \bibitem{We} C.A.~Weibel,  {\it  The K-book: An Introduction to Algebraic K-theory}, AMS Graduate Studies in Mathematics
Volume {\bf 145} (2013), 618 pp. 
  
\end{thebibliography}
\end{document}